\documentclass[reqno,twoside]{article}
\usepackage[utf8]{inputenc}
\usepackage[english]{babel}
\usepackage[T1]{fontenc}
\usepackage{microtype}
\usepackage{enumitem}

\usepackage{mathrsfs}
\usepackage{geometry} 
\usepackage{enumitem} 
	\setlist[enumerate]{label=(\roman*)}  
	\setlist[enumerate,2]{label=(\alph*)} 
\setlist{nolistsep} 

\usepackage{booktabs} 
\usepackage{array} 
\usepackage[labelfont=bf]{caption}

\title{Categories of abstract and noncommutative measurable spaces}
\usepackage{fancyhdr} 
\makeatletter
\let\runtitle\@title
\makeatother
\usepackage{titletoc}
\titlecontents{section}[0em]{}{\thecontentslabel{.\hspace{6pt}}}{\hspace*{0em}}{\titlerule*[0.5pc]{.}\contentspage}

\usepackage[outermarks]{titlesec}

\titleformat{\section}
{\bfseries\Large}
{\thesection.}
{1ex}
{}
[]
\titleformat{\subsection}
{\bfseries\large}
{\thesubsection.}
{1ex}
{}
[]

\usepackage{tikz}
\usetikzlibrary{shapes,cd,decorations.pathmorphing}
\usepackage{amssymb}
\usepackage{amsmath}

\usepackage{calc}
\usepackage{mathtools}
\DeclarePairedDelimiter\abs{\lvert}{\rvert}
\DeclarePairedDelimiter\norm{\lVert}{\rVert}
\makeatletter
\let\oldabs\abs
\def\abs{\@ifstar{\oldabs}{\oldabs*}}
\let\oldnorm\norm
\def\norm{\@ifstar{\oldnorm}{\oldnorm*}}
\makeatother

\usepackage{amsthm}
\definecolor{myurlcolor}{rgb}{0,0,0.5}
\definecolor{mycitecolor}{rgb}{0,0.5,0}
\definecolor{myrefcolor}{rgb}{0.5,0,0}
\usepackage[pagebackref,draft=false]{hyperref}
\hypersetup{colorlinks,
	linkcolor=myrefcolor,
	citecolor=mycitecolor,
	urlcolor=myurlcolor}

\usepackage[capitalize,noabbrev]{cleveref}

\numberwithin{equation}{subsection}

\theoremstyle{plain} 
\newtheorem{thm}[equation]{Theorem}\Crefname{thm}{Theorem}{Theorems}
\newtheorem{cor}[equation]{Corollary}\Crefname{cor}{Corollary}{Corollaries}
\newtheorem{lem}[equation]{Lemma}\Crefname{lem}{Lemma}{Lemmas}
\newtheorem{prop}[equation]{Proposition}\Crefname{prop}{Proposition}{Propositions}
\newtheorem{conj}[equation]{Conjecture}\Crefname{conj}{Conjecture}{Conjectures}
\Crefname{propdef}{Proposition/Definition}{Propositions/Definitions}
\newtheorem{fact}[equation]{Fact}
\theoremstyle{definition} 
\newtheorem{defn}[equation]{Definition}\Crefname{defn}{Definition}{Definitions}

\newtheorem{ex}[equation]{Example}
\newtheorem{rem}[equation]{Remark}\Crefname{rem}{Remark}{Remarks}

\newtheorem{nota}[equation]{Notation}
\newtheorem{conv}[equation]{Convention}

\makeatletter
\def\cref@thmoptarg[#1]#2#3#4{%
	    \ifhmode\unskip\unskip\par\fi%
	    \normalfont%
	    \trivlist%
	    \let\thmheadnl\relax%
	    \let\thm@swap\@gobble%
	    \thm@notefont{\fontseries\mddefault\upshape}%
	    \thm@headpunct{.}
	    \thm@headsep 5\p@ plus\p@ minus\p@\relax%
	    \thm@space@setup%
	    #2
	    \@topsep \thm@preskip               
	    \@topsepadd \thm@postskip           
	    \def\@tempa{#3}\ifx\@empty\@tempa%
	      \def\@tempa{\@oparg{\@begintheorem{#4}{}}[]}%
	    \else%
	      \refstepcounter[#1]{#3}
	      \@namedef{cref@#3@alias}{#1}
	      \def\@tempa{\@oparg{\@begintheorem{#4}{\csname the#3\endcsname}}[]}%
	    \fi%
	    \@tempa}%
\makeatother

\usepackage[nottoc]{tocbibind}
\usepackage{bbm}
\usepackage{xcolor}
\usepackage{array}
\usepackage{multirow,hhline}

\usepackage{todonotes}

\newcommand{\id}{\operatorname{id}}
\newcommand{\op}{\mathrm{op}}
\newcommand{\cstar}{$C^{\ast}$-algebra}
\newcommand{\wstar}{$W^{\ast}$-algebra}
\newcommand{\cstars}{$\sigma C^{\ast}$-algebra}
\newcommand{\Cstars}{$\sigma C^{\ast}$-algebra}
\newcommand{\state}{\mathsf{States}}				
\newcommand{\starshom}{${\sigma}$\text{-}ho\-mo\-mor\-phism}	
\newcommand{\starsiso}{${\sigma}$\text{-}iso\-mor\-phism}	
\newcommand{\bsigma}{Boo\-lean $\sigma$-al\-ge\-bra}
\newcommand{\bsigmahom}{$\sigma$-ho\-mo\-mor\-phism}
\newcommand{\bsigmaiso}{$\sigma$-iso\-mor\-phism}
\newcommand{\cpu}{cpu}
\newcommand{\scpu}{$\sigma$-\cpu}
\newcommand{\scpus}{\sigma\!\operatorname{cpu}}
\newcommand{\cpus}{\operatorname{cpu}}
\newcommand{\shom}{\sigma}		
\newcommand{\sa}{\mathrm{sa}}
\newcommand{\newterm}[1]{\emph{\textbf{#1}}}
\newcommand{\spectrum}{\operatorname{spec}} 
\newcommand{\baire}{\mathsf{Baire}}
\newcommand{\spec}{\mathsf{Gelfand}} 
\newcommand{\specs}{\mathsf{Gelfand}_{\sigma}} 
\newcommand{\Linf}{\mathscr{L}^{\infty}}
\newcommand{\Linfbool}{\mathscr{L}_{\operatorname{abs}}^{\infty}}
\newcommand{\stone}{\mathsf{Stone}}
\newcommand{\stones}{\mathsf{Stone}_{\sigma}}
\newcommand{\clopen}{\mathsf{Clopen}}
\newcommand{\proj}{\mathsf{Proj}}
\newcommand{\A}{\mathcal{A}}
\newcommand{\B}{\mathcal{B}}
\newcommand{\N}{\mathcal{N}}	
\newcommand{\C}{\mathcal{C}}
\newcommand{\D}{\mathcal{D}}
\newcommand{\I}{\mathcal{I}}
\newcommand{\M}{\mathcal{M}}
\renewcommand{\Re}{\operatorname{Re}}	
\renewcommand{\Im}{\operatorname{Im}}	
\newcommand{\hilb}{\mathcal{H}}
\newcommand{\bhilb}[1][\mathcal{H}]{\mathcal{B}(#1)}
\newcommand{\khilb}[1][\mathcal{H}]{\mathcal{K}(#1)}
\newcommand{\ball}[2]{B_{#2}(#1)} 
\newcommand{\cball}[2]{\overline{B_{#2}(#1)}}
\newcommand{\ev}{\operatorname{ev}}
\newcommand{\ordlim}{\operatornamewithlimits{Lim}}
\newcommand{\univ}[1]{{#1}^{\sigma}}
\newcommand{\reg}[1]{{#1}^{r}}
\newcommand{\unitensor}{\otimes^{\sigma}}
\newcommand{\regtensor}{\otimes^r}
\newcommand{\eps}{\varepsilon}
\renewcommand{\emptyset}{\varnothing}

\newcommand{\meas}{\mathsf{Meas}}
\newcommand{\stoch}{\mathsf{Stoch}}
\newcommand{\sobmeas}{\mathsf{SobMeas}}
\newcommand{\sobstoch}{\mathsf{SobStoch}}
\newcommand{\borelmeas}{\mathsf{StBorelMeas}}
\newcommand{\borelstoch}{\mathsf{StBorelStoch}}
\newcommand{\bairemeas}{\mathsf{BaireMeas}}
\newcommand{\bairestoch}{\mathsf{BaireStoch}}
\newcommand{\cHaus}{\mathsf{cHaus}}
\newcommand{\cHausstoch}{\mathsf{cHausStoch}}
\newcommand{\Stone}{\mathsf{StoneSp}}

\newcommand{\Bool}{\mathsf{Bool}}
\newcommand{\Boolshom}{\Bool_{\shom}}
\newcommand{\SigmaB}{\sigma\mathsf{Bool}_{\shom}}
\newcommand{\csigmaB}{\mathsf{Conc}\sigma\mathsf{Bool}_{\shom}}

\newcommand{\Calg}{\mathsf{C}^{\ast}}
\newcommand{\cCalg}{\mathsf{c}\mathsf{C}^{\ast}}
\newcommand{\Calgcpu}{\Calg_{\cpus}}
\newcommand{\cCalgcpu}{\cCalg_{\cpus}}
\newcommand{\Calgshom}{\Calg_{\shom}}	
\newcommand{\cCalgshom}{\cCalg_{\shom}}
\newcommand{\Calgscpu}{\Calg_{\scpus}}
\newcommand{\cCalgscpu}{\cCalg_{\scpus}}

\newcommand{\SigmaC}{\shom\mathsf{C}^{\ast}_{\shom}}
\newcommand{\cSigmaC}{\mathsf{c}\shom\mathsf{C}^{\ast}_\shom}
\newcommand{\SigmaCcpu}{\sigma\mathsf{C}^{\ast}_{\scpus}}
\newcommand{\cSigmaCcpu}{\mathsf{c}\shom\mathsf{C}^{\ast}_{\scpus}}

\newcommand{\msigmaC}{\mathsf{PureRep}\sigma\mathsf{C}^{\ast}_{\shom}}
\newcommand{\cmsigmaC}{\mathsf{cPureRep}\sigma\mathsf{C}^{\ast}_{\shom}}
\newcommand{\msigmaCcpu}{\mathsf{PureRep}\sigma\mathsf{C}^{\ast}_{\scpus}}
\newcommand{\cmsigmaCcpu}{\mathsf{cPureRep}\sigma\mathsf{C}^{\ast}_{\scpus}}

\newcommand{\PBenv}{\mathsf{PedBaireC}^{\ast}_{\shom}}
\newcommand{\cPBenv}{\mathsf{cPedBaireC}^{\ast}_{\shom}}
\newcommand{\pbenvcpu}{\mathsf{PedBaireC}^{\ast}_{\scpus}}
\newcommand{\cpbenvcpu}{\mathsf{cPedBaireC}^{\ast}_{\scpus}}
\newcommand{\pbsep}{\mathsf{PedBaireSepC}^{\ast}_{\shom}}	
\newcommand{\cpbsep}{\mathsf{cPedBaireSepC}^{\ast}_{\shom}}
\newcommand{\pbsepcpu}{\mathsf{PedBaireSepC}^{\ast}_{\scpus}}
\newcommand{\cpbsepcpu}{\mathsf{cPedBaireSepC}^{\ast}_{\scpus}}

\newcommand{\repsigmaC}{\mathsf{Rep}\sigma\mathsf{C}^{\ast}_{\shom}}

\begin{document}
\titlepage
\begin{center}
	{\LARGE \textbf{Categories of abstract and\\[2pt] noncommutative measurable spaces}}

\par\vspace{3ex}
\textit{Tobias Fritz\textnormal{\textsuperscript{1}} \emph{and} Antonio Lorenzin\textnormal{\textsuperscript{2}}}

\par \vspace{1ex}
{\footnotesize \textsuperscript{1} Universit\"{a}t Innsbruck, Institut f\"{u}r Mathematik, Innsbruck, Austria
\par\vspace{-1ex}
\textsuperscript{2} University College London, Computer Science Department, London, UK
}

\end{center}

\begin{abstract}
\noindent 
Gelfand duality is a fundamental result that justifies thinking of general unital $C^*$-algebras as {noncommutative} versions of compact Hausdorff spaces. 
Inspired by this perspective, we investigate what noncommutative measurable spaces should be. 
This leads us to consider categories of monotone $\sigma$-complete $C^*$-algebras as well as categories of \bsigma{}s, which can be thought of as abstract measurable spaces.

Motivated by the search for a good notion of noncommutative measurable space, we provide a unified overview of these categories, alongside those of measurable spaces, and formalize their relationships through functors, adjunctions and equivalences.
This includes an equivalence between \bsigma{}s and commutative monotone $\sigma$-complete $C^*$-algebras, as well as a Gelfand-type duality adjunction between the latter category and the category of measurable spaces.
This duality restricts to two equivalences: one involving standard Borel spaces, which are widely used in probability theory, and another involving the more general Baire measurable spaces.
Moreover, this result admits a probabilistic version, where the morphisms are $\sigma$-normal cpu maps and Markov kernels, respectively.

We hope that these developments can also contribute to the ongoing search for a well-behaved Markov category for measure-theoretic probability beyond the standard Borel setting —-- an open problem in the current state of the art.
\end{abstract}
\tableofcontents

\section{Introduction}

\subsubsection*{Motivation}

About a century ago, John von Neumann introduced the formalism of Hilbert spaces as a mathematical foundation for quantum mechanics. 
Today, this framework has become standard, indicating that the expressiveness of the abstract theory is sufficient to capture the essential features of quantum phenomena.
From an algebraic perspective, structures like \cstar{}s play an important role in this story, serving as algebras of observables.
Due to the inherent probabilistic nature of quantum theory, one often speaks of \emph{quantum probability} in this context.
Optimally, one may expect classical probability to be \emph{precisely} the special case of quantum probability where the observable algebras under consideration are commutative.

A first step in this direction is provided by the well-known \emph{Gelfand duality}, which establishes an equivalence of categories between compact Hausdorff spaces and commutative \cstar{}s.\footnote{In the present paper, \cstar{}s and homomorphisms between them are always unital, and we usually leave this implicit.}
Conceptually, the commutativity condition enforces classical behavior: the Heisenberg uncertainty principle no longer applies, and all observables can meaningfully be assigned joint values.
Moreover, the celebrated Riesz--Markov--Kakutani representation theorem shows that states on a commutative \cstar{} $C(X)$ correspond bijectively to Radon probability measures on the compact Hausdorff space $X$. 
This correspondence has been extended to \emph{probabilistic Gelfand duality}~\cite{furber_jacobs_gelfand}, which relates continuous Markov kernels between compact Hausdorff spaces to completely positive unital maps (\emph{cpu maps}) between commutative \cstar{}s.

However, the topological nature of these results is unsatisfactory: our criterion that the commutative case of quantum probability should precisely be classical measure-theoretic probability does not hold, since in the latter context, Markov kernels are only required to be measurable, not continuous.
Indeed those Markov kernels that arise in mathematical practice (e.g.~as regular conditional probabilities) often fail to be continuous. 
In addition, there arguably is something wrong already at the level of objects, as algebras of measurable functions seem more natural in the probability context than algebras of continuous functions.

These reasons make us think that the category of \cstar{}s is not the best setting for quantum probability theory.
Further evidence comes from quantum measurement theory: on physical and operational grounds, one would expect that a morphism of observable algebras maps a positive operator-valued measure (POVM) on one algebra to a POVM on the other algebra.
But if every $*$-homomorphism, or even every cpu map, is a morphism, then this is not guaranteed: the image of a POVM may fail the $\sigma$-additivity required of a POVM.

It is often suggested that a good setting for quantum probability theory is that of \wstar{}s (von Neumann algebras) instead.
This is indeed well-motivated by the idea that \wstar{}s can be thought of as algebras of noncommutative random variables modulo almost everywhere equality, and there is a Gelfand-type duality that makes this precise~\cite{pavlov2022duality}.
But from our point of view, this is not satisfactory either, because quotienting by almost everywhere equality means that one needs to fix a reference measure.
In the non-atomic case, one thereby loses access to one of the most basic concepts of probability theory, namely delta measures.

A closely related issue manifests itself very cleanly and precisely in the recent theory of Markov categories~\cite{chojacobs2019strings,fritz2019synthetic}.
This framework allows us to study probability theory abstracted away from measure-theoretic details.
In these categories, each object comes equipped with a \emph{copy morphism}, which in typical Markov categories is given by the diagonal map $X \to X \times X$.
Among other things, this Markov category structure enables categorical definitions of determinism, conditioning, conditional independences, and almost sure equalities (\cite[\S~10--13]{fritz2019synthetic}). 
In other words, we are able to describe these notions inside any Markov category, without additional structure.
This fresh perspective enables a more abstract approach to the study of probability theory, which has consequently stimulated considerable interest in particular from the computer science community, and facilitated general categorical proofs of classical theorems such as the de Finetti~\cite{fritz2021definetti} and Aldous--Hoover theorems~\cite{chen2024aldoushoover}.

The issue is now that categories of measure spaces tend to \emph{not} be Markov categories: since maps between measure spaces are usually required to be measure-preserving, the diagonal map cannot be a morphism except in trivial cases.\footnote{For example for the Lebesgue measure on $[0,1]$, the pushforward along the diagonal gives the Lebesgue measure on the diagonal of the unit square $[0,1] \times [0,1]$, while the monoidal structure of the category equips $[0,1] \times [0,1]$ with the product measure, which is the usual Lebesgue measure on the unit square and therefore different.}
Taking the opposite of a category of $W^*$-algebras does not yield a Markov category either for similar reasons.\footnote{There is an interesting way around this investigated by Furber~\cite{furber2024monad}, which merits further consideration. By using a suitably defined tensor product of $W^*$-algebras, he does obtain a Markov category. However, the $W^*$-algebras that appear in this context tend to be very large, such as the double dual $C([0,1])^{**}$.}
This indicates that working with measure spaces or $W^*$-algebras obscures important concepts such as the ones listed above (determinism, conditioning, conditional independence, and almost sure equalities).

The highlighted limitations of both the continuous and the measure world prompts us to ask what a quantum, or noncommutative, measurable space should be. 
As far as we are aware, this question has not been seriously addressed in the literature before.
We believe that working on this question can lead to a deeper conceptual understanding of quantum probability theory, for example by shedding light on how notions such as conditional quantum probability can be formulated in simple categorical terms.

Another important motivation for this work comes purely from the classical side.
The usual Markov category for measure-theoretic probability uses standard Borel measurable spaces as objects.
This is limiting insofar as this class of measurable spaces is not closed under uncountable products.
It remains an open problem to find a well-behaved Markov category for measure-theoretic probability in which products of arbitrary families of objects exist.\footnote{What we mean by ``well-behaved'' (at a minimum) is that all conditionals exist, and the ``products'' should be Kolmogorov products in the sense of~\cite{fritzrischel2019zeroone}. With standard Borel spaces, we do have conditionals but only countable Kolmogorov products.}

Thus the goal of the present work is to investigate categories of measurable-like spaces for both quantum and classical probability theory.
To summarize, the motivating problems are the following:
\begin{enumerate}[itemsep=1ex,topsep=1ex]
	\item Find a definition of ``quantum measurable space'' which specializes to a well-behaved classical notion of measurable space or measurable-like space via a Gelfand-type duality.
	\item Find a well-behaved Markov category for measure-theoretic probability which has uncountable products.
\end{enumerate}
We hope that the present paper can lay some groundwork for resolving both of these problems by providing a taxonomy of approaches, as well as relations between them in the form of dualities.
Moreover, as an auxiliary notion intermediate between point-based concepts of measurable space and $C^*$-algebraic ones, we also consider \emph{\bsigma{}s} as an abstract generalization of measurable spaces.
Since these have been extensively studied, examining this setting first can provide useful hints on how to generalize to the noncommutative case.

\subsubsection*{Related work}

In the classical case, recent work by Jamneshan and Tao~\cite{jamneshan23foundational} has a similar goal as our second one.
As argued by their study, a well-behaved generalization of standard Borel spaces is given by \emph{Baire measurable spaces}. This is explicitly stated on p.~7, where they write that the Baire $\sigma$-algebra is the most natural choice of $\sigma$-algebra when working with compact Hausdorff spaces.
For instance, conditional probabilities still exist~\cite[Corollary 452N]{fremlin4}. 
This result is similar to the disintegration theorem \cite[Theorem 1.6]{jamneshan23foundational}, which follows from the strong Lusin property (Proposition 7.4) of a certain class of Stone spaces: every bounded Baire measurable function is equal almost everywhere to a unique continuous function. 

While the present paper may initially appear similar to the treatment by Jamneshan and Tao, there are several key distinctions.
Notably, although both works consider Baire measurable, Stone and Rickart spaces (called $\text{Stone}_{\sigma}$-spaces in \cite{jamneshan23foundational}), our approach is mostly point-free, as we have generalizations to the noncommutative setting in mind.
In addition, Jamneshan and Tao focus on measure spaces and deterministic morphisms, which we find unsatisfactory in light of the discussion on Markov categories above.
Instead, our objective is to describe a setting where objects are measurable-like spaces and morphisms are either deterministic or correspond to Markov kernels (such as regular conditional probabilities).

We comment on the relation to~\cite{jamneshan23foundational} in detail wherever any overlap appears in the main text.

\subsubsection*{Overview}

In \cref{tab:meas}, we display all the categories of measurable spaces that we consider.
These are subsequently smaller full subcategories of $\meas$, corresponding to spaces that are progressively more well-behaved.
We also discuss several categories of Boolean algebras as abstract counterparts of measurable spaces, which are summarized in \cref{tab:bool}, likewise in progressively more well-behaved subcategories.
Finally, \cref{tab:cstar} similarly lists the categories of \cstar{}s that we investigate.
For each category of measurable spaces and of \cstar{}s, we also consider a probabilistic version, where the morphisms are Markov kernels and cpu maps, respectively.
We do not do this for Boolean algebras, since there does not seem to be any natural concept of probabilistic morphism between Boolean algebras which does not proceed through measurable spaces or \cstar{}s (and would therefore result in equivalences of categories that trivially hold by definition).

With deterministic morphisms, the relations between all of these categories, as far as we consider them in this paper, are summarized in \cref{fig:det_side}.
For probabilistic morphisms instead, they are detailed in \cref{fig:prob_side}.
Before summarizing all of this in some more detail, let us note that we do not strive to be exhaustive in our treatment: for example, we do not consider the left adjoint to $\sobmeas \to \meas$ (sobrification), nor do we discuss Boolean algebra versions of the categories $\cPBenv$ and $\pbsep$, although all of this should be possible.

\begin{table}
	\centering
	\begin{tabular}{c|c|c|c}
		\multicolumn{1}{c}{\textit{Category}} & \multicolumn{1}{c}{\textit{Objects}}&\multicolumn{1}{c}{\textit{Morphisms}}&\multicolumn{1}{c}{\textit{Definition}}\\
		\toprule 
		$\meas$ & Measurable spaces& \multirow{4}{*}{Measurable functions} & \ref{def:meas}\\
		$\sobmeas$ & Sober measurable spaces && \ref{def:sobmeas}\\
		$\bairemeas$ & Baire measurable spaces&&\ref{def:bairemeas}\\
		$\borelmeas$ & Standard Borel spaces &&\ref{def:borelmeas}\\
	\end{tabular}
	\caption{Categories of measurable spaces. Following the conventions of categorical probability~\cite{fritz2019synthetic}, we replace ``$\meas$'' by ``$\stoch$'' for any of these whenever we consider Markov kernels as morphisms instead of measurable functions.}
	\label{tab:meas}
\end{table}

\setlength{\tabcolsep}{2ex}
\begin{table}
	\centering
	\begin{tabular}{c|c|c|c}
		\multicolumn{1}{c}{\textit{Category}} & \multicolumn{1}{c}{\textit{Objects}}&\multicolumn{1}{c}{\textit{Morphisms}}&\multicolumn{1}{c}{\textit{Definition}}\\
		\toprule  
		$\Bool$ & Boolean algebras& Homomorphisms & \ref{def:Bool}\\
		\hline
		$\Boolshom$ & Boolean algebras & \multirow{3}{*}{\bsigmahom{}s}&\ref{def:boolshom}\\
		$\SigmaB$ & \bsigma{}s && \ref{def:sigmaB}\\
		$\csigmaB$ & Concrete \bsigma{}s && \ref{def:meas}
	\end{tabular}
	\caption{Categories of Boolean algebras.}\label{tab:bool}
\end{table}

\begin{table}
	\centering
	\begin{tabular}{c|>{\centering\arraybackslash}m{0.4\textwidth}|c|c}
		\multicolumn{1}{c}{\textit{Category}} & \multicolumn{1}{c}{\textit{Objects}}&\multicolumn{1}{c}{\textit{Morphisms}}&\multicolumn{1}{c}{\textit{Definition}}\\
		\toprule 
		$\Calg$ & \cstar{}s& $\ast$-homomorphisms & \ref{def:calg}\\
		\hline 
		$\Calgshom$ & \cstar{}s & \multirow{6}{*}{\starshom{}s} & \ref{def:calg}\\
		$\SigmaC$ & \cstars{}s && \ref{def:calg}\\
		$\msigmaC$ & Purely $\sigma$-representable \cstars{}s &&\ref{def:msigmaC}\\
		$\PBenv$ & Pedersen--Baire envelopes of \cstar{}s && \ref{def:pbenv} \\
		$\pbsep$ & Pedersen--Baire envelopes of separable \cstar{}s && \ref{def:pbenv} \\
	\end{tabular}
	\caption{Categories of \cstar{}s. The subscript ``$\sigma$'' reminds us that the morphisms in those categories that carry it are \starshom{}s. In each case, we may also add an additional subscript ``$\mathrm{cpu}$'' to indicate that we are considering completely positive maps instead of \starshom{}s. We can also add an additional prefix ``$\mathsf{c}$'' whenever we require commutativity. For example, $\cCalgscpu$ is the category of commutative \cstar{}s with $\sigma$-normal completely positive maps.}
	\label{tab:cstar}
\end{table}

Let us start the discussion with the categories of measurable spaces, which are in the left column of \cref{fig:det_side}, and the relation to the categories of Boolean algebras in the second column.
For every measurable space $X$, its $\sigma$-algebra $\Sigma_X$ is a Boolean algebra.
Due to the closure under countable unions (or equivalently countable intersections), it is even a \newterm{\bsigma{}}, i.e.~a Boolean algebra with countable suprema (or equivalently infima).
Every measurable map $X \to Y$ induces a Boolean algebra homomorphism $\Sigma_Y \to \Sigma_X$ which moreover preserves countable suprema, which makes it into a \newterm{\bsigmahom{}}.
This explains the functor $\Sigma$ in \cref{fig:det_side}.
This functor is part of the most important relation between the first two columns, namely a contravariant idempotent adjunction between $\meas$ and $\SigmaB$ 
called \newterm{Loomis--Sikorski duality}~\cite{chen23univ}.\footnote{Calling this a duality despite the fact that it is not a dual equivalence aligns with the general categorical theory of duality: see, for example, Isbell duality, but also the general study of concrete dualities~\cite[Section~VI.4]{johnstone82stone}.}
In the opposite direction to $\Sigma$, its adjoint $\stones$ assigns to every \bsigma{} $A$ the set of \bsigmahom{}s $A \to \lbrace \bot,\top\rbrace$, equipped with the $\sigma$-algebra given by the sets of the form
\begin{equation*}
[a] \coloneqq \lbrace \phi \in \stones(A) \mid \phi(a)=\top \rbrace
\end{equation*}
for any $a \in A$.
As usual for concrete dualities~\cite{johnstone82stone}, the action on morphisms is given by composition.
Loomis--Sikorski duality is not a dual equivalence for two reasons.
First, not every \bsigma{} can be realized as the $\sigma$-algebra of a measurable space, and those that do are called \newterm{concrete}.
Second, the measurable spaces of the form $\stones(A)$ are special in that their $\sigma$-algebras separate points, and actually enjoy the stronger property that the $\{0,1\}$-valued probability measures are in bijection with the underlying set of points.
Such measurable spaces are called \newterm{sober}.
Like every idempotent adjunction, Loomis--Sikorski duality restricts to a dual equivalence between suitable full subcategories, and in this case these are the categories of sober measurable spaces and concrete \bsigma{}s, as indicated in \cref{fig:det_side}.

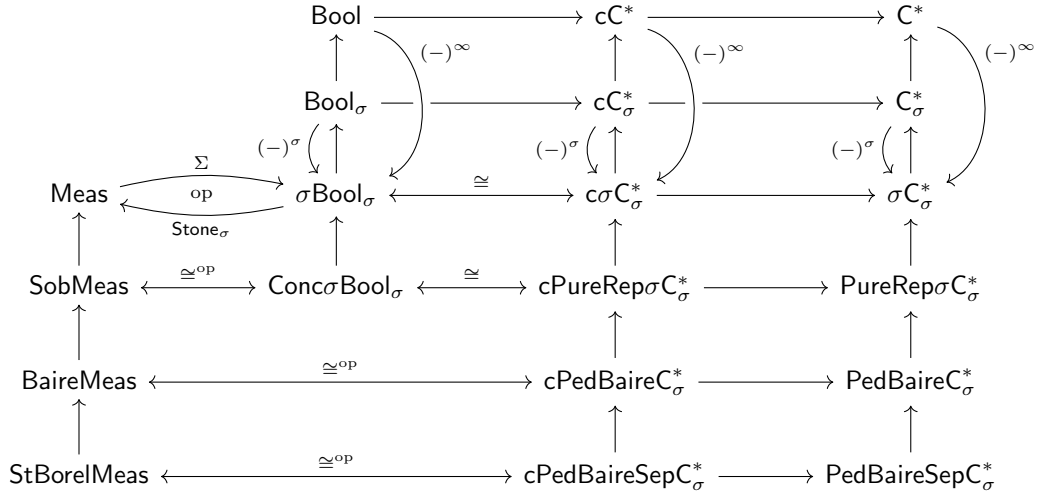
\begin{figure}
	\[
		\begin{tikzcd}[column sep=3pc]
		& \Bool	\ar[r] & \cCalg \ar[r]  & \Calg \ar[dd,bend left=70, "(-)^{\infty}", pos=0.33] \\
		& \Boolshom \ar[r]\ar[u]\ar[d,bend right=40, "\univ{(-)}" left] & \cCalgshom\ar[r]\ar[u]\ar[d,bend right=40, "\univ{(-)}" left] & \Calgshom\ar[u]\ar[d,bend right=40, "\univ{(-)}" left] \\
		\meas \ar[r,bend left=12,"\Sigma"]\ar[r,phantom,"\op"{font=\scriptsize}] & \SigmaB \ar[from=uu, bend left=70, "(-)^{\infty}", pos=0.33, crossing over] \ar[l,bend left=12,"\stones"]\ar[r,leftrightarrow, "\cong"]\ar[u] & \cSigmaC \ar[from=uu, bend left=70, "(-)^{\infty}", pos=0.33, crossing over] \ar[r]\ar[u] & \SigmaC\ar[u]\\
		\sobmeas\ar[r,leftrightarrow,"\cong^{\op}"] \ar[u]& \csigmaB\ar[r,leftrightarrow,"\cong"]\ar[u] & \cmsigmaC  \ar[r]\ar[u] & \msigmaC\ar[u]\\
		\bairemeas\ar[u]\ar[rr,leftrightarrow,"\cong^{\op}"]&& \cPBenv \ar[r]\ar[u] & \PBenv\ar[u] \\
		\borelmeas\ar[u]\ar[rr,leftrightarrow,"\cong^{\op}"]&& \cpbsep \ar[r]\ar[u] & \pbsep\ar[u]
		\end{tikzcd}
	\]
	\caption{The deterministic situation. All upward arrows are forgetful functors, bent arrows are left adjoints, and $\op$ indicates contravariance.}
	\label{fig:det_side}
\end{figure}

Moving up in the second column of \cref{fig:det_side}, there are forgetful functors from $\SigmaB$ to $\Boolshom$, the category of Boolean algebras with \bsigmahom{}s, and further to $\Bool$, the category of Boolean algebras with homomorphisms.
Both of these forgetful functors have left adjoints.
From $\Bool$ back to $\SigmaB$, this adjoint $(-)^{\infty}$ assigns to every Boolean algebra $A$ its \newterm{Baire envelope} $A^\infty$, which is the \bsigma{} of Baire sets on its Stone space (and in particular concrete).
The adjoint $\univ{(-)}$ from $\Boolshom$ back to $\SigmaB$ maps every $A$ to its \newterm{universal $\sigma$-completion} $\univ{A}$, which can be constructed as a quotient of the Baire envelope $A^\infty$.
The crucial difference between these two functors is that the universal homomorphism $A \to \univ{A}$ is a \bsigmahom{}, while this is typically not the case for $A \to A^\infty$.
There generally are other ways to $\sigma$-complete a Boolean algebra to a \bsigma{}, such as the \newterm{regular $\sigma$-completion}.
While the latter is of some interest in the study of tensor products, it does not seem to enjoy a simple universal property, and is not even functorial on homomorphisms (\cref{rem:nonfunctoriality_regular}).

We now turn to the third column in \cref{fig:det_side}, which involves \cstar{}s.
Under Stone and Gelfand duality, the functor $\Bool \to \cCalg$ corresponds to the inclusion functor of Stone spaces into compact Hausdorff spaces.
It restricts to various subcategories as indicated.
Particularly notable here is the equivalence between $\SigmaB$ and $\cSigmaC$, which is one of our main results (\cref{cor:proj_equiv}).
Its composition with Loomis--Sikorski duality is a contravariant idempotent adjunction between $\meas$ and $\cSigmaC$, which we call \newterm{measurable Gelfand duality} (\cref{thm:measurableGelfand}).
In one direction, it assigns to a measurable space $X$ the \cstars{} $\Linf(X)$ of bounded complex-valued measurable functions on $X$.
In the other direction, it takes a commutative \cstars{} $\A$ to its \newterm{Gelfand $\sigma$-spectrum}, by which we mean the measurable space $\specs(\A)$ consisting of the \starshom{}s $\A \to \mathbb{C}$.
While Loomis--Sikorski duality restricts to a dual equivalence between sober measurable spaces and concrete \bsigma{}s, measurable Gelfand duality extends this to a further equivalence with $\cmsigmaC$, the category of commutative purely $\sigma$-representable \cstar{}s with \starshom{}s.

Looking at the vertical arrows in the third column of \cref{fig:det_side}, we have forgetful functors and left adjoints analogous to those from the Boolean case.
These are the \newterm{Pedersen--Baire envelope} functor $(-)^{\infty}$ and the \newterm{universal $\sigma$-completion} functor $\univ{(-)}$.
Moving further down in the third column, we also consider full subcategories of $\cSigmaC$.
Namely $\cPBenv$ is given by those \cstars{}s that arise as Pedersen--Baire envelopes of commutative \cstar{}s, and we show that these correspond to Baire measurable spaces, i.e. measurable spaces whose $\sigma$-algebra comes from a compact Hausdorff topology (\cref{def:baire}). 
Since Baire envelopes of Boolean algebras correspond only to those Baire measurable spaces arising from Stone spaces, it is not clear whether the equivalence $\bairemeas \cong \cPBenv$ factors through Baire envelopes; we leave this open for future investigation.\footnote{By Kuratowski's theorem, uncountable standard Borel spaces are all measurably isomorphic to the Cantor space, so the conjectural factorization may follow from a generalization of Kuratowski's theorem.}
Further restricting, $\pbsep$ consists of the Pedersen--Baire envelopes of commutative separable \cstar{}s, and these enjoy a dual equivalence with standard Borel spaces.

Finally, the fourth column of \cref{fig:det_side} extends the third column to the noncommutative setting, but otherwise works in the same way.
In fact, our treatment of the functors in the third column directly takes place in the noncommutative setting, but restricts to the commutative case.

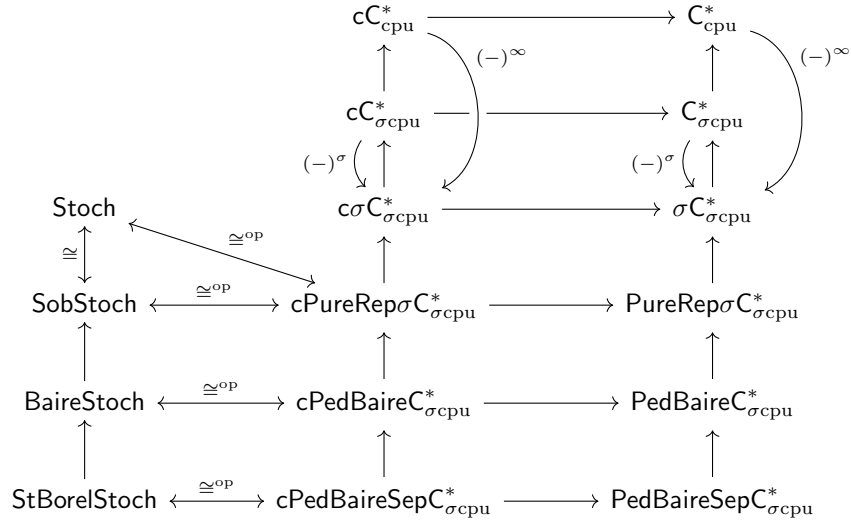
\begin{figure}
	\[
		\begin{tikzcd}[column sep=3pc]
		& \cCalgcpu \ar[r] & \Calgcpu \ar[dd,bend left=70, "(-)^{\infty}", pos=0.33] \\
		& \cCalgscpu \ar[r] \ar[u]\ar[d,bend right=40, "\univ{(-)}" left] & \Calgscpu\ar[u]\ar[d,bend right=40, "\univ{(-)}" left]\\
		\stoch \ar[rd, leftrightarrow, "\cong^{\op}"] & \cSigmaCcpu \ar[from=uu, bend left=70, "(-)^{\infty}", pos=0.33, crossing over] \ar[r]\ar[u] & \SigmaCcpu\ar[u]\\
		\sobstoch\ar[r,leftrightarrow,"\cong^{\op}"] \ar[u,leftrightarrow,"\cong"]& \cmsigmaCcpu \ar[r]\ar[u] & \msigmaCcpu\ar[u]\\
		\bairestoch\ar[u]\ar[r,leftrightarrow,"\cong^{\op}"]& \cpbenvcpu \ar[r]\ar[u] & \pbenvcpu\ar[u]\\
		\borelstoch\ar[u]\ar[r,leftrightarrow,"\cong^{\op}"]& \cpbsepcpu \ar[r]\ar[u] & \pbsepcpu\ar[u]
		\end{tikzcd}
	\]
	\caption{The probabilistic situation. All upward arrows are forgetful functors, bent arrows are left adjoints, and $\op$ indicates contravariance.}
	\label{fig:prob_side}
\end{figure}

Let us now turn to the probabilistic situation depicted in \cref{fig:prob_side}.
Since there is no natural notion of probabilistic morphism between Boolean algebras, the corresponding column is now absent, and we only deal with categories of measurable spaces and of \cstar{}s.
With this in mind, the functors and dualities depicted are largely analogous to those in the deterministic case of \cref{fig:det_side}, because the corresponding functors simply extend from deterministic morphisms to Markov kernels and cpu maps.
The only exception is the situation with $\stoch$.
This is the only one among our categories for which the deterministic morphisms do \emph{not} form a subcategory of the probabilistic morphisms, i.e.~the functor $\meas \to \stoch$ is \emph{not} faithful.
In fact, in $\stoch$ every measurable space is naturally isomorphic to its sobrification, and this upgrades the forgetful functor $\sobmeas \to \meas$ to an equivalence $\sobstoch\cong \stoch$.
However, our measurable Gelfand duality adjunction between $\meas$ and $\cCalg$ does \emph{not} extend to an adjunction between $\stoch$ and $\cSigmaCcpu$. 
This is a consequence of the existence of \cstars{}s admitting $\sigma$-states but no pure $\sigma$-states, such as $L^{\infty}([0,1])$, where $[0,1]$ carries the Lebesgue measure (\cref{rem:no_prob_adj}).

For both of our goals, it is important to also have a notion of product on objects, and we therefore also conduct a partial study of monoidal structures on our categories and consider to what extent our functors preserve these.
For our categories of measurable spaces with deterministic morphisms, the most reasonable choice of product is simply the usual product of measurable spaces.
This can be justified formally by noting that in every Markov category, the monoidal product is the categorical product on the subcategory of deterministic morphisms~\cite[Remark~10.13]{fritz2019synthetic}.
For our categories of Boolean algebras and \cstar{}s, which behave dually, we correspondingly consider coproducts.
For $A, B \in \SigmaB$, their coproduct $A \unitensor B$ can be constructed as the universal $\sigma$-completion of the coproduct $A \otimes B$ in $\Bool$, and the proof of this fact due to Sikorski is surprisingly \emph{not} purely formal.
The coproduct in $\csigmaB$ can be similarly constructed as the \emph{regular} $\sigma$-completion $A \regtensor B$ (\cref{prop:sigma_tensor_product}).
Since this equivalence $\sobmeas \cong \csigmaB$ necessarily preserves products, it follows that the $\sigma$-algebra on a product of two sober measurable spaces is the regular $\sigma$-completion of the tensor product of the $\sigma$-algebras (\cref{prop:concrete_regular}).

By using the equivalences of \cref{fig:det_side}, the universal and regular tensor products of \bsigma{}s transfer to analogous constructions for commutative \cstars{}s $\A$ and $\B$.
This produces the coproduct $\A \unitensor \B$ in $\cSigmaC$, and if both algebras are purely $\sigma$-representable, also the coproduct $\A \regtensor \B$ is in $\cmsigmaC$.
Naturally, both of these can be constructed as the regular and universal $\sigma$-completion of the usual $C^*$-tensor product $\A \otimes \B$, respectively.
We also consider the functoriality of these constructions with respect to probabilistic morphisms.
For example, we show that the regular tensor product on $\cmsigmaCcpu$ makes the dual equivalence with $\stoch$ into an equivalence of symmetric monoidal categories (\cref{prop:linf_sym}), which in fact is an equivalence of Markov categories (\cref{rem:markov_functor}).
Fortunately, the subtle distinction between $\unitensor$ and $\regtensor$ becomes unnecessary for commutative Pedersen--Baire envelopes, as the two coincide on $\cPBenv$ (\cref{prop:pb_tensor}, see also \cite[Proposition 6.7.(iv)]{jamneshan23foundational}).

In the noncommutative setting, the treatment of tensor products of \cstars{}s is more involved.
Already for vanilla \cstar{}s, there are different notions of tensor product, such as the minimal tensor product $\A \otimes_{\min} \B$ and the maximal tensor product $\A \otimes_{\max} \B$.
These are completions of the algebraic tensor product with respect to $C^*$-norms that are different in general.
To get a symmetric monoidal structure on a category of \cstars{}s, these tensor products will have to be completed to \cstars{}s of their own.
The first problem that arises with this is that it is unclear whether either tensor product of two \starshom{}s can be guaranteed to be a \starshom{} again.
This puts the functoriality of both $\unitensor_{\min}$ and $\unitensor_{\max}$ into question, and using the regular $\sigma$-completion instead will only make this more difficult.

A good solution may be to define a symmetric monoidal structure on $\PBenv$ only, where we can define $\A^\infty \otimes^\infty_{\max} \B^\infty \coloneqq (\A \otimes_{\max} \B)^\infty$.
This avoids the difficulties with showing that the tensor product of \starshom{}s is again a \starshom{}, and we indeed obtain a symmetric monoidal structure (\cref{prop:nc_pb_tensor}), which even extends to $\pbenvcpu$.

This finishes our overview of the categories and their relations as they appear in this paper.
How do these categories fare with regard to our goals?
Concerning the goal of finding a well-behaved Markov category for measure-theoretic probability, sober measurable spaces are well-known to be too general, since regular conditional probabilities between them do not need to exist.
The situation is much better for Baire measurable spaces~\cite[452N]{fremlin4}.
This is one reason for why we think that $\bairestoch$, or equivalently the opposite category of $\cpbenvcpu$, could be a well-behaved Markov category for measure-theoretic probability.
Another reason is that the Kolmogorov extension theorem holds for arbitrary families of Baire measurable spaces.\footnote{This is easiest to find in the literature through the correspondence between Baire probability measures and Radon probability measures~\cite[Theorem~7.3.1]{dudley02real}, and the fact that the theorem holds for the latter~\cite[Section~I.10]{schwartz1973radon}.}
On the other hand, this is not the case for $\borelstoch$, since standard Borel spaces are not closed under uncountable products.
Thus $\bairestoch$ is currently our best candidate, and the dual equivalence with $\cpbenvcpu$ facilitates an algebraic approach to it in the spirit of~\cite{segal1965algebraic}.

Fortunately, this matches up well with our tentative answer to the other question: in light of the relevance of well-behavedness of Baire measurable spaces in measure-theoretic probability and the existence of a nice symmetric monoidal structure, $\PBenv$ and $\pbenvcpu$ look like sensible candidates for well-behaved categories of quantum measurable spaces with deterministic and probabilistic morphisms, respectively.

\subsubsection*{Structure of the paper}

The paper is divided into four sections: \cref{sec:bool_sigma} discusses the theory of \bsigma{}s, \cref{sec:sigma_cstar} that of \cstars{}s, \cref{sec:equivalences_dualities} contains the dualities depicted in \cref{fig:det_side,fig:prob_side}, and \cref{sec:tensor_cstars} is devoted to tensor products of \cstars{}s.

We now give a more detailed summary.
After discussing basic definitions, \cref{sec:concrete_loomis_sikorski} is devoted to Loomis--Sikorski duality in the form introduced in~\cite{chen23univ}. 
At the beginning of \cref{sec:baire_envelopes}, we focus on the Loomis--Sikorski representation theorem and show how this ensures the existence of Baire envelopes (\cref{prop:baire_univ}). 
We then proceed to $\sigma$-completions in \cref{sec:bool_sigma_completions}, which are essential for the study of tensor products in \cref{sec:tensor_bool}.

\cref{sec:sigma_cstar} opens with some basics on \cstars{}s.
We then discuss, in \cref{sec:sigma_normal}, the concept of $\sigma$-normality for $*$-homomorphisms and cpu maps, a notion of $\sigma$-additivity adapted to the context of \cstars{}s.
Pedersen--Baire envelopes and their universal properties are the main focus of \cref{sec:pb_envelopes}.
This universal property facilitates a simple development of measurable functional calculus, as shown in \cref{sec:functional_calculus}.
The section then concludes in \cref{sec:sigma_compl} with a study of $\sigma$-completions.

\cref{sec:equivalence} puts together the Boolean and $C^*$-algebraic settings by constructing the equivalence $\SigmaB \cong \cSigmaC$.
We offer two distinct approaches for achieving this, where one uses Gelfand and Stone dualities (\cref{sec:gelfand_stone_ap}), while the other is more direct but less straightforward (\cref{sec:measurable_ap}).
In \cref{sec:concrete_duality}, we use this equivalence, combined with Loomis--Sikorski duality, to obtain measurable Gelfand duality and its corollaries.
The corresponding results with probabilistic morphisms are presented in \cref{sec:probabilistic}.

In \cref{sec:tensor_commutative_cstars}, we discuss tensor products of commutative \cstars{}s by mimicking the situation for \bsigma{}s. 
In particular, we obtain a technically important result, \cref{prop:pb_tensor}, which describes the universal tensor product of two commutative Pedersen--Baire envelopes.
This is then used to show that some of the equivalences in \cref{sec:equivalences_dualities} can be extended to symmetric monoidal equivalences. 
We conclude by providing a tensor product for Pedersen--Baire envelopes in \cref{sec:tensor_pbenv}.

\subsubsection*{Preliminaries}

To ensure accessibility for a broad range of researchers, we will cover detailed background on \bsigma{}s and monotone $\sigma$-complete \cstar{}s in \cref{sec:bool_sigma,sec:sigma_cstar}.
However, we assume that the reader has some basic knowledge of Boolean algebras, \cstar{}s and category theory.
Concerning the latter, familiarity with the concepts of adjunction, equivalence, and symmetric monoidal categories is recommended.

\subsubsection*{Acknowledgments}

We acknowledge support by the Austrian Science Fund (FWF P 35992-N). We further thank Tommaso Russo for fruitful discussions on functional analytic aspects, and the referee for helpful comments.
\section{\texorpdfstring{\bsigma s}{Boolean σ-algebras}}
\label{sec:bool_sigma}

To start our discussion on algebraic descriptions of measurable spaces, \bsigma{}s are a very natural choice, as they are the abstract version of $\sigma$-algebras. 
Here, we provide an introduction to \bsigma{}s with all the relevant results that will be used in the paper. 
A reader who is well-versed in the theory may skip most of this section, with the exception of \cref{sec:concrete_loomis_sikorski}, which discusses Loomis--Sikorski duality following \cite{chen23univ}.
Our only meaningful original contribution may be the use of sober measurable spaces, defined according to \cite{moss2022probability}. 
Throughout, we also introduce the relevant categories from \cref{tab:meas,tab:bool}.
\cref{rem:name_sigmacompl} justifies our choice of terminology for $\sigma$-completions, which differs from prior usage. 

\subsection{Basic definitions}
For the basics of Boolean algebras, we refer to~\cite[\S~1,2]{sikorski69boolean}.
Further precise references on Boolean $\sigma$-algebras will be given as we go along. 

\begin{defn}\label{def:Bool}
	$\Bool$ is the category of Boolean algebras and (Boolean algebra) homomorphisms.	
\end{defn}

\begin{nota}
	In a Boolean algebra $A$, we use $\wedge$, $\vee$, and $\neg$, to denote the conjunction, the disjunction, and the negation, respectively. To avoid confusion with the zero and unit of a \cstar, we write $\bot$ and $\top$ for the least and greatest elements, respectively.

	As is well-known, \newterm{Stone duality} is a dual equivalence 
	between $\Bool$ and the category of Stone spaces $\Stone$,
	\[
		\begin{tikzcd}
			\Bool \ar[r,leftrightarrow,"\cong^\op"] & \Stone.
		\end{tikzcd}
	\]
	The equivalence is given by the following functors $\clopen$ and $\stone$.
	For a Stone space $X$, the clopen sets form a Boolean algebra $\clopen(X)$.
	For a Boolean algebra $A$, the Stone space $\stone(A)$ is the set of ultrafilters on $A$.
	The elements of $A$ can be identified with clopen sets in $\stone(A)$, and the latter generate the topology of $\stone(A)$.
\end{nota}

The Boolean algebras that we consider in this paper will often be Boolean algebras of projections in a commutative $*$-algebra.
For this reason, we denote generic elements of a Boolean algebra by $p$, $q$ and $r$.

\begin{defn}[{\cite[\S~20]{sikorski69boolean}}]
    A \newterm{\bsigma}{} is a Boolean algebra which admits countable infima and suprema. For every countable sequence $( p_n )_{n\in \mathbb{N}}$, we write
    \[
    	\inf_n p_n \qquad \text{and} \qquad \sup_n p_n,  
    \]
    for the infimum and the supremum, respectively.
\end{defn}

Clearly the $\sigma$-algebra of a measurable space is a \bsigma, in which countable suprema and infima are countable unions and intersections, respectively.
However, not all \bsigma{}s can be obtained in this way; we will provide more details in \cref{sec:concrete_loomis_sikorski}, which is dedicated to the relation between \bsigma{}s and measurable spaces.

\begin{rem}[{\cite[\S~19]{sikorski69boolean}}]
	Given a sequence of elements $( p_n )_{n \in \mathbb{N}}$ in a \bsigma{} and an additional element $q$, we have
	\begin{equation}
		\label{eq:sup_vee_wedge}
		\sup_n \: (p_n \vee q) = \left(\sup_n p_n \right)  \vee q, \qquad 
		\sup_n \: (p_n \wedge q) = \left(\sup_n p_n \right)  \wedge q,
	\end{equation}
	and likewise for infima.
\end{rem}

\begin{defn}
    A \newterm{\bsigmahom}{} $\phi \colon A \to B$ between \bsigma s is a homomorphism such that
    \begin{equation}
        \phi \left (\sup_n p_n \right )= \sup_{n} \phi (p_n)    \label{eq:bsigmahom}    
    \end{equation}
    for every countable sequence $(p_n)_{n \in \mathbb{N}}$.
\end{defn}
  
Since homomorphisms preserve negation, this also implies that $\phi \left (\inf_n p_n \right )= \inf_{n} \phi (p_n)$.    
Other descriptions are given below, but we need the following notion of disjointness.

\begin{defn}\label{def:disjoint}
Two elements $p$ and $q$ in a Boolean algebra are called \newterm{disjoint} if $p\wedge q = \bot$.
\end{defn}

\begin{rem}\label{rem:sigmahom_add_mon}
	For a homomorphism $\phi \colon A \to B$ between \bsigma{}s, the following assertions are equivalent.
	\begin{enumerate}
		\item\label{it:sigmahom} $\phi$ is a \bsigmahom{};
		\item\label{it:inf0} For every sequence $(p_n)$ with $\inf_n p_n = \bot$, also $\inf_n \phi(p_n)=\bot$;
		\item\label{it:sigmaadd} ($\sigma$-additivity) For every sequence $(p_n)$ of mutually disjoint elements, \eqref{eq:bsigmahom} holds;
		\item\label{it:sigmamon} ($\sigma$-monotonicity) For every sequence $(p_n)$ that is monotone increasing, \eqref{eq:bsigmahom} holds.
	\end{enumerate}
	That \ref{it:sigmahom} implies the other properties is obvious.
	To prove that \ref{it:sigmaadd} and \ref{it:sigmamon} both imply \ref{it:sigmahom}, note that for any sequence $(p_n)$, one can define
	\[
	q_n \coloneqq p_n \wedge \neg p_{n-1} \wedge \dots \wedge \neg p_1, \qquad r_n \coloneqq p_n \vee p_{n-1} \vee \dots \vee p_1.
	\]
	By construction, $(q_n)$ is a sequence of mutually disjoint elements, while $(r_n)$ is monotone increasing, and $\sup_n p_n =\sup_n q_n = \sup_n r_n$. Therefore, the implications follow.
	If $\phi$ satisfies \ref{it:inf0}, take any supremum $p= \sup_n p_n$ and define $q_n \coloneqq p \wedge \neg p_n$. 
	By \eqref{eq:sup_vee_wedge}, $\inf_n q_n=\bot$, and hence also
	\[
		\phi(p) \wedge \neg \sup_n \phi(p_n) = \inf_n \phi(q_n) = \bot,
	\]
	since $\phi$ is a homomorphism satisfying \ref{it:inf0}.
	This gives $\phi(p) \le \sup_n \phi(p_n)$, while the direction $\phi(p) \ge \sup_n \phi(p_n)$ is trivial.
\end{rem}
  
\begin{defn}\label{def:sigmaB}
    $\SigmaB$ is the category whose objects are \bsigma s and whose morphisms are \bsigmahom s.
\end{defn}

\begin{rem}\label{rem:iso_sigmaiso}
    Every isomorphism of \bsigma s is a \bsigmahom{} (see~\cite[Equations~18.(6)-(6')]{sikorski69boolean}),\footnote{In particular, this means that the forgetful functor $\SigmaB \to \Bool$ \emph{forgets at most property-like structure} as defined in the nLab page ``stuff, structure, property'' (\href{https://ncatlab.org/nlab/show/stuff,+structure,+property}{https://ncatlab.org/nlab/show/stuff,+structure,+property}).} 
	but this is not true for general homomorphisms. 

    For the sake of an example, take the power set of natural numbers $\mathscr{P}(\mathbb{N})$, and consider the ideal $I$ given by finite subsets. Then any maximal ideal $M\supset I$ gives rise to a homomorphism $\phi \colon \mathscr{P}(\mathbb{N})\to \mathscr{P}(\mathbb{N})/M \cong \lbrace \bot, \top\rbrace $ whose kernel
    is precisely $M$.
	In other words, $M$ is associated to a \emph{non-principal ultrafilter} $U_M$ by setting $U_M \coloneqq \phi^{-1}(\top)$.
    Since every \bsigmahom {} $\mathscr{P}(\mathbb{N})\to \lbrace \bot, \top\rbrace$ is the characteristic function of a natural number, we conclude that $\phi$ cannot be a \bsigmahom, because we have $\phi(\lbrace n\rbrace)=\bot$ for every natural number $n$.
\end{rem}

\begin{rem}
    By definition, the category $\SigmaB$ is the category of models of a Lawvere theory with countably infinite arity, so in particular it is monadic over the category of sets~\cite[Section 9]{linton69outline}.
    Therefore, $\SigmaB$ has all small limits and colimits (see, for instance, \cite[Theorem~4.3.5]{borceux94handbook2}).
    In particular, the coproduct of two \bsigma s $A$ and $B$ exists, and we denote it by $A \unitensor B$.
\end{rem}
We also briefly discuss the notion of $\sigma$-ideal.
\begin{defn}
	A subset $I$ of a \bsigma{} $A$ is a \newterm{$\sigma$-ideal} if it is an ideal in the sense of Boolean algebras and furthermore $\sup_n p_n \in I$ whenever $(p_n) \subseteq I$. 
\end{defn}
\begin{fact}[{\cite[Theorem 21.1 and Example 22.F]{sikorski69boolean}}]
	Let $A$ be a \bsigma{} and let $I$ be a $\sigma$-ideal of $A$. Then the quotient $A/I$ is a \bsigma{} and the quotient map $A \to A/I$ is a \bsigmahom{}.
\end{fact}

We conclude this subsection with a characterization of \bsigma{}s from the point of view of Stone duality.
We start by discussing when countable suprema exist.

\begin{lem}[{\cite[Lemma 21.1]{halmos74boolean}}]
	\label{lem:sup_clopens}
	For a Stone space $X$ and a sequence of clopen sets $(U_n)_{n \in \mathbb{N}}$ in $X$, the supremum $\sup_n U_n$ exists in $\clopen(X)$ if and only if
	\[
		V \coloneqq \overline{\bigcup_n \, U_n}
	\]
	is open, and in this case $\sup_n U_n = V$.
\end{lem}

Dually, the infimum of the sequence $(U_n)$ is the interior of their intersection if that interior is closed.

\begin{proof}
	If $V$ is open, and if $W$ is any clopen containing all $U_n$, then clearly also $V \subseteq W$ since $W$ is closed.
	Therefore $V$ is indeed the supremum.
	Conversely if $W$ is the smallest clopen containing all $U_n$, then again $V \subseteq W$ as $W$ is closed.
	If there was $x \in W \setminus V$, then since a Stone space is ultranormal~\cite[Proposition 8.8]{prolla82topics}, there would exist a clopen $C \subseteq W$ with $V \subseteq C$ but $x \not \in C$, contradicting the minimality of $W$.
\end{proof}

\begin{defn}[{\cite[p.~130]{grove84substonean}}]
	\label{def:rickart}
	A compact Hausdorff space is a \newterm{Rickart space} if
	\begin{enumerate}
		\item its topology is generated by clopen sets, and
		\item the closure of a countable family $(U_n)$ of clopen sets is again clopen.
	\end{enumerate}
\end{defn}

These spaces are sometimes also called Boolean $\sigma$-spaces~\cite{halmos74boolean}, or even $\text{Stone}_{\sigma}$-spaces~\cite{jamneshan23foundational}.
By \cref{lem:sup_clopens}, the following is immediate.

\begin{cor}[{\cite[Theorem 12]{halmos74boolean}}]\label{cor:rickart_stone}
	Rickart spaces correspond to \bsigma{}s via Stone duality.	
\end{cor}

\subsection{Concrete \texorpdfstring{\bsigma s}{Boolean sigma-algebras} and Loomis--Sikorski duality}\label{sec:concrete_loomis_sikorski}

The $\sigma$-algebra of any measurable space is a \bsigma.
It is useful to have a name for those \bsigma s that arise like this.

\begin{defn}
	\label{def:concrete_bool}
    A \bsigma{} $A$ is \newterm{concrete} if there is an isomorphism $A \cong \Sigma(X)$ for some measurable space $(X,\Sigma(X))$.
\end{defn}

\begin{ex}[Not all \bsigma{}s are concrete, I]\label{ex:nonconcI}
	Let $\baire([0,1])$ be the $\sigma$-algebra given by the Borel sets of $[0,1]$.\footnote{The notation $\baire(X)$ will be introduced systematically in \cref{sec:baire_envelopes}. Here we have chosen $[0,1]$ instead of $\mathbb{R}$ because the general definition (\cref{def:baire}) only considers compact Hausdorff spaces. 
	However, for the purpose of this example, there is no relevant difference between  $\mathbb{R}$ and $[0,1]$.}
	We consider the $\sigma$-ideals $M,N \subset \baire([0,1])$, where $M$ is the family of meager sets and $N$ is the family of Lebesgue measure zero sets.
	Then two well-known examples of non-concrete \bsigma{}s are $\baire([0,1])/M$ and $\baire([0,1])/N$ (see, for instance, \cite[Example 24.A]{sikorski69boolean}).
	We postpone the rigorous proof to \cref{ex:nonconcII}; roughly speaking, these cannot be concrete because they fail to recognize their underlying points.	
\end{ex}

\begin{defn}\label{def:meas}
    $\csigmaB$ is the full subcategory of $\SigmaB$ on the concrete \bsigma s, and $\meas$ is the category of measurable spaces with measurable functions.
\end{defn}

Using the theory of concrete dualities \cite[Section~VI.4]{johnstone82stone}, it is possible to describe a contravariant adjunction between $\meas$ and $\SigmaB$, which mimics the Stone duality between Stone spaces and $\Bool$.
This adjunction was discussed explicitly in \cite[Section 4]{chen23univ}, where Chen coined the term \emph{Loomis--Sikorski duality}.
Let us discuss this result in detail, as it will be important throughout this paper.

In one direction, the functor
\[
	\Sigma \: : \: \meas \to \SigmaB^\op
\]
is the obvious one sending a measurable space $(X,\Sigma(X))$ to its $\sigma$-algebra $\Sigma(X)$. 
In the other direction, we start with a \bsigma{} $A$ and take the set of \bsigmahom{}s $A \to \lbrace \bot,\top\rbrace$, denoted by $\stones(A)$, equipped with the $\sigma$-algebra given by the sets of the form
\begin{equation}\label{eq:a_sigmaalg}
[a] \coloneqq \lbrace \phi \in \stones(A) \mid \phi(a)=\top \rbrace
\end{equation}
for any $a \in A$. This defines a functor $\stones : \SigmaB^\op \to \meas$, with action on morphisms given by composition.
The resulting contravariant adjunction is implemented by the bijection
\begin{equation}
	\label{eq:adjunction_bool_meas}
	 \meas (X, \stones(A))\cong \SigmaB (A, \Sigma(X))
\end{equation}
given by identifying both sides with the set of maps $A \times X \to \{ \bot, \top \}$ that are $\sigma$-homomorphisms in the first argument and measurable in the second, and this is clearly natural in $A$ and $X$.
The (co)unit of this adjunction in $\SigmaB$ is the natural map
\begin{align*}
	A & \longrightarrow \Sigma(\stones(A))	\\
	a & \longmapsto [a],
\end{align*}
which is a surjective \bsigmahom{} by definition. 
It is injective, and therefore an isomorphism, if and only if $A$ is concrete:

\begin{lem}[{\cite[Theorem 24.1]{sikorski69boolean}}]\label{lem:concretebool}
    A \bsigma{} $A$ is concrete if and only if for every $a >\bot$ in $A$ there exists a $\sigma$-homomorphism $\phi\colon A \to \lbrace \bot, \top \rbrace$ such that $\phi(a)= \top$. 
\end{lem}

By \cite[Section VI.4.5]{johnstone82stone}, this is enough to conclude that the adjunction \eqref{eq:adjunction_bool_meas} is idempotent. 
We have therefore proved the wanted duality statement.

\begin{thm}[Loomis--Sikorski duality]\label{thm:ls_duality} 
	There is a contravariant idempotent adjunction
	\begin{equation}\label{eq:ls_duality}
		\begin{tikzcd}
			\meas \ar[rr,bend left=20,"\Sigma"]\ar[rr,phantom,"\bot^{\op}" pos=0.55] && \SigmaB. \ar[ll,bend left=20,"\stones"]
		\end{tikzcd}
	\end{equation}
\end{thm}

\begin{conv}\label{conv:botop}
In situations such as \eqref{eq:ls_duality}, the symbol $\bot^{\op}$ indicates that the upper functor ($\Sigma$) is left adjoint to the lower functor ($\stones$) if we consider them as covariant functors between the category on the left ($\meas$) and the opposite of the category on the right ($\SigmaB^{\op}$).
\end{conv}

Being an idempotent adjunction, Loomis--Sikorski duality restricts to an equivalence between the two full subcategories on which the unit and counit components are isomorphisms.
On the side of $\SigmaB$, we have just seen that this is the subcategory of concrete \bsigma{}s.

On the other side, we have the full subcategory of $\meas$ consisting of those measurable spaces $X$ for which the unit component
\[
	X \to \stones(\Sigma(X))
\]
is an isomorphism.
We now relate these measurable spaces to the concept of sobriety introduced in \cite[Section 5]{moss2022probability}. 
To define sobriety, we briefly recall that the \emph{Giry monad} on $\meas$ has an underlying functor $P \colon \meas \to \meas$ associating to any measurable space $X$ its set of probability measures $PX$, equipped with the smallest $\sigma$-algebra making the evaluation map
\[
    \begin{array}{clcr}
        \operatorname{ev}_B \: \colon & PX &\longrightarrow &[0,1]\\
        & p & \longmapsto &p(B)
    \end{array}
\]
measurable for every $B \in \Sigma(X)$. 
Its unit is the natural transformation $\delta \colon \id \Rightarrow P$ sending an element $x \in X$ to the associated delta measure $\delta_x$.
The monad multiplication is not needed for our purposes, but we refer to the original paper~\cite{giry82probability} for the omitted details.

\begin{defn}[{\cite[Definition~5.1]{moss2022probability}}]
    A measurable space $X$ is \newterm{sober} if
    \begin{equation}
	    \label{eq:sober}
    \begin{tikzcd}
        X \ar[r,"\delta"] & PX \ar[r,shift right=0.7ex,"P\delta" below]\ar[r,shift left=0.7ex, "\delta"] & PPX
    \end{tikzcd}
    \end{equation}
    is an equalizer diagram in $\meas$.
\end{defn}

\begin{defn}\label{def:sobmeas}
    $\sobmeas$ is the full subcategory of $\meas$ consisting of sober measurable spaces.
\end{defn}

\begin{lem}
	\label{lem:sober}
	For a measurable space $X$, the following are equivalent:
	\begin{enumerate}
		\item\label{it:sober}
			$X$ is sober.
		\item\label{it:01valued}
			The formation of delta measures is a bijection between points of $X$ and $\lbrace 0,1 \rbrace$-valued probability measures on $X$.
		\item\label{it:unitiso}
			The unit component $X \to \stones(\Sigma(X))$ is an isomorphism.
	\end{enumerate}
\end{lem}
Item~\ref{it:01valued} is precisely the definition of $\sigma$-perfect $\sigma$-algebras according to~\cite[p.~98]{sikorski69boolean}.
\begin{proof}
	Although the equivalence of \ref{it:sober} and \ref{it:01valued} is known\footnote{See \url{https://ncatlab.org/nlab/revision/sober+measurable+space/7}.}, we offer a detailed proof for completeness.
	For general $X$, the equalizer of the two parallel maps in~\eqref{eq:sober} is given by the space of probability measures $p \in PX$ with $(P\delta)(p)= \delta(p)$. 
	Our goal is to show that these are precisely the $\lbrace 0,1 \rbrace$-valued probability measures, as then the equivalence of~\ref{it:sober} and \ref{it:01valued} is clear.

	By definition of $\delta$ and $P$, the equation $(P\delta)(p)= \delta(p)$ means that
	\begin{equation}
		\label{eq:equalizer_condition}
		p(\delta^{-1}(T)) = \begin{cases} 1 & \text{if } p \in T, \\ 0 & \text{otherwise} \end{cases}
	\end{equation}
	for all measurable $T \subseteq PX$.
	With $T \coloneqq \operatorname{ev}_S^{-1}((r,1])$ for measurable $S \subseteq X$ and arbitrary $r \in [0,1)$, we have $\delta^{-1}(T) = S$, and thus the equation~\eqref{eq:equalizer_condition} implies $p(S) \in \{0,1\}$.
	Conversely if $p$ is $\lbrace 0,1 \rbrace$-valued, then \emph{both} sides of~\eqref{eq:equalizer_condition} are $\lbrace 0,1 \rbrace$-valued, and it is therefore enough to check their equality on a generating class of measurable sets by the $\pi$-$\lambda$ theorem.
	For these one can again take the sets of the form $T \coloneqq \operatorname{ev}_S^{-1}((r,1])$ for $r \in [0,1)$ as above, and on these the equality is easy to check.

	We finish by showing the equivalence of \ref{it:01valued} and \ref{it:unitiso}.
	Note first that the $\lbrace 0,1 \rbrace$-valued probability measures on $X$ are in bijective correspondence with $\sigma$-homomorphisms $\Sigma(X) \to \lbrace \bot,\top\rbrace$, and we therefore obtain a natural injection $\stones(\Sigma(X)) \to PX$ whose image is precisely the set of $\lbrace 0,1 \rbrace$-valued probability measures, and which is a measurable isomorphism onto its image.
	The claim now follows from the fact that this isomorphism is also compatible with the unit map from $X$, i.e.~the diagram
	\[
		\begin{tikzcd}[sep=small]
			& \stones(\Sigma(X)) \ar[dd] \\
			X \ar[ur] \ar[dr] \\
			& PX
		\end{tikzcd}
	\]
	commutes.
\end{proof}

By \cref{lem:concretebool,lem:sober}, we therefore conclude the following.

\begin{thm}\label{thm:concrete_duality}
    Loomis--Sikorski duality restricts to a contravariant equivalence 
    \[
    \begin{tikzcd}
	    \csigmaB \ar[rr,bend left=15,"\stones"] \ar[rr,phantom,"\cong^\op"] && \sobmeas. \ar[ll,bend left=15,"\Sigma"]
    \end{tikzcd}
    \]
\end{thm}
This result is already stated in~\cite[Section 4]{chen23univ} in slightly different form. Our original contribution is the connection with the concept of sobriety.

\begin{rem}
	It is interesting to note how this construction is deeply reminiscent of that of Stone spaces and Stone duality, which is also the reason for the notation $\stones$. 
	We will briefly return to this analogy between topological and measurable dualities in \cref{sec:concrete_duality}, in particular \cref{prop:top_vs_meas_gelfand}.
\end{rem}

Let us recall that the \newterm{product} $X \times Y$ of two measurable spaces $X$ and $Y$ is defined as the cartesian product of sets, equipped with the $\sigma$-algebra generated by subsets of the form $U \times V$, where $U \subseteq X$ and $V \subseteq Y$ are measurable.

\begin{rem}\label{rem:sober_tensor_product}
	The product of two sober measurable spaces is sober. 
	Indeed, let $X$ and $Y$ be measurable spaces, and consider a $\lbrace 0,1 \rbrace$-valued probability measure on $X \times Y$. 
	Then this restricts to a $\lbrace 0,1 \rbrace$-valued probability measure on $X$ and on $Y$ separately, and these are equal to $\delta_x$ and $\delta_y$ for some $x \in X$ and some $y \in Y$, respectively. 
	Then the measure on the product is given by $\delta_{(x,y)}$ by a direct check. 
	Moreover, whenever $\delta_{(x,y)} = \delta_{(x',y')}$ holds for some $(x,y), (x',y')\in X \times Y$, sobriety of $X$ and $Y$ implies that $x=x'$ and $y=y'$, and therefore $(x,y)= (x',y')$.
	Hence by \Cref{lem:sober}\ref{it:01valued}, $X\times Y$ is sober as well.
\end{rem}

\subsection{Baire envelopes}
\label{sec:baire_envelopes}

Here we briefly discuss Baire envelopes and $\sigma$-completions of Boolean algebras.

\begin{defn}\label{def:baire}
	\begin{enumerate}
		\item The \newterm{Baire $\sigma$-algebra} $\baire(X)$ of a compact Hausdorff space $X$ is the smallest $\sigma$-algebra on $X$ making all continuous maps $X \to \mathbb{R}$ measurable. 
		\item A \newterm{Baire measurable space} is a measurable space $X$ whose $\sigma$-algebra is the Baire $\sigma$-algebra of a compact Hausdorff topology on $X$.
		\item The \newterm{Baire envelope} $A^\infty$ of a Boolean algebra $A$ is the Baire $\sigma$-algebra of its Stone space,
			\[
				A^\infty \coloneqq \baire(\stone(A)).
			\]
	\end{enumerate}
\end{defn}
\begin{defn}\label{def:bairemeas}
	The category $\bairemeas$ is the full subcategory of $\meas$ whose objects are Baire measurable spaces.
\end{defn}

Note that $A^\infty$ is a concrete \bsigma{} by definition.

\begin{rem}\label{rem:baire_borel}
	For second countable compact Hausdorff spaces, Baire sets and Borel sets coincide.\footnote{For example, use~\cite[Theorem~7.1.1]{dudley02real} together with Urysohn's metrization theorem.}
    	In general, this is not true: if we take an uncountable product of compact Hausdorff spaces with at least two points each, then each singleton set is Borel but not Baire~\cite[p.~223]{dudley02real}.
\end{rem}
\begin{rem}
	For a Stone space, the Baire $\sigma$-algebra coincides with the $\sigma$-algebra generated by the clopen sets (see \cite[Section 51]{halmos74measure} and the exercises of \cite[Section 23]{halmos74boolean}).
	Hence identifying a Boolean algebra $A$ with the clopens in $\stone(A)$ gives a natural inclusion $\iota \colon A \to A^\infty$.
\end{rem}

\begin{rem}
	\label{rem:baire_sober}
	Every Baire measurable space is sober, and therefore $\bairemeas \subseteq \sobmeas$ is a full subcategory.
	
	Indeed by~\cite[Theorem~7.3.1]{dudley02real} and the Riesz--Markov--Kakutani representation theorem, the probability measures on the Baire $\sigma$-algebra of a compact Hausdorff space $X$ are in bijective correspondence with the normalized positive functionals $C(X) \to \mathbb{C}$.\footnote{See the proof of \cref{cor:baire_eq} for a more general version of the present argument.}
	Under this correspondence, the $\{0,1\}$-valued probability measures turn into the multiplicative functionals.
	By Gelfand duality, the latter can be identified with the points of $X$, and it is clear that this bijection is indeed the map that corresponds to the unit map $X \to \stones(\Sigma(X))$.
\end{rem}

\begin{ex}[Not all \bsigma{}s are concrete, II]\label{ex:nonconcII}
	It now follows from \cref{rem:baire_sober} that $\baire([0,1])/M$ and $\baire([0,1])/N$, as defined in \cref{ex:nonconcI}, are not concrete.

	Let us consider a \bsigmahom\ $\phi \colon \baire([0,1])/M \to \lbrace \bot, \top \rbrace$. By composing it with the quotient map $\baire([0,1]) \to \baire({[0,1]})/M$, any $\phi$ would correspond to testing membership of a point of $[0,1]$, since $[0,1]$ is sober.
	But since the singleton set of any point is meager, no such \bsigmahom{} can factor across $\baire([0,1]) / M$, and we are done.
	The case of $\baire([0,1])/N$ is analogous.
\end{ex}

The following result, which is key in the theory of \bsigma{}s, involves the Baire category theorem in its proof.

\begin{thm}[Loomis--Sikorski representation theorem, {\cite[Theorem 29.1]{sikorski69boolean}}]\label{thm:ls_rep}
	For a \bsigma{} $A$, let $M$ be the $\sigma$-ideal of meager sets in $\stone(A)$. 
	Then the composite
	\[
	\begin{tikzcd}
		A \ar[r,"\iota"]& A^\infty \ar[r,two heads]& A^\infty/M
	\end{tikzcd}
	\]
	is a \bsigmaiso.
\end{thm}

In particular, the isomorphism $A \cong A^{\infty}/M$ shows that every \bsigma{} is a quotient of a concrete one.
Moreover, it gives us a canonical \bsigmahom{}
\begin{equation}
	\label{eq:loomis_sikorski_counit}
	\pi \colon A^\infty \longrightarrow A
\end{equation}
which is a retraction of the inclusion $\iota$, i.e.~$\pi\iota = \id_A$.
This has an important categorical consequence.

\begin{prop}[Universal property of the Baire envelope]\label{prop:baire_univ}
	Let $A$ be a Boolean algebra and $B$ be a \bsigma{}. Then every homomorphism $\phi\colon A \to B$ extends uniquely to a \bsigmahom{}  $\Phi\colon A^{\infty} \to B$ via the natural inclusion $A \to A^{\infty}$.
\end{prop}

Although this result may be known to readers well-versed in Boolean algebras, we have not found it stated explicitly in the existing literature. 
Its analogue for \cstar{}s (\cref{thm:pb_univ}) has been formulated before.

\begin{proof}
	By Stone duality, we have a continuous map $f \colon \stone(B) \to \stone(A)$, such that $\phi= f^{-1}$ takes clopen sets to clopen sets. We now consider the Baire envelopes $A^{\infty}$ and $B^{\infty}$.
	It is a direct check that $f^{-1}$ also maps Baire sets to Baire sets, so that it can also be considered as a map $A^{\infty} \to B^{\infty}$.
	It is a \bsigmahom{} because the countable suprema and infima are unions and intersections of sets, respectively, which are trivially preserved by $f^{-1}$.
	So by the Loomis--Sikorski representation theorem (\cref{thm:ls_rep}), we can compose with $\pi \colon B^\infty \to B$ and get a \bsigmahom{} $A^{\infty} \to B$. This extends $\phi$ by construction, since $\pi$ is a retraction.
	Uniqueness of the extension is obvious because $A^{\infty}$ is $\sigma$-generated by $A$. 
\end{proof}

Categorically, we have proved the existence of an adjunction
\[
    \begin{tikzcd}
        \Bool \ar[rr,phantom, "\bot"]\ar[rr,bend left=15,"(-)^{\infty}"] && \SigmaB \ar[ll,bend left=15]
    \end{tikzcd}
\]
where the lower arrow is the obvious forgetful functor. 
Its counit components $A^{\infty} \to A$ are exactly the surjections~\eqref{eq:loomis_sikorski_counit}.

\subsection{\texorpdfstring{$\sigma$}{σ}-completions}
\label{sec:bool_sigma_completions}

In order to move on to $\sigma$-completions, we switch gears a little and return to the basic definitions.

\begin{rem}\label{rem:bsigmahom_general} 
    The notion of $\sigma$-homomorphism $A \to B$ still makes sense when $A$ and $B$ are just Boolean algebras. 
    In this case, a \emph{\bsigmahom{}} $A \to B$ is a homomorphism satisfying 
	\[ 
		\phi \left (\sup_n p_n \right )= \sup_{n} \phi (p_n)  
	\] 
	whenever $\sup_n p_n$ exists in $A$.
	Moreover, the equivalent formulations of \cref{rem:sigmahom_add_mon} clearly generalize to this setting.
\end{rem}
\begin{defn}\label{def:boolshom}
	$\Boolshom$ is the category of Boolean algebras with \bsigmahom{}s. 
\end{defn}

\begin{defn}[{\cite[Section~2]{Sikorski50CartesianProd}}]
    Given a Boolean algebra $A$, a \newterm{$\sigma$-completion} is given by a pair $(B,j)$, where 
    \begin{enumerate}
        \item $B$ is a \bsigma;
        \item $j\colon A \to B$ is an injective $\sigma$-homomorphism;
        \item $B$ is $\sigma$-generated by $j(A)$, i.e.~$B$ is the smallest Boolean $\sigma$-subalgebra of $B$ containing $j(A)$.
    \end{enumerate}
\end{defn}

The set of $\sigma$-completions of $A$ carries a canonical preorder $\preceq$: we say that a $\sigma$-completion $(B,j)$ is \emph{below} a $\sigma$-completion $(C,k)$ if and only if there exists a \bsigmahom {} $\phi\colon B \to C$ such that 
\[
\begin{tikzcd}
    A \ar[r,"j"]\ar[rd,"k" below left] & B\ar[d, dashed, "\phi"]\\
    & C
\end{tikzcd}
\] 
commutes. 
The assumption that $B$ is $\sigma$-generated by $j(A)$ guarantees that such $\phi$ is unique if it exists.
Also it is clear that this $\preceq$ is reflexive and transitive.

\begin{rem}\label{rem:alg_sets_not_sigma}
	In general, the Baire envelope $A^\infty$ together with the natural inclusion $A \to A^{\infty}$
	is not a $\sigma$-completion of $A$, since this inclusion need not be a \bsigmahom{}.
	For instance, take $A \coloneqq \mathscr{P}(\mathbb{N})$. 
	If $A \to A^{\infty}$ were a $\sigma$-homomorphism, then \cref{prop:baire_univ} would immediately imply that all homomorphisms out of $A$ are $\sigma$-homomorphisms, but this is false because there are non-principal ultrafilters on $\mathbb{N}$ (\cref{rem:iso_sigmaiso}).

	Relatedly, if $A$ is an algebra of subsets of some set $X$, the $\sigma$-algebra $\sigma(A)$ generated by $A$ together with the inclusion $A \to \sigma(A)$ is not necessarily a $\sigma$-completion of $A$, again because the inclusion is typically not a \bsigmahom{}.
	For example if we consider $A = \clopen(X)$ for a Stone space $X$, then we get $A^\infty = \baire(X) = \sigma(A)$, and we are back in the situation of the previous paragraph.
\end{rem}

\begin{defn}
	Let $A$ be a Boolean algebra.
	\begin{enumerate}
        	\item The \newterm{universal $\sigma$-completion} of $A$ is a $\sigma$-completion $(B,j)$ which is minimal with respect to $\preceq$.
		\item The \newterm{regular $\sigma$-completion} of $A$ is a $\sigma$-completion $(B,j)$ such that $j$ is a \emph{complete} homomorphism, i.e. every existing supremum (and infimum) in $A$ is preserved by $j$.
	\end{enumerate}
\end{defn}

The existence of the universal and the regular $\sigma$-completions is proved in Sikorski's article \cite{Sikorski50CartesianProd} and also discussed in his later book~\cite[Sections 35 and 36]{sikorski69boolean}.
The use of the definite article ``the'' is to emphasize that these $\sigma$-completions are unique up to unique isomorphism. 
For the universal $\sigma$-completion this can be obtained from \cref{prop:sigma_completion} below, while for the regular one we refer to \cite[35.3]{sikorski69boolean}.
\begin{rem}\label{rem:name_sigmacompl}
The names adopted here are not standard in the literature, where the terms ``minimal'' and ``maximal'' are used instead (the latter because the regular $\sigma$-completion can be proven to be maximal with respect to $\preceq$). 
However, although e.g.~Halmos agrees with our convention for the direction of $\preceq$~\cite{halmos74boolean}, Sikorski uses the opposite convention, and so in his works minimal and maximal are used to mean the opposite of Halmos's.
Moreover, Sikorski later chose a different terminology in his book~\cite{sikorski69boolean}, where the regular $\sigma$-completion is simply called \emph{$\sigma$-completion}, possibly to suggest that this is the preferred one, while the universal $\sigma$-completion was called \emph{maximal $\sigma$-extension}.
In order not to dwell on the dispute, we have opted for \emph{universal} and \emph{regular}.
The former is motivated by \cref{prop:sigma_completion} below, while the latter hints at the well-behavedness of the Boolean algebra inside this $\sigma$-completion. 
In addition, while it is true that the universal $\sigma$-completion is indeed the unique $\preceq$-minimal $\sigma$-completion, there may be other $\preceq$-maximal $\sigma$-completions besides the regular one~\cite[p.~169]{sikorski69boolean} (see also~\cite[p.~303]{wright76regular}).
\end{rem}

\begin{nota}
	Given any Boolean algebra $A$, we write $\univ{A}$ and $\reg{A}$ to denote the universal and the regular $\sigma$-completions, respectively.
	To simplify notation, we identify the embedding of $A$ with the set-theoretic inclusion whenever an explicit description is unnecessary.
\end{nota}

The following construction of the universal $\sigma$-completion is instructive.
It essentially reformulates a discussion by Sikorski, found specifically in the \emph{first edition} of \cite{sikorski69boolean}.\footnote{See p.\ 121, 
noting that what we refer to as a universal $\sigma$-completion is a maximal $\sigma$-extension in his terminology.}

\begin{lem}
	\label{lem:univ_sigma_completion_bool}
	For a Boolean algebra $A$, let $I \subseteq A^{\infty}$ be the $\sigma$-ideal generated by all elements of the form $\inf_n p_n$, where $(p_n)$ is any sequence in $A$ such that $\inf_n p_n = \bot$ in $A$. 
	Then $A^\sigma \coloneqq A^{\infty}/I$ together with the composite $A \to A^\infty \to A^\infty/I$ is the universal $\sigma$-completion of $A$.
\end{lem}

\begin{proof}
	The composite $u \colon A \to A^\infty \to A^\infty/I$ is a \bsigmahom{} by construction, and $A^\infty/I$ is $\sigma$-generated by $A$ because already $A^\infty$ is.
	To show that $u$ is injective, take $A = \clopen(X)$ for a Stone space $X$ without loss of generality.
	Then $A^\infty = \baire(X)$ shows that $I$ is contained in the $\sigma$-ideal of meager sets in $X$, since by \cref{lem:sup_clopens} $I$ is $\sigma$-generated by the boundaries of countable unions of opens whose closure is open, and all these sets are meager.
	Therefore the injectivity follows by the Baire category theorem.

	If $j \colon A \to B$ is any $\sigma$-completion, then by \cref{prop:baire_univ}, we obtain an induced \bsigmahom{} $\tilde{j} \colon A^{\infty} \to B$.
	The fact that $I \subseteq \ker(\tilde{j})$ is clear by the assumption that $j$ is a \bsigmahom{}, and we therefore obtain the desired factorization $A^{\infty}/I \to B$.
\end{proof}

The following universal property of the universal $\sigma$-completion extends the minimality property from the definition.

\begin{prop}[{\cite[Theorem 13.4]{Sikorski50CartesianProd}}]
	\label{prop:sigma_completion}
	Let $A$ be a Boolean algebra and $A \subseteq \univ{A}$ its universal $\sigma$-completion.
	Then every $\sigma$-homomorphism $\phi \colon A \to B$ to a \bsigma{} $B$ extends uniquely to a \bsigmahom{} $\tilde{\phi} \colon \univ{A} \to B$.
\end{prop}

The construction $\univ{A} \coloneqq A^{\infty}/I$ from \cref{lem:univ_sigma_completion_bool} together with \cref{prop:baire_univ} immediately proves the statement.

\begin{rem}\label{rem:univ_adjunction_bool}
	\cref{prop:sigma_completion} means that there is an adjunction 
	\[
    \begin{tikzcd}
        \Boolshom \ar[rr,phantom, "\bot"]\ar[rr,bend left=15,"\univ{(-)}"] && \SigmaB \ar[ll,bend left=15]
    \end{tikzcd}
\]
	where the lower arrow is the forgetful functor, while its left adjoint $\univ{(-)}$ sends a Boolean algebra to its universal $\sigma$-completion.
	Since the unit components $A \to \univ{A}$ are monomorphisms, this left adjoint is faithful by abstract nonsense.
\end{rem}

We now discuss regular $\sigma$-completions, noting first that universal $\sigma$-completions and regular $\sigma$-com\-ple\-tions do not necessarily coincide, as we will see in \cref{ex:tensor_product} below.

\begin{prop}[{\cite[Example 35.J]{sikorski69boolean}}]
	\label{prop:boolean_regular}
	Let $A$ be a Boolean algebra. 
	Then its regular $\sigma$-com\-ple\-tion is given by $A^{\infty}/M$, where $M \subseteq A^{\infty}$ is the $\sigma$-ideal of meager subsets of $\stone(A)$.
\end{prop}

Interestingly, regular $\sigma$-completions do not behave functorially with respect to \bsigmahom{}s, as we now investigate.

\begin{rem}[Non-functoriality of the regular $\sigma$-completion]\label{rem:nonfunctoriality_regular}

	Let us now consider a Boolean algebra $A$ such that $\reg{A}\neq \univ{A}$. 
	If regular $\sigma$-completions were functorial, then every \bsigmahom{} $A\to B$ would extend to a \bsigmahom{} $\reg{A}\to \reg{B}$.
	Consider the inclusion $A \hookrightarrow \univ{A}$. Since $\univ{A}$ is a \bsigma{}, we must have $\reg{(\univ{A})}=\univ{A}$, so functoriality would induce a \bsigmahom{} $\reg{A}\to \univ{A}$ extending $A \hookrightarrow \univ{A}$. 
	This implies that $\reg{A}\preceq \univ{A}$, and since $\univ{A}$ is $\preceq$-minimal, we must have $\reg{A}=\univ{A}$, contradicting the hypothesis.
	Thus, regular $\sigma$-completions are not functorial.
\end{rem}

\subsection{Tensor products}\label{sec:tensor_bool}

Given Boolean algebras $A$ and $B$, we write $A \otimes B$ for their tensor product~\cite[\S{}13]{sikorski69boolean}, which is the coproduct in $\Bool$ with respect to the inclusion $A \to A \otimes B$ mapping $a \mapsto a \otimes \top$, and similarly for $B$.

\begin{prop}[{\cite[Theorem 6.3]{Sikorski50CartesianProd}}]
	\label{prop:sigma_tensor_product}
	Let $A$ and $B$ be two \bsigma s. Then their \newterm{universal tensor product}
	\[
		A \unitensor B \coloneqq (A \otimes B)^\sigma
	\]
	is their coproduct in $\SigmaB$ with respect to the usual inclusions $A \to A \otimes B$ and $B \to A \otimes B$.
\end{prop}

Interestingly, this is \emph{not} an immediate consequence of the universal properties of the individual constructions. 
First, it is already nontrivial that the two inclusions are \bsigmahom{}s.
Second, even if they are, it is not obvious that the unique extension $A \otimes B \to C$ of a \bsigmahom{} $\phi \colon A \to C$ and $\psi \colon B \to C$, where $C$ is a \bsigma{}, is necessarily a \bsigmahom{} again.
Sikorski's result is that these facts do indeed hold.

\begin{rem}\label{rem:unitensor_SigmaB}
	By \cref{prop:sigma_tensor_product}, $\unitensor$ gives rise to a symmetric monoidal structure on $\SigmaB$.
\end{rem}

It is worth noting that $A \unitensor B$ need not be concrete even if $A$ and $B$ are (see \cref{ex:tensor_product}).
For this property to hold, one needs to consider a \emph{regular} version of the tensor product. 

\begin{defn}
    Let $A$ and $B$ be two \bsigma s. Then their \newterm{regular tensor product} $A \regtensor B$ is the regular $\sigma$-completion of the tensor product of $A$ and $B$ as Boolean algebras.
\end{defn}

That the regular tensor product makes $\csigmaB$ into a symmetric monoidal category follows from Loomis--Sikorski duality. 
At this stage, it is unclear to us whether $\regtensor$ induces a symmetric monoidal structure on $\SigmaB$ (cf.~\cref{rem:nonfunctoriality_regular}).

\begin{prop}[{\cite[Theorems 12.2 and 12.3]{Sikorski50CartesianProd}}]
	\label{prop:concrete_regular}
	Let $A$ and $B$ be two concrete \bsigma s. Then $A \regtensor B$ is concrete as well, and the associated measurable space 
	is the product of the measurable spaces associated to $A$ and $B$, in the sense that
	\[
		\begin{tikzcd}[row sep=large]
			\csigmaB \times \csigmaB \ar[rr,"\regtensor"] \ar[d,"\stones"] && \csigmaB \ar[d,"\stones"] \\
			\sobmeas \times \sobmeas \ar[rr,"\times"] && \sobmeas
		\end{tikzcd}
	\]
	commutes up to canonical natural isomorphism.
\end{prop}

\begin{cor}[Universal property of the regular tensor product]\label{cor:univprop_regular}
Let $A$ and $B$ be two concrete \bsigma s. Then $A \regtensor B$ is the coproduct in $\csigmaB$.
\end{cor}

We can now conclude that the universal and the regular tensor product do not coincide in general.

\begin{ex}[{\cite[Example 37.A]{sikorski69boolean}}]
	\label{ex:tensor_product}
	Consider $\mathbb{R}$ together with an analytic non-Borel subset $X\subset \mathbb{R}$.
	Let $\Sigma$ be the $\sigma$-algebra of Borel subsets of $\mathbb{R}\setminus X$ and $\baire(\mathbb{R})$ the Borel $\sigma$-algebra of $\mathbb{R}$ (recall \cref{rem:baire_borel}). 
	Then $\Sigma \regtensor \baire(\mathbb{R})$ coincides with the $\sigma$-algebra generated in the cartesian product $(\mathbb{R}\setminus X) \times \mathbb{R}$.
	But this regular tensor product can be shown not to be the coproduct in $\SigmaB$, and therefore cannot coincide with $\Sigma \unitensor \baire(\mathbb{R})$.
	
	It also follows that $\Sigma \unitensor \baire(\mathbb{R})$ is not concrete, since otherwise it would also be the coproduct in $\csigmaB$.
\end{ex}

Nevertheless, the two tensor products actually coincide for standard Borel spaces, as we show next. First, let us recall their definition.

\begin{defn}\label{def:borelmeas}
	\begin{enumerate}
		\item A \newterm{standard Borel space} is a complete separable metric space equipped with its Borel (or equivalently Baire) $\sigma$-algebra.
		\item The category $\borelmeas$ is the full subcategory of $\meas$ whose objects are standard Borel spaces.
	\end{enumerate}
\end{defn}

It is easy to see that standard Borel spaces are sober~\cite[Example 5.4]{moss2022probability}.

\begin{prop}\label{prop:tensor_standardborel}
    Let $A$ and $B$ be $\sigma$-algebras of standard Borel spaces. 
    Then
    \[
	    A \unitensor B = A \regtensor B,
    \]
    and in particular both are concrete.
\end{prop}

In fact, this holds for all Baire measurable spaces~\cite[Proposition 6.7.(iv)]{jamneshan23foundational}.
Here we restrict our focus to standard Borel spaces. 
The general statement will be proved in \cref{sec:tensor_cstars} (\cref{prop:pb_tensor}), where our categorical framework will enable a concise treatment.

\begin{proof}
    By \cref{prop:sigma_tensor_product}, we know there is a unique \bsigmahom{} $\phi \colon A \unitensor B \to A \regtensor B$ that commutes with the coproduct inclusions $A \to A \regtensor B$ and $B \to A\regtensor B$.

    To construct an inverse of $\phi$, consider the inclusions
    \[
	    i \colon A \to A \unitensor B, \qquad j \colon B \to A \unitensor B.
    \]
    The Loomis--Sikorski representation theorem (\cref{thm:ls_rep})
    ensures that $A \unitensor B$ can be written as a quotient of some concrete \bsigma {} $C$.
    By the standard Borel assumption and~\cite[32.5]{sikorski69boolean}, $i$ and $j$ factor through $C$, and therefore by the universal property of $A \regtensor B$ we obtain $\psi \colon A \regtensor B \to C \to A \unitensor B$. 
    The universal properties of the two tensor products are now sufficient to conclude that $\phi\psi = \id$ and $\psi \phi = \id$, since these compositions preserve the inclusions.
\end{proof}

\section{\texorpdfstring{\Cstars s}{σC*-algebras}}
\label{sec:sigma_cstar}
We now focus on the theory of \cstars{}s, also known as \emph{monotone $\sigma$-complete \cstar{}s}.
These have been studied less than the more common \emph{monotone complete \cstar{}s}, which have a closer connection with $W^*$-algebras.
In this section, we will both review results from the literature~\cite{saito15monotone,pedersen2018automgroups} and present a number of original contributions.
This includes results that are not available in the existing literature but may nevertheless be considered expected, such as the GNS construction for \cstars{}s (\cref{rem:gns}) and the measurable functional calculus introduced in \cref{sec:functional_calculus}.
We also offer less straightforward original contributions, such as a characterization of the Pedersen--Baire envelope (\cref{thm:char_pb}), applied in \cref{ex:hilb_pb}, and an alternative description of the universal $\sigma$-completion used to derive the universal property given in \cref{cor:univ_minimal_2}.

\begin{conv}
	Since we are interested in the case of \emph{unital} \cstar s, we focus on this setting without considering the nonunital case in any form.
	Thus throughout the rest of the paper, ``\cstar'' will always mean ``unital \cstar'', and all $\ast$-homomorphisms are assumed to preserve the unit.
\end{conv}

\subsection{Basic definitions}\label{sec:basic_cstars}

We first recall the notion of monotone $\sigma$-complete \cstar{}, which we will more concisely write as \cstars. Given a \cstar{} $\A$, we write $\A_{\sa}$ for the subspace of self-adjoint elements, which we equip with its usual order induced by the cone of positive elements $\A_+$.

\begin{defn}
    A unital \cstar{} $\A$ is a \newterm{\cstars}{} if any norm-bounded and monotone increasing sequence $(a_n)_{n \in \mathbb{N}}$ in $\A_{\sa}$, i.e.~a sequence satisfying
    \[
    a_n \le a_{n+1}, \qquad\text{and}\qquad \norm{a_n} \le \lambda 
    \]
    for all $n$ and fixed $\lambda \ge 0$, admits a supremum $\sup_n a_n$ in the poset $(\A_{\sa},\le)$. 
\end{defn}

In symbols, we also write $(a_n) \nearrow a$ to mean that $(a_n)$ is a norm-bounded monotone increasing sequence with supremum $a$, and similarly for decreasing sequences $(a_n) \searrow a$. This also applies for general \cstar{}s whenever the supremum exists.
For the sake of brevity, we usually leave it understood that the elements of the sequence are self-adjoint whenever we consider a monotone sequence.

\begin{rem}
    Since we are always working in the unital case, a sequence is norm-bounded if and only if it is order bounded with respect to $\le$.
    For this reason, from now on we will simply call such a sequence \emph{bounded} with no danger of confusion.

    In the nonunital case, however, norm-boundedness and order boundedness are different, and even requiring the existence of a supremum for bounded monotone sequences results in two different notions of $\sigma$-completeness~\cite[Example 2.2.11]{saito15monotone}.
\end{rem}

\begin{ex}
	\label{ex:Wstar}
	In a $W^*$-algebra $\A$, even every bounded monotone net has a supremum, and therefore $\A$ is a \cstars{} in particular.
	For example, $\bhilb$ is a \cstars{} for every Hilbert space $\mathcal{H}$.
\end{ex}

\begin{ex}
	\label{ex:commutative_linf}
    In the commutative case, the main examples of interest to us are the algebras of bounded complex-valued measurable functions $\Linf(X)$ on a measurable space $X$.\footnote{We emphasize that no notion of almost sure equality is involved, already because no choice of measure on $X$ is given.
	In other words, the elements of $\Linf(X)$ are measurable functions, and not equivalence classes thereof.}
    It is easy to see that this is a commutative \cstar{} with respect to the pointwise operations, with $\Linf(X)_{\sa}$ being the set of bounded real-valued measurable functions.
    The monotone $\sigma$-completeness now holds because pointwise suprema of monotone sequences of measurable functions are measurable, and hence these pointwise suprema are the suprema in $\Linf(X)_{\sa}$; this pointwise supremum is itself a bounded function because of the assumption of an upper bound on the sequence, which is equivalent to it being uniformly bounded.

    It is worth noting that $\Linf(X)$ typically does \emph{not} have suprema for arbitrary bounded monotone nets and is therefore not a $W^*$-algebra in general.
    For example if $X$ is such that all singletons are measurable, then it is easy to see that suprema of such nets also have to be pointwise whenever they exist.
    Therefore if $S \subseteq X$ is not measurable, then the indicator functions of finite subsets of $S$ form a bounded monotone net which does not have a least upper bound.
\end{ex}

\begin{ex}
	\label{ex:product}
	If $(\A_i)_{i \in I}$ is a family of  \cstars{}s, then the (possibly infinite) product $\prod_{i \in I} \A_i $ is a \cstars{} as well, because the suprema can be taken componentwise.\footnote{The reasoning is similar to that of the proof of~\cite[Lemma 3.2.3]{saito15monotone}.}
\end{ex}

Another important notion of limit is that of infinite sums.
\begin{defn}\label{def:infinite_sum}
	Let $(a_i)_{i \in I}$ be a family of positive elements in a \cstar{}. 
	Its \newterm{infinite sum} is defined by 
	\[
	\sum_{i \in I} a_i \coloneqq \sup_{\substack{F \subset \Lambda \\ \text{finite}}}\ \sum_{i \in F} a_i,
	\]
	whenever it exists.
\end{defn}

In a \cstars, the infinite sum of a positive sequence exists as soon as the sequence of partial sums $(\sum_{i=1}^n a_i)_{n \in \mathbb{N}}$ is bounded.

Next, we highlight some important results on suprema, which will be used throughout the rest of the paper.
These are covered in \cite[Section 2.1]{saito15monotone}.\footnote{Specifically, Lemma 2.1.5 and Proposition 2.1.10.}

\begin{fact}\label{fact:suprema_commute}
	In every \cstar{} $\A$, suprema commute with addition, and compression whenever they exist: if $(a_n)\nearrow a$, $(b_n)\nearrow b$ and $c$ is an arbitrary element, then
		\begin{align*}
			\sup_n \left( a_n + b_n \right) & = \sup_n a_n + \sup_n b_n,\\ 
			\sup_n c^* a_n c & = c^* \left( \sup_n a_n \right) c.
		\end{align*}
	Moreover, $\sup_n a_n = -\inf_n (-a_n)$.
\end{fact}

We also have the following connection with norm convergence.
\begin{fact}[{\cite[Lemma 2.1.7]{saito15monotone}}]\label{fact:suprema_norms}
	In a \cstar{} $\A$, consider a bounded monotone increasing sequence $(a_n)$ norm-convergent to $a$. Then $a=\sup_n a_n$.
\end{fact}

In the commutative case, we have the following characterization of suprema and infima analogous to \cref{lem:sup_clopens}.

\begin{lem}
	\label{lem:sup_chaus}
	Let $X$ be a compact Hausdorff space.
	\begin{enumerate}
		\item\label{it:sup_chaus_description} A bounded monotone increasing sequence $(f_n)$ in $C(X)$ has a supremum if and only if the function $g : X \to \mathbb{R}$ defined by
		\[
			g(x) \coloneqq \inf_{U \ni x} \sup_{x' \in U, \: n \in \mathbb{N}} f_n(x')
		\]
		is continuous, where $U$ ranges over all open neighborhoods of $x$.
		In this case, we have $g = \sup_n f_n$.
		When it exists, this supremum is pointwise except on a meager set.
		\item\label{it:sup_chaus_inf0} A monotone decreasing sequence $(f_n)$ in $C(X)_+$ satisfies $\inf_n f_n = 0$ if and only if the set
		\[
			\{ x \in X \mid \inf_n f_n(x) > 0 \}
		\]
		is meager (cf.~\cite[Theorem 3.3]{wright72measures}).
	\end{enumerate} 
\end{lem}

\begin{proof}
	\begin{enumerate}
		\item If $g$ is continuous, we show that it is the desired supremum.
		Hence assume that $h \in C(X)$ satisfies $f_n \le h$ for all $n$.
		Then clearly also
		\[
			\inf_{U \ni x} \sup_{x' \in U, \: n \in \mathbb{N}} f_n(x') \le \inf_{U \ni x} \sup_{x' \in U} h(x').
		\]
		By continuity of $h$, the right-hand side is exactly $h(x)$.\footnote{This statement is precisely the upper semi-continuity of $h$.}
		Therefore $g \le h$, which makes $g$ the supremum of $(f_n)$.

		Conversely, suppose that $f \coloneqq \sup_n f_n$ exists in $C(X)$.
		Then $g \le f$ follows by the same argument as in the previous paragraph.
		For $f \le g$, suppose that there is $x \in X$ with $f(x) > g(x)$.
		Then for sufficiently small $\eps > 0$, by definition of $g$ there is a neighborhood $U \ni x$ such that $f_n(x') \le f(x') - \eps$ for all $n$ and $x' \in U$.
		Applying Urysohn's lemma gives a continuous $h : X \to [0,\eps]$ that vanishes outside $U$ and satisfies $h(x) = \eps$.
		Then $f - h$ is a continuous function lower bounded by all $f_n$, which contradicts the assumption that $f$ is the supremum.

		For the final statement, it is enough to show that for every $\eps > 0$ the set
		\[
			\left\{ x \in X \:\bigg|\: \sup_n f_n(x) < g(x) - \eps \right\}
		\]
		has empty interior.
		But this holds by the same argument as in the previous paragraph: if there was a nonempty interior, then $g$ could not be the supremum.
		\item If $\inf_n f_n=0$, consider the monotone increasing sequence $(- f_n)$, which by assumption has supremum $0$. The final statement in \ref{it:sup_chaus_description} then ensures that  
		\[
			\left\{ x\in X \:\bigg|\: \inf_n f_n(x)> 0 \right\} = \left\{ x \in X \:\bigg|\: \sup_n \, (-f_n(x))<0 \right\}
		\]
		is meager.
		Conversely, if the set is meager and $g \in C(X)$ is such that $g \le f_n$ for all $n$, then the assumption gives $g(x) \le 0$ for all $x$ outside of a meager set.
		The claimed $g \le 0$ then follows by continuity and the Baire category theorem.
		\qedhere
	\end{enumerate}
\end{proof}

We now record the $C^*$-analogue of \cref{cor:rickart_stone}. 
It is worth noting that the proof in this case is more subtle and is not an immediate consequence of \cref{lem:sup_chaus}.

\begin{fact}[{\cite[Theorem 2.1]{grove84substonean}}]
	\label{fact:rickart_gelfand}
	Under Gelfand duality, commutative \cstars{}s correspond to Rickart spaces.
\end{fact}


A key concept in the study of measurability is that of $\sigma$-generation, as it enables $\lambda$-$\pi$ arguments.
It is therefore natural to explore analogous concepts for \cstars{}s.
Recall first that a \emph{$\ast$-subspace} of a \cstar{} is a linear subspace that is closed under the involution.

\begin{defn} Let $\A$ be a \cstar{}.
   \begin{enumerate}
	   \item A \newterm{$\sigma$-subspace} $V \subseteq \A$ is a $\ast$-subspace that is \newterm{$\sigma$-closed}: whenever $(a_n)\nearrow a$ for $(a_n)\subseteq V$, then $a \in V$.
    \item The \newterm{$\sigma$-closure} of a $\ast$-subspace $W \subseteq \A$ is the smallest $\sigma$-subspace containing $W$ (equivalently, it is the intersection of all $\sigma$-subspaces containing $W$). 
    \item A $\sigma$-subspace $V \subseteq \A$ is \newterm{$\sigma$-generated} by a $\ast$-subspace $W \subseteq V$ if it is the $\sigma$-closure of $W$ in $\A$.
    \item A \newterm{$\sigma$-ideal} $I \subseteq \A$ is a two-sided closed ideal which is also $\sigma$-closed.
   \end{enumerate}
\end{defn}

\begin{rem}
	\begin{enumerate}
		\item Taking additive inverses shows that a $\sigma$-subspace is automatically closed under infima in the same sense as under suprema.
		\item It may seem misleading to reserve the term $\sigma$-ideal for two-sided ideals.
			However, we follow the existing literature in this regard~\cite[Definition 2.2.16]{saito15monotone}, and note that $\sigma$-closure as presently defined does not apply more generally because a genuinely one-sided ideal cannot be a $\ast$-subspace.
	\end{enumerate}
\end{rem}

We recall the following important facts on $\sigma$-closures, noting already that a $\sigma$-ideal is automatically norm-closed (\cref{lem:sigmaclosed_is_closed}).

\begin{fact}[{\cite[Proposition 2.2.21 and Corollary 5.3.8]{saito15monotone}}]\label{fact:sigma_closure}
	Let $\A$ be a \cstars{}. 
	\begin{enumerate}
		\item\label{it:sigmaclosedideal} If $I$ is a two-sided $\sigma$-ideal, then $\A/I$ is a \cstars{}.
		\item The $\sigma$-closure of any $C^{\ast}$-subalgebra $\B \subseteq \A$ is again a \cstars{}.
	\end{enumerate}
\end{fact}

Although we lack a counterexample, it does not seem an immediate consequence of the definition that a $\sigma$-subspace is norm-closed.
Nonetheless, this holds under minor additional requirements: for example, one can proceed as in the first step of the proof of \cite[Theorem 5.3.7]{saito15monotone} to show that a $\sigma$-subspace $V$ is norm-closed if $1 \in V$.
We also highlight the following.
\begin{lem}\label{lem:sigmaclosed_is_closed}
Let $V$ be a $\sigma$-subspace such that for every self-adjoint $a \in V$, also $a^2 \in V$.
Then $V$ is norm-closed.
\end{lem}
In particular, a $\sigma$-ideal is simply a $\sigma$-closed ideal (i.e., we do not need to additionally require norm-closure).
\begin{proof}
	Given self-adjoint $a \in \A$, we denote by $a^+$ and $a^-$ its \emph{positive} and \emph{negative parts}, respectively.
	We start by claiming that for every self-adjoint $a \in V$, we have $a^+\in V$. The idea is to take the positive element $a^2$ and the sequence of polynomials defined inductively as follows:\footnote{This sequence is taken directly from Prahlad Vaidyanathan's answer at \href{https://math.stackexchange.com/questions/2538604}{math.stackexchange.com/questions/2538604}.}
	\[
	p_0 (t) \coloneqq 0, \qquad p_{n+1}(t) \coloneqq p_n(t) + \frac{1}{2} (t-p_n(t)^2).
	\]
	This is a monotone increasing sequence that tends to $\sqrt{t}$ uniformly on compact subsets of $\mathbb{R}_+$.
	Thus by functional calculus, the sequence $(p_{n}(a^2))$ is norm-convergent to $\abs{a}=\sqrt{a^2}$. 
	Moreover, it is a bounded monotone increasing sequence, which belongs to $V$ by the assumption and \cref{fact:suprema_norms}, which give $\abs{a}=\sup_n p_n(a^2) \in V$.
	Therefore also $a^+ = \frac{1}{2} (a+\abs{a}) \in V$.
	
	Let now $(b_n)\subseteq V$ be a norm-convergent sequence with limit $b \in \A$. 
	To show that $b\in V$, we proceed as in \cite[p.~17]{saito15monotone}. 
	By considering self-adjoint and anti-self-adjoint parts separately, we can assume that the sequence is self-adjoint.
	Furthermore we can assume without loss of generality that $b_1 = 0$ and that $\norm{b_n - b} < \frac{1}{2^n}$ for all $n$, the latter by passing to a subsequence if necessary.
	Now set $a_n \coloneqq b_{n+1} - b_n$.
	Then it is clear that the partial sum sequences $(\sum_{j=1}^n a_j^+)$ and $(\sum_{j=1}^n a_j^-)$ are bounded monotone increasing. 
	Note that $a_j^+\in V$ by the previous paragraph, and hence also $a_j^-\in V$ by linearity. 
	Moreover, these sequences are Cauchy and therefore converge in norm to some $c, d \in \A$. 
	By \cref{fact:suprema_norms}, we conclude that $c=\sum_{j} a_j^+$ and $d=\sum_{j} a_j^-$, and both belong to $V$ by $\sigma$-closedness.
	Finally, $b = c-d$ is clear by $b_{n+1} = \sum_{j=0}^n a_j$ and the additivity of limits, and so $b \in V$.
\end{proof}


\begin{lem}
	\label{lem:sigma_closure_ideal}
	In a \cstars{} $\A$, the $\sigma$-closure of any two-sided $*$-ideal $I \subseteq \A$ is a $\sigma$-ideal.
\end{lem}

\begin{proof}
	Let $I^\sigma$ denote the $\sigma$-closure of $I$.
	Then the only nontrivial claim is that $I^\sigma$ is also closed under multiplication by arbitrary elements of $\A$.
	By the polarization identity, it suffices to show that $I^{\sigma}$ is closed under compression, which holds true because 
	\[
		J \coloneqq \{ x \in I^\sigma \mid a^* x a \in I^\sigma \text{ for all } a \in \A \}
	\]
	is a $\sigma$-subspace (by \cref{fact:suprema_commute}) containing $I$.
\end{proof}

\begin{ex}
	\label{ex:ideal_from_ideal}
	Let $X$ be a measurable space and $I \subseteq \Linf(X)$ a $\sigma$-ideal in the sense of \bsigma{}s.
	Then the set of functions with support in $I$,
	\[
		\I \coloneqq \{ f \in \Linf(X) \mid \mathrm{supp}(f) \in I \}
	\]
	is a $\sigma$-ideal in $\Linf(X)$.
	Moreover, it is the $\sigma$-subspace generated by the indicator functions $\chi_E$ for $E \in I$.

	Indeed that $\I$ is a $\sigma$-ideal is straightforward to see.
	Now if $V$ is any $\sigma$-subspace containing the $\chi_E$ for $E \in I$, then $V$ trivially also contains all simple functions with support in $I$.
	If $f \in V$ is any nonnegative function with support in $I$, then it can be written as the supremum of a monotone increasing sequence of simple nonnegative functions with support in $I$, and hence $f \in V$.
	Finally, a general $f \in \I$ can be written as a linear combination of nonnegative functions with support in $I$, and therefore indeed $\I \subseteq V$.
\end{ex}

\begin{ex}
	\label{ex:meager_ideal}
	Central to our investigations is a special case of the previous example.
	For a compact Hausdorff space $X$, we also write $\Linf(X)$ for the \cstars{} of bounded Baire functions on $X$.\footnote{This is a small abuse of notation, since we already introduced the notation $\Linf(X)$ for the \cstars{} of bounded measurable functions on a \emph{measurable} space in \cref{ex:commutative_linf}.}
	We denote by $\M$ the set of functions $f \in \Linf(X)$ for which the support
	\[
		\{ x \in X \mid f(x) \ne 0 \}
	\]
	is meager.
	Then $\M \subseteq \Linf(X)$ is a $\sigma$-ideal, and it is generated by the indicator functions of meager Baire sets.
	We will use this fact repeatedly.
\end{ex}

\subsection{\texorpdfstring{$\sigma$}{σ}-normal maps}\label{sec:sigma_normal}

In this subsection, we discuss maps between \cstars{}s which preserve countable suprema.
Similar to how for Boolean $\sigma$-algebras it was important to consider the preservation of countable suprema even if they do not all exist (\cref{rem:bsigmahom_general}), here we work with \cstar{}s throughout rather than merely \cstars{}s. This is relevant in the context of $\sigma$-completions in particular (\cref{sec:sigma_compl}).

\begin{defn}
	\label{def:sigma_normal}
	Let $\A$ and $\B$ be \cstar{}s.
	\begin{enumerate}
		\item A positive unital map $\phi \colon \A \to \B$ is \newterm{$\sigma$-normal} if for every $(a_n)\nearrow a$,
    			\begin{equation}
	    			\phi(\sup_n a_n) = \sup_n \phi(a_n).\label{eq:sigma_normal}
                \end{equation}
		\item A \newterm{$\sigma$-state} on $\A$ is a $\sigma$-normal state $\phi \colon \A \to \mathbb{C}$.
        \item A \newterm{\starshom}{} is a $\sigma$-normal $\ast$-homomorphism $\phi \colon \A \to \B$.\footnote{In this paper, since we always deal with \emph{unital} \cstar{}s, $*$-homomorphisms are required to be unital.}
	\item A \newterm{$\sigma$-representation} is a $\sigma$-normal $\ast$-representation $\pi \colon \A \to \bhilb$ for some Hilbert space $\mathcal{H}$.
	\end{enumerate}
\end{defn}
\begin{nota}
    For the sake of brevity, we abbreviate ``completely positive unital'' to \newterm{\cpu{}}, and write \newterm{\scpu{}} for maps that are both \cpu{} and $\sigma$-normal.
\end{nota}

\begin{defn}\label{def:calg}
	We define the following categories:
	\begin{center}
	\begin{tabular}{c|c|c}
		\multicolumn{1}{c}{\textit{Category}} & \multicolumn{1}{c}{\textit{Objects}}&\multicolumn{1}{c}{\textit{Morphisms}}\\
		\toprule 
		$\Calg$ & \cstar{}s& $\ast$-homomorphisms\\
		\hline 
		$\Calgshom$ & \cstar{}s & \multirow{2}{*}{\starshom{}s}\\
		$\SigmaC$ & \cstars{}s &\\
		\hline 
		$\Calgcpu$ & \cstar{}s& \cpu{} maps \\
		\hline 
		$\Calgscpu$ & \cstar{}s & \multirow{2}{*}{\scpu{} maps}\\
		$\SigmaCcpu$ & \cstars{}s &
	\end{tabular}
	\end{center}
	In each case, we put an additional $\mathsf{c}$ in front to refer to the full subcategory on commutative \cstar{}s. 
	For instance, $\cSigmaC$ is the category of commutative \cstars{}s with \starshom{}s.
\end{defn}

\begin{rem}\label{rem:sigmanormal_inf0}
		To show that a given map $\phi$ is $\sigma$-normal, it is enough to show that $\inf_n \phi (b_n)=0$ whenever $(b_n) \searrow 0$.
			Indeed the equivalence to \eqref{eq:sigma_normal} is immediate by considering $b_n \coloneqq a-a_n$ for any monotone increasing sequence $(a_n) \nearrow a$.
			We will frequently work with this version of the definition.
\end{rem}

\begin{rem}
	For any \cstars{} $\A$ and any $\sigma$-ideal $I \subseteq \A$, the quotient map $\A \to \A/I$ is easily seen to be a \starshom{}.
\end{rem}

\begin{ex}
	\label{ex:sigma_Linf}
	Let $X$ and $Y$ be measurable spaces.
	\begin{enumerate}
		\item\label{it:sigmastates_Linf}
			The $\sigma$-states on $\Linf(X)$ are in bijective correspondence with probability measures on $X$, where this correspondence assigns to every probability measure the $\sigma$-state given by integrating against it.
			For example, this can be seen as a special case of the Daniell--Stone theorem~\cite[Proposition~16.23]{royden1988analysis}.
		\item For measurable $f \colon X \to Y$, the induced map
    			\[
	    			f^*\colon \Linf(Y)\to \Linf(X)
    			\]
    			obtained by precomposing with $f$ is a \starshom, as is easy to see by recalling that the suprema are pointwise.
	\end{enumerate}
\end{ex}

An important consideration to keep in mind is how PVMs and POVMs arise as special cases, which we sketch next.

\begin{ex}[PVMs and POVMs]\label{ex:pvms}
	For a measurable space $X$ and a Hilbert space $\mathcal{H}$, we have the following correspondences, which we will develop in broader generality in \cref{thm:povm_integration} and \cref{cor:pvm} (see also \cref{def:povm} for the terminology of POVMs and PVMs).
	\begin{enumerate}
		\item $\bhilb$-valued POVMs on $X$ are in bijective correspondence with \scpu{} maps $\Linf(X) \to \bhilb$.
			Indeed starting with $\mu\colon \Sigma(X) \to \bhilb$, the associated \scpu{} map is usually written as an integral,
	\begin{align*}
		\Linf(X) & \longrightarrow \bhilb \\
		f & \longmapsto \int f \, \mathrm{d}\mu.
	\end{align*}
		For $\hilb = \mathbb{C}$, this specializes to the correspondence between $\sigma$-states and probability measures from \cref{ex:sigma_Linf}\ref{it:sigmastates_Linf}.
	\item By restricting to $\bhilb$-valued PVMs on $X$, we have a bijective correspondence between these and \starshom{}s $\Linf (X) \to \bhilb$.
	\end{enumerate}
\end{ex}

\begin{ex}
	Every {normal} positive unital map between $W^*$-algebras is trivially also $\sigma$-normal, since the preservation of suprema of all bounded monotone nets has the preservation of suprema of bounded monotone sequences as a special case.
\end{ex}

In some cases, $\sigma$-normality actually coincides with normality:

\begin{ex}
        \label{ex:sigma_normal_is_normal}
	On a $W^*$-algebra with separable predual, the $\sigma$-states are precisely the normal states.
	It is independent of ZFC whether this holds for \emph{all} $W^*$-algebras.
	
	Both statements are a consequence of the following discussion involving real-valued measurable cardinals. 
	Let us recall that an uncountable cardinal $\kappa$ is \emph{real-valued measurable} if there exists a $\kappa$-additive\footnote{By definition, this means additivity for families of cardinality \emph{strictly less} than $\kappa$.}
	probability measure on $\mathscr{P}(\kappa)$ which is not atomic, or equivalently not completely additive.
	It is well-known that the existence of a real-valued measurable cardinal is independent of ZFC, although this is hard to find in the literature stated in this precise form.\footnote{For example, one direction is clear because the first real-valued measurable cardinal has to be weakly inaccessible~\cite[Corollary 10.15]{jech03settheory}. The other direction is a consequence of~\cite[Lemma~17.3]{jech03settheory}.}
	
	We now assume that $\kappa$ is a cardinal strictly smaller than all real-valued measurable cardinals; in case that there is no real-valued measurable cardinal, this requirement vacuously holds. 
	For example, $\kappa \coloneqq \aleph_0$ is such a cardinal (in ZFC).
	Then consider a \wstar\ $\N$ 
	such that
	\[
	\sup \left \lbrace \abs{\Lambda} \ \middle\vert\, \begin{array}{l} \text{there exists a family of pairwise}\\
		\text{orthogonal nonzero projections } (p_{\lambda})_{\lambda \in \Lambda} \text{ in }\N \end{array} \right \rbrace \le \kappa,
	\]
	where $\abs{\Lambda}$ denotes the cardinality of $\Lambda$.
	A $W^*$-algebra satisfying this with $\kappa = \aleph_0$ is called \emph{countably decomposable}, and every $W^*$-algebra with separable predual satisfies this~\cite[Definition~5.5.14]{kadison1983fundamentals}.
	Let now $\rho\colon \N \to \mathbb{C}$ be a $\sigma$-state, for which we prove that it is normal.
    	By~\cite[Theorem~7.1.12]{kadison1986fundamentals},
        it is enough to show that $\rho$ is completely additive on projections: for any family $(p_\lambda)_{\lambda \in \Lambda}$ of pairwise orthogonal projections, we must have
        \begin{equation}
		\label{eq:completely_additive}
                \rho\left( \sum_{\lambda \in \Lambda} p_\lambda \right) = \sum_{\lambda \in \Lambda} \rho(p_\lambda).
        \end{equation}
	Defining
	\[
		\mu(\Omega) \coloneqq \rho\left( \sum_{\lambda \in \Omega} p_\lambda \right)
	\]
	for every $\Omega \subseteq \Lambda$ gives a measure on the power set of $\Lambda$.
	Since the cardinal $|\Lambda|$ is not real-valued measurable by assumption, we conclude that $\mu$ is atomic, and in particular completely additive.
	Therefore~\eqref{eq:completely_additive} holds.

	On the other hand, consider a real-valued measurable cardinal $\kappa'$.
	By assumption, there is an atomless $\kappa'$-additive probability measure $\mu \colon \mathscr{P}(\kappa') \to \mathbb{R}$, which therefore is $\sigma$-additive in particular.
	Then we can define a non-normal $\sigma$-state $\rho$ on $\N \coloneqq \ell^\infty(\kappa')$ in terms of the Lebesgue integral as
	\[
		\rho(a) \coloneqq \int_{\kappa'} a(\lambda) \, \mathrm{d}\mu(\lambda).
	\]
	That $\rho$ is not normal follows by construction, since normality would imply complete additivity of $\mu$.
\end{ex}

\begin{ex}
	\label{ex:lebesgue}
	On $L^{\infty}([0,1])$, the $\sigma$-states are precisely the normal states by \cref{ex:sigma_normal_is_normal}, and therefore correspond precisely to those probability measures on $[0,1]$ that are absolutely continuous with respect to the Lebesgue measure.
\end{ex}

\begin{ex}
	\label{ex:compression}
	For $b \in \A$ in any \cstar{} $\A$, the map $a \mapsto b^* a b$ is \cpu{}, and actually \scpu{} by \cref{fact:suprema_commute}.
	This is an analogue of the separate weak-$*$ continuity of the multiplication in a \wstar .
	
	We give an example application.
	If $(p_n)\nearrow p$ is a sequence of projections, then $p$ is a projection too, because $p$ is necessarily positive and
	\begin{equation}
		\label{eq:projection_sup}
		p^3 = p \left( \sup_n p_n \right) p = \sup_n p p_n p = \sup_n p_n = p,
	\end{equation}
	so the spectrum of $p$ is contained in $\lbrace 0,1\rbrace$ by functional calculus.
	If $(p_n)\searrow p$ with $(p_n)$ projections, $p$ is also a projection by considering $(1-p_n)\nearrow (1-p)$.\footnote{See also \cite[Proposition 2.2.7]{saito15monotone}.}
	In fact, any sequence of projections has a supremum and an infimum, regardless of monotonicity~\cite[Corollary 2.2.6]{saito15monotone}.
	This will be crucial for our comparison between commutative \cstars{}s and \bsigma{}s in \cref{sec:equivalences_dualities}.
\end{ex}

\begin{rem}
	Every $*$-isomorphism between \cstars{}s is a \starsiso,\footnote{In particular, this means that the forgetful functor $\SigmaC \to \Calg$ \emph{forgets at most property-like structure}, in the sense of the nLab page ``stuff, structure, property'' (\href{https://ncatlab.org/nlab/show/stuff,+structure,+property}{https://ncatlab.org/nlab/show/stuff,+structure,+property}).} but not every $*$-homomorphism is a \starshom{} (cf.\ \cref{rem:iso_sigmaiso} and \cref{rem:alg_sets_not_sigma}).
\end{rem}

We recall that any \cstar{} $\A$ has an associated compact Hausdorff space, called the \newterm{state space}, which we denote by $\state(\A)$. It is given by the set of states $\phi \colon \A \to \mathbb{C}$ equipped by the weak-* topology, i.e.\ the weakest topology making the evaluation maps $\phi \mapsto \phi(a)\in \mathbb{C}$ continuous for all $a \in \A$.

The following is an important general example of a \scpu{} map.

\begin{prop}
	For any \cstar{} $\A$, the map
	\begin{align*}
		\A & \longrightarrow C(\state(\A)) \\
		a & \longmapsto \left( \phi \mapsto \phi(a) \right)
	\end{align*}
	is \scpu{}.
\end{prop}

\begin{proof}
	The positivity is clear, and this implies complete positivity by commutativity of the codomain.
	So the main challenge is to prove $\sigma$-normality.

	Let $(a_n)$ be a monotone decreasing sequence in $\A$ with $\inf_n a_n = 0$.
	Then the following proof that its image in $C(\state(\A))$ has infimum $0$ bears some similarity to the proof of \cref{lem:sup_chaus}.
	Working with the weak-* topology as usual, we will prove that for every $\eps > 0$, the set
	\begin{equation}
		\label{eq:phi_ge_eps}
		\left\{ \phi \in \state(\A) \mid \phi(a_n) \ge \eps \:\: \forall n \right\}
	\end{equation}
	has empty interior in $\state(\A)$.
	This will imply that the set of $\phi \in \state(\A)$ with $\inf_n \phi(a_n) > 0$ is meager, which completes the proof.

	To this end, we consider $\A_+^*$, the cone of positive functionals on $\A$, together with its subset
	\[
		F \coloneqq \left\{ \phi \in \A_+^* \mid \exists n \: : \: \phi(a_n) < \eps \right\}.
	\]
	It is clear that $F$ is convex and hereditary: if $\phi \in F$ and $\psi \in \A_+^*$ with $\psi \le \phi$, then $\psi \in F$.
	Moreover, $F \cap \state(\A)$ is exactly the complement of \eqref{eq:phi_ge_eps} in $\state(\A)$.
	The weak-* closure $\overline{F}$ is also convex and hereditary. 
	To prove the claim, we reason by contraposition and assume that \eqref{eq:phi_ge_eps} has nonempty interior.
	This means that there is $\phi \in \state(\A) \setminus \overline{F}$, to which the separation theorem given by item (4) of~\cite{choi2025haagerup} applies\footnote{This new result of Choi establishes a conjecture of Haagerup~\cite[Problem~2.7]{haagerup1975weights}.} and yields $b \in \A_+$ such that
	\[
		\psi(b) \le 1 \quad \forall \psi \in \overline{F}, \qquad \phi(b) > 1.
	\]
	Therefore $c \coloneqq \frac{\eps}{\norm{b-1}} \left( b - 1 \right)$ satisfies
	\[
		\phi(c) > 0, \qquad 
		\psi(c) \le \eps \quad \forall \psi \in \state(\A), \qquad
		\psi(c) \le 0 \quad \forall \psi \in \overline{F} \cap \state(\A).
	\]
	Since $\psi(a_n) \ge \eps$ for all $\psi \in \state(A) \setminus \overline{F}$ and $\psi(a_n) \ge 0$ in general, the latter two conditions imply $a_n \ge c$ for all $n$.
	But then the resulting $c \le \inf_n a_n = 0$ gives $\phi(c) \le 0$, contradicting $\phi(c) > 0$.
\end{proof}

We now discuss in more detail $\sigma$-states and $\sigma$-representations.

\begin{rem}\label{rem:gns}
	Every $\sigma$-state induces, via the GNS construction, a $\sigma$-representation. This follows immediately by specializing the proof of~\cite[Proposition~3.3.9]{pedersen2018automgroups} from nets to sequences and using~\cref{fact:suprema_commute}.
\end{rem}

For any two $\sigma$-states $\phi$ and $\psi$, their convex combination $\lambda \phi + (1-\lambda) \psi$ is still a $\sigma$-state.
The $\sigma$-states constitute a subset of the state space, and it is of a very special type, namely a face:

\begin{lem}
	\label{lem:sigma_state_face}
	If $\phi$ and $\psi$ are states on a \cstar{} $\A$ and $\lambda \in (0,1)$, then
	\[
		\lambda \phi + (1-\lambda) \psi \: \textrm{ is a $\sigma$-state }  \quad \Longleftrightarrow \quad \phi, \psi \: \textrm{ are $\sigma$-states}.
	\]
\end{lem}

\begin{proof}
	The implication from right to left is clear, because suprema in $\mathbb{R}$ commute with scalar multiplication and addition.

    	So assume that $\lambda \phi + (1-\lambda) \psi$ is a $\sigma$-state.
    	If now $(a_n) \nearrow a$ in $A$, then $\sup_n \phi(a_n) \le \phi(a)$, and likewise for $\psi$.
	Therefore
	\begin{equation*}
    	\begin{split}
		\left( \lambda \phi + (1-\lambda) \psi \right) (a) & = \left( \lambda \phi + (1-\lambda) \psi \right) \left( \sup_n a_n \right) \\
								   & = \sup_n \left( \lambda \phi + (1-\lambda) \psi \right) (a_n) \\
								   & = \lambda \sup_n \phi(a_n) + (1-\lambda) \sup_n \psi(a_n) \\
								   & \le \lambda \phi(a) + (1-\lambda) \psi(a) .
    	\end{split}
	\end{equation*}
	Since we have the same expression both on the left and on the right, the two inequalities we used, namely $\sup_n \phi(a_n) \le \phi(a)$ and $\sup_n \psi(a_n) \le \psi(a)$, must be equalities.
\end{proof}

Despite being a face, the set of $\sigma$-states is typically not weak-* closed in the state space of a \cstar{}.
Let us give an instructive example of this.

\begin{ex}\label{ex:linfn}
    We consider $\ell^{\infty}(\mathbb{N})$, the \cstar{} of bounded sequences in $\mathbb{C}$. 
    Its states correspond bijectively to the finitely additive probability measures $\mathscr{P}(\mathbb{N}) \to [0,1]$, while the $\sigma$-states correspond to the ($\sigma$-additive) probability measures.
    To prove that the set of $\sigma$-states is not closed in the weak-* topology, it therefore suffices to show that there exists a sequence of probability measures having a merely finitely additive probability measure as an accumulation point.

    For example, take the sequence of probability measures that are uniform on the first $n$ natural numbers,
    \[
	    p_n\coloneqq \frac{1}{n} \sum_{i=0}^{n-1} \delta_i,
    \]
    where $\delta_i$ is the delta measure concentrated at $i$.
    By the compactness of the state space, this sequence must have some accumulation point, which must be a limiting relative frequency in the sense of \cite[Theorem 3]{kadane95finitelyadditive}.
    However, no such limiting relative frequency is a probability measure, because it must send every finite set to zero. So we have found a sequence of probability measures whose closure contains merely finitely additive probability measures.
\end{ex}

Also, there are examples where the set of $\sigma$-states is so small that it fails to separate the elements of the algebra, and it can even be empty:

\begin{prop}
	\label{prop:meager_nosigmastate}
	Let $X$ be a second countable compact Hausdorff space with no isolated points and $\M \subseteq \Linf(X)$ be the $\sigma$-ideal of functions with meager support as in \cref{ex:meager_ideal}.
	Then the \cstars{} $\Linf(X)/\M$ has no $\sigma$-states at all and no nonzero $\sigma$-representations. 
\end{prop}

The Boolean $\sigma$-algebra analogue of this example has appeared in the literature in slightly less generality~\cite[Example 21.F]{sikorski69boolean}.
In the context of monotone closed \cstar{}s, a related result is \cite[Theorem 4.2.17]{saito15monotone}, but this considers normal states instead of $\sigma$-normal states.

By second countability, the Baire and the Borel $\sigma$-algebras on $X$ coincide, as observed in \cref{rem:baire_borel}.
Thus $\Linf(X)$ as written in the statement can also be taken to be the \cstar{} of bounded Borel functions on $X$.

\begin{proof}
	Let us consider a $\sigma$-state $\phi$ on $\Linf(X)/\M$. 
	By composing $\phi$ with the quotient map $\Linf(X) \to \Linf(X)/\M$, we obtain a $\sigma$-state on $\Linf(X)$.
	This $\sigma$-state is given by integration against a probability measure $\mu$ on $X$ which sends each measurable meager set to $0$. 
	Since there are no isolated points, all singleton sets in $X$ are nowhere dense, and in particular meager; this implies that $\mu$ is nonatomic.
	We can thus apply \cite[Theorem 16.5]{oxtoby80measure},\footnote{Recall that for compact Hausdorff spaces, second countability is equivalent to metrizability by Urysohn's metrization theorem, and also equivalent to the separability of $C(X)$. Note also that assumption (ii) of the cited theorem is automatic since $\mu$ is necessarily outer regular (see, for instance, \cite[Lemma 1.16]{kallenberg97foundations}). Oxtoby himself also discusses the case of finite Borel measures after the proof.} which shows that $X$ can be written as the union of a meager set and a $G_{\delta}$ set of measure zero. 
	Therefore, $1=\mu(X)=0$, so we have arrived at a contradiction.

	If $\pi : \Linf(X)/\M \to \bhilb$ is a $\sigma$-representation, then if $\mathcal{H} \neq 0$, the composition with any vector state on $\bhilb$ gives a $\sigma$-state on $\Linf(X)/\M$.
	But since no such $\sigma$-states exist, we must have $\mathcal{H} = 0$.
\end{proof}

\Cref{prop:meager_nosigmastate} is in stark contrast to the situation for $W^*$-algebras, which can always be faithfully normally represented on a Hilbert space, and in particular have separating families of normal states.

\begin{defn}
	\label{def:sigma_representable}
	A \cstar{} $\A$ is \newterm{$\sigma$-representable} if $\A$ has a faithful $\sigma$-representation on some Hilbert space.
\end{defn}

In \cite[Section 4.5]{pedersen2018automgroups}, $\sigma$-representable \cstars{}s are called \emph{Borel $*$-algebras}.\footnote{The \emph{$\Sigma^*$-algebras} of~\cite{davies68borel} are closely related and coincide with Borel $*$-algebras in the commutative case~\cite[\S~4.5.14]{pedersen2018automgroups}.}
However, we prefer not to keep this naming, since these algebras are not generally associated with measurable spaces, as discussed in more detail in \cref{sec:concrete_duality}, and this may lead to confusion. 
Moreover, we believe that our terminology is more reminiscent of the defining property and thus improves readability.

\begin{ex}
	\begin{enumerate}
		\item Every $W^*$-algebra is $\sigma$-representable,
		\item The \cstar{} $\Linf(X)/\M$ from \cref{prop:meager_nosigmastate} is not $\sigma$-representable.
	\end{enumerate} 
\end{ex}

By the GNS construction of \Cref{rem:gns}, $\sigma$-representability is equivalent to the requirement that the $\sigma$-states should separate the elements of the \cstar{}.
Also it is clear that every \cstar{} has a universal $\sigma$-representable quotient.
For example, every $W^*$-algebra is $\sigma$-representable and thus coincides with its universal $\sigma$-representable quotient.
At the opposite extreme, the universal $\sigma$-representable quotient of $\Linf(X) / \M$ in the notation of \Cref{prop:meager_nosigmastate} is trivial.

\begin{rem}\label{rem:adj_representable}
	Recalling Loomis--Sikorski duality (\cref{thm:ls_duality}), we can consider the composite functor
	\[
	\Sigma \circ \stones \colon \SigmaB \to \csigmaB.
	\]
	This sends each \bsigma{} to its ``concretization'' and is left adjoint to the forgetful functor $\csigmaB \to \SigmaB$.
	
	Similarly, the forgetful functor $\repsigmaC \to \SigmaC$ has a left adjoint as well, where $\repsigmaC$ is the category of $\sigma$-representable \cstars{}s with \starshom{}s.
	There is an obvious candidate for defining this functor on objects: 
	Given a \cstars{} $\A$, consider the $\sigma$-representation $\A \to \bhilb$ induced by all its $\sigma$-states and write $\textsf{Rep}(\A)$ for its image in $\bhilb$.
	Then $\textsf{Rep}(\A)$ is a $\sigma$-representable \cstars{} by direct check.
	Moreover, every \starshom{} $\A \to \B$ with $\B \in \repsigmaC$ descends to a \starshom{} $\textsf{Rep}(\A) \to \B$. 
\end{rem}

\begin{rem}\label{rem:norm_sup}
	Whenever a \cstar{} $\A$ is $\sigma$-representable, then for every positive element $a\in \A$,
	\[
	\norm{a} = \sup \lbrace \phi(a) \mid \phi\colon \A \to \mathbb{C} \text{ is a }\sigma\text{-state}\rbrace.
	\]
	Indeed it suffices to let $\phi$ range over vector states in a faithful $\sigma$-representation of $\A$.

	It is instructive to note that in general there is no hope to attain this supremum. 
	For instance, take $\ell^{\infty}(\mathbb{N})$ and consider the positive element $(0,\frac{1}{2}, \frac{2}{3}, \frac{3}{4}, \dots)$. Its norm is $1$, but there is no $\sigma$-state realizing this norm, since the expectation value of this element with respect to any probability measure on $\mathbb{N}$ is strictly less than $1$.
\end{rem}

\begin{defn}
	Let $\A$ be a \cstar{}. Then a $\sigma$-state $\phi\colon \A \to \mathbb{C}$ is a \newterm{pure $\sigma$-state} if it is extremal in the convex set of all states.
\end{defn}

By \Cref{lem:sigma_state_face}, we can equivalently require $\phi$ to be extremal among the $\sigma$-states.

\begin{rem}
	Since the pure states on a commutative \cstar{} $\A$ are exactly the $*$-ho\-mo\-mor\-phisms $\A \to \mathbb{C}$,
	the pure $\sigma$-states on $\A$ are exactly the \starshom{}s $\A \to \mathbb{C}$.
\end{rem}

Similarly to \cref{def:sigma_representable} above, we can use pure $\sigma$-states to define a class of \cstars{}s that will be important for us by producing an equivalence of categories with $\csigmaB$ in the commutative case.

\begin{defn}
	\label{def:measurable_cstar}
    A \cstars{} is \newterm{purely $\sigma$-representable} if pure $\sigma$-states separate elements, meaning that for every positive $x\ge 0$, whenever $\phi(x)=0$ for all pure $\sigma$-states $\phi$, then $x=0$.
\end{defn}

For example, the $W^*$-algebra $L^{\infty}([0,1])$ is not a purely $\sigma$-representable \cstars{} (compare \cref{ex:lebesgue} with \cref{ex:nonconcI}).

\begin{defn}\label{def:msigmaC}
	We define the following categories.
	\begin{center}
		\begin{tabular}{c|c|c}
			\multicolumn{1}{c}{\textit{Category}} & \multicolumn{1}{c}{\textit{Objects}}&\multicolumn{1}{c}{\textit{Morphisms}}\\
			\toprule 
			$\msigmaC$ & purely $\sigma$-representable \cstars{}s & \starshom{}s \\
			\hline 
			$\msigmaCcpu$ & purely $\sigma$-representable \cstars{}s & \scpu{} maps\\
		\end{tabular}
	\end{center}
	Following \cref{def:calg}, an additional $\mathsf{c}$ in front indicates that the envelopes are commutative, or equivalently envelopes of commutative \cstar{}s.
\end{defn}

\begin{rem}\label{rem:meas_not_inclusion}
	A sub-\cstars{} of a purely $\sigma$-representable \cstar{} need not be purely $\sigma$-representable again.

	For instance, $\bhilb[{L^2([0,1])}]$ is a purely $\sigma$-representable \cstar{} since all vector states are pure $\sigma$-states.
	It contains the commutative $W^*$-algebra $L^{\infty}([0,1])$ of multiplication operators, which is not purely $\sigma$-representable.
	But the inclusion $L^{\infty}([0,1]) \subset \bhilb[{L^2([0,1])}]$ is normal, hence a \starshom{}.\footnote{The normality of this inclusion follows, for instance, from the GNS construction, as it corresponds to the faithful normal state induced by the Lebesgue measure.}

	In particular, despite the fact that the construction in \cref{rem:adj_representable} can be modified to obtain a purely $\sigma$-representable \cstars{} associated with any \cstars{} (use \cite[Proposition 2.11.8]{dixmier77Calg}), this does not yield a left adjoint to the forgetful functor $\msigmaC\to\SigmaC$ due to the lack of closure under inclusion.
\end{rem}

\subsection{Pedersen--Baire envelopes}\label{sec:pb_envelopes}

Given any \cstar{} $\A$, its enveloping $W^*$-algebra $\A^{**}$ has a nice universal property: every $*$-homomorphism $\A \to \B$ with $\B$ a $W^*$-algebra extends uniquely to a normal $*$-homomorphism $\A^{**} \to \B$.
One may now ask whether $\A$ has an analogous enveloping \cstars{}: this is exactly the Pedersen--Baire envelope we consider in this subsection.
The name Pedersen--Baire envelope, which we take from \cite[Definition 2.2.23]{saito15monotone}, hints both at the connection with the Baire $\sigma$-algebra that we develop in \cref{cor:baire_eq} and the fact that Pedersen has apparently been the first to study such envelopes~\cite[Section~4.5]{pedersen2018automgroups}.

\begin{defn}\label{def:pb}
	The \newterm{Pedersen--Baire envelope} $\A^{\infty}$ of a \cstar{} $\A$ is the $\sigma$-closure of $\A$ in its enveloping $W^*$-algebra $\A^{**}$.
\end{defn}

A manifestly equivalent description is that $\A^{\infty}$ is the $\sigma$-closure of $\A$ in $\bhilb[\mathcal{H}_u]$, where $(\mathcal{H}_u, \pi_u)$ is the universal representation of $\A$.

\begin{defn}\label{def:pbenv}
	We define the following categories.
\begin{center}
	\begin{tabular}{c|c|c}
		\multicolumn{1}{c}{\textit{Category}} & \multicolumn{1}{c}{\textit{Objects}}&\multicolumn{1}{c}{\textit{Morphisms}}\\
		\toprule 
		$\PBenv$ & Pedersen--Baire envelopes &\starshom{}s \\
		\hline 
		\multirow{2}{*}{$\pbsep$} & Pedersen--Baire envelopes &\multirow{2}{*}{\starshom{}s} \\
		&of separable \cstar{}s & \\
		\hline 
		$\pbenvcpu$ & Pedersen--Baire envelopes & \scpu{} maps\\
		\hline 
		\multirow{2}{*}{$\pbsepcpu$} & Pedersen--Baire envelopes & \multirow{2}{*}{\scpu{} maps}\\
		&of separable \cstar{}s &
	\end{tabular}
\end{center}
	Keeping the convention of \cref{def:calg}, an additional $\mathsf{c}$ in front indicates that the envelopes are commutative, or equivalently envelopes of commutative \cstar{}s.
\end{defn}

\begin{thm}[{Universal property of Pedersen--Baire envelopes, \cite[Proposition 1.1]{wright76regular}}]\label{thm:pb_univ}
    Let $\A$ be a \cstar{} and $\B$ a \cstars. Then every $*$-homomorphism $\phi \colon \A \to \B$ extends uniquely to a \starshom{} $\Phi \colon \A^{\infty}\to \B$.
\end{thm}
This theorem, translated in categorical language, states the existence of an adjunction
\begin{equation}\label{eq:pb_univ}
    \begin{tikzcd}
        \Calg \ar[rr,phantom, "\bot"]\ar[rr,bend left=15,"(-)^{\infty}"] && \SigmaC \ar[ll,bend left=15]
    \end{tikzcd}
\end{equation}
where the lower arrow is the functor forgetting $\sigma$-normality.\footnote{Compare with the case of the enveloping $W^*$-algebra: \url{https://ncatlab.org/nlab/revision/enveloping+von+Neumann+algebra/11}.}

\begin{rem}\label{rem:monos}
	By abstract nonsense, since the forgetful functor $\SigmaC \to \Calg$ is right adjoint, it preserves monomorphisms, and therefore all monomorphisms in $\SigmaC$ are injections because this is true in $\Calg$. 
	Conversely, it is clear that all injections are clearly monomorphisms.
	Therefore the monomorphisms in $\SigmaC$ are precisely the injective \starshom{}s.

	Concerning epimorphisms, it is clear that every surjective \starshom{} is an epimorphism.
	We do not know whether the converse holds too, although this is true in $\Calg$~\cite{hofmann95epim}.
\end{rem}

Analogously to the discussion around \cref{thm:ls_rep}, by abstract nonsense the counit components $\B^{\infty} \to \B$ are epimorphisms in $\SigmaC$.
These epimorphisms are actually surjective:

\begin{thm}[{\cite[Theorem 5.4.5]{saito15monotone}}]\label{thm:ls_rep_cstar}
	For every \cstars{} $\B$, the counit component $\B^{\infty} \to \B$ is surjective. 
	In particular, there exists a $\sigma$-ideal $\M \subseteq \B^\infty$ such that $\B\cong \B^{\infty}/\M$.
\end{thm}
\begin{rem}\label{rem:univ_baire_no_ls}
	The known proofs of \cref{thm:pb_univ} make use of \cref{thm:ls_rep_cstar} (see for instance~\cite[Section 4.5]{saito15monotone}), so even if we knew that epimorphisms in $\SigmaC$ were surjective, we could not deduce \cref{thm:ls_rep_cstar} without falling into circular reasoning.
	This is similar to the situation for Boolean algebras, where the proof of \cref{prop:baire_univ} involved the Loomis--Sikorski representation theorem.
\end{rem}

The universal property of \Cref{thm:pb_univ} also holds for positive maps, as shown in \cite[Proposition 5.4.7, Exercise 5.4.8]{saito15monotone}. 
For completeness, we provide a detailed proof for the case of \cpu{} maps using Stinespring's dilation theorem.

\begin{cor}[Universal property of Pedersen--Baire envelopes, II]\label{cor:pb_univ2}
    Let $\A$ be a \cstar{} and $\B$ a \cstars{}. 
	Then every \cpu{} map $\phi \colon \A \to \B$ extends uniquely to a \scpu{} map $\Phi \colon \A^{\infty} \to \B$.
\end{cor}
In categorical language, this gives an adjunction
\[
    \begin{tikzcd}
        \Calgcpu \ar[rr,phantom, "\bot"]\ar[rr,bend left=15,"(-)^{\infty}"] && \SigmaCcpu \ar[ll,bend left=15]
    \end{tikzcd}
\]
extending the one in \eqref{eq:pb_univ}.
\begin{proof}
	The uniqueness is clear as $\A^{\infty}$ is $\sigma$-generated by $\A$, so that any $\sigma$-cpu map out of $\A^{\infty}$ is uniquely determined by its restriction to $\A$.

	For the existence, suppose first that $\B$ is $\sigma$-representable.
	In this case, we can assume that it is given as a $\sigma$-subspace $\B \subseteq \bhilb$ for some Hilbert space $\mathcal{H}$.
	Then by Stinespring's dilation theorem, there is a Hilbert space $\mathcal{K}$ such that the given $\phi \colon \A \to \bhilb$ factors into a $*$-homomorphism $\tilde{\phi} \colon \A \to \bhilb[\mathcal{K}]$ followed by a \cpu{} map $\bhilb[\mathcal{K}] \to \bhilb$ of the form $a \mapsto v^* a v$ for an isometry $v : \mathcal{H} \to \mathcal{K}$, which is normal and hence $\sigma$-normal (cf.~\Cref{ex:compression}).
	Now by \cref{thm:pb_univ}, $\tilde{\phi} : \A \to \bhilb[\mathcal{K}]$ extends to a $\sigma$-homomorphism $\A^{\infty} \to \bhilb[\mathcal{K}]$. The composition of this map with the \scpu{} map  $\bhilb[\mathcal{K}] \to \bhilb$ gives a \scpu{} map $\A^{\infty} \to \bhilb$. 
    	Its image belongs to $\B$, since the image of $\A$ is contained in $\B$, which is a $\sigma$-subspace of $\bhilb$.
    	As wanted, we therefore have a \scpu{} map $\A^{\infty}\to \B$ which extends $\phi$ by construction.

	The general case can be reduced to the $\sigma$-representable case as follows.
	For given $\phi : \A \to \B$, the assumption that $\B$ is a \cstars{} yields a $\sigma$-surjection $\B^{\infty} \twoheadrightarrow \B$ fixing $\B \subseteq \B^{\infty}$ by \cref{thm:ls_rep_cstar}.
	Since $\B^\infty$ is $\sigma$-representable by construction, we can extend the composite $\A \to \B \subseteq \B^{\infty}$ to a \scpu{} map $\A^{\infty} \to \B^{\infty}$.
	Then the composite $\A^{\infty} \to \B^{\infty} \twoheadrightarrow \B$ extends the original $\phi \colon \A \to \B$.
\end{proof}

\begin{rem}\label{rem:sigmageneral}
    Let $\A \subseteq \B$, where $\B$ is a \cstars{} $\sigma$-generated by $\A$.
    Then by the universal property, we obtain a \starshom{} $\phi\colon \A^{\infty} \to \B$, which is surjective because $\B$ is $\sigma$-generated by $\A$. 
    In particular, $\B \cong \A^{\infty}/\I$, where $\I=\ker \phi$ is a $\sigma$-ideal.
\end{rem}

We now work towards an alternative characterization of Pedersen--Baire envelopes with a more probabilistic flavor.
This is inspired by the first of the four characterizations of standard Borel spaces given in~\cite{bjornsson79four}.\footnote{The reader interested in further characterizations of standard Borel spaces may also consult the follow-up paper~\cite{bjornsson80note}.}
We start with a basic observation on \cstar{}s in general.

\begin{lem}\label{lem:state_factor_injection}
	Let $\phi \colon \A \to \B$ be a $*$-homomorphism between \cstar{}s such that all states on $\A$ factor through $\phi$. Then $\phi$ is injective.
\end{lem}
\begin{proof}
    We want to show that whenever $a \in \A$ satisfies $\phi(a)= 0$, then $a=0$. Since $\phi$ is a $*$-homomorphism, we also have that $\phi(a^* a)= \phi(a^*)\phi(a) = 0$. Therefore, 
    \[ 
    \norm{a^*a}= \max \lbrace \psi (a^* a) \mid \psi \text{ is a state on }A \rbrace \le \norm{\phi(a^* a)} = 0,
    \] 
    from which $a=0$ follows.
\end{proof}

\begin{thm}[Characterization of the Pedersen--Baire envelope]\label{thm:char_pb}
	Let $\A$ be a \cstar {} and $\B$ a \cstars{} together with a $\ast$-homomorphism $j \colon \A \to \B$ such that $\B$ is $\sigma$-generated by $j(\A)$.
   Then the following are equivalent:
\begin{enumerate}
    \item\label{it:pb_env} $\B$ is isomorphic to the Pedersen--Baire envelope of $\A$, with $j$ being the inclusion;
    \item\label{it:pb_states_char} $j$ induces a bijective correspondence between $\sigma$-states of $\B$ and states of $\A$.
\end{enumerate}
\end{thm}

\begin{proof}
        The implication \ref{it:pb_env}$\Rightarrow$\ref{it:pb_states_char} is an instance of~\cref{cor:pb_univ2}, so we focus on the converse. 
 
	By \cref{lem:state_factor_injection}, $j$ is an injection. Let us take the universal representation $\pi_{u,\A}$ of $\A$.
		By assumption and the GNS construction, we can build a $\sigma$-representation $\pi_{\B}$ of $\B$ which restricts to $\pi_{u,\A}$.
        Since $\B$ is $\sigma$-generated by $j(\A)$, the image of $\pi_{\B}$ is $\sigma$-generated by the image of $\pi_{u,\A}$; therefore, $\pi_{\B}$ restricts to a \starshom{} $\B \to \A^{\infty}$.

        Moreover, the inclusion $j\colon \A \hookrightarrow \B$ extends uniquely to a \starshom {} $\A^{\infty} \to \B$ by \cref{thm:pb_univ}. 
        Therefore, $\A^{\infty} \to \B \to \A^{\infty}$ must be the identity, since it is the unique extension of the inclusion $\A \hookrightarrow \A^{\infty}$. 
        In particular, $\A^{\infty} \to \B$ is a $\sigma$-monomorphism and hence injective.
        Since $\B$ is $\sigma$-generated by $j(\A)$ and contains a copy of $\A^{\infty}$, it must coincide with this copy of $\A^{\infty}$, and so the claim follows.
\end{proof}

The following statement is a rephrasing of \cref{thm:char_pb}, which emphasizes the obtained characterization of Pedersen--Baire envelopes without a fixed \cstar{} $\A$. 
\begin{cor}\label{cor:pb_env_char_states}
    The following are equivalent for a \cstars{} $\B$:
    \begin{enumerate}
        \item $\B$ is isomorphic to the Pedersen--Baire envelope of some \cstar {} $\A$;
        \item\label{it:pb_states_char2} $\B$ is $\sigma$-generated by some $C^{\ast}$-subalgebra $\A$ such that every state on $\A$ uniquely extends to a $\sigma$-state on $\B$.
    \end{enumerate}
\end{cor}

\begin{rem}\label{rem:sigmagen}
	In~\ref{it:pb_states_char2}, the requirement of $\sigma$-generation cannot be omitted even if the bijection between states on $\A$ and $\sigma$-states on $\B$ holds.
    	To construct an example, note that $\A \coloneqq \mathbb{C}$ is its own Pedersen--Baire envelope (see also the forthcoming \cref{ex:findim}). We now consider a nontrivial \cstars{} $\A'$ without any $\sigma$-states (\Cref{prop:meager_nosigmastate}). 
	Then the direct sum $\B=\mathbb{C} \oplus \A'$ is a \cstars{} as well by \Cref{ex:product} and contains $\mathbb{C}$ via the diagonal inclusion, which is unital. 
	A state on $\B$ is a convex combination of a state on $\mathbb{C}$ and a state on $\A'$, and this description applies to $\sigma$-states too since we can reason componentwise.
    	But since $\A'$ has no $\sigma$-states by assumption, it follows that the $\sigma$-states on $\B$ are in bijective correspondence with the $\sigma$-states on $\mathbb{C}$.
	Therefore, the inclusion induces a bijection between states on $\mathbb{C}$ and $\sigma$-states on $\B$, but the latter is clearly not isomorphic to $\mathbb{C}$, so it is not the Pedersen--Baire envelope.
\end{rem}

We now discuss some examples of Pedersen--Baire envelopes, starting with the commutative case, which will be relevant for our treatment of functional calculus (\Cref{sec:functional_calculus}).

\begin{cor}[{\cite[Remark 5.3.4(iii)]{saito15monotone}}] \label{cor:baire_eq}
	For a compact Hausdorff space $X$, the Pedersen--Baire envelope of $C(X)$ is $\Linf(X)$, where $X$ is equipped with the Baire $\sigma$-algebra.
\end{cor}

Thus a commutative \cstars{} is a Pedersen--Baire envelope if and only if it is isomorphic to some $\Linf(X)$ in the sense of \cref{ex:meager_ideal} for suitable compact Hausdorff $X$.

\begin{proof}
	Since countable suprema in $\Linf(X)$ are pointwise, the definition of Baire function shows that the inclusion $C(X) \subseteq \Linf(X)$ is $\sigma$-generating.
	The bijection between states on $C(X)$ and $\sigma$-states on $\Linf(X)$ needed for an application of \cref{thm:char_pb} holds because both correspond to Baire probability measures on $X$: for $C(X)$, this is by the Riesz--Markov--Kaku\-tani representation theorem and \cite[Theorem 7.3.1]{dudley02real}, while for $\Linf(X)$ see~\Cref{ex:sigma_Linf}\ref{it:sigmastates_Linf}.
\end{proof}

\begin{rem}
	\label{rem:baire_borel_Cstar}
	The universal property of $\Linf(X)$ as the Pedersen-Baire envelope of $C(X)$ is well-known in operator theory, where it is considered with $\bhilb$ as the target \cstars{}~\cite[Theorem~2.6.3]{arveson02spectral} and used to derive Borel functional calculus and the spectral theorem.\footnote{Note that Arveson also assumes $X$ to be metrizable, which is because he works with Borel rather than Baire functions.}
	We will discuss this further in \cref{sec:functional_calculus}.
\end{rem}

For the moment, let us give some other examples of Pedersen--Baire envelopes.

\begin{ex}\label{ex:findim}
	Every finite-dimensional \cstar{} is its own Pedersen--Baire envelope, for example by \cref{thm:char_pb} because all its states are $\sigma$-states.
\end{ex}

\begin{ex}\label{ex:hilb_pb}
	Let $\mathcal{H}$ be a separable infinite-dimensional Hilbert space, and let $\khilb_1$ be the \cstar{} of compact operators on $\mathcal{H}$ in its unital form, i.e.~the algebra of bounded operators that differ from a compact operator by a multiple of the identity.
	Then we claim that its Pedersen--Baire envelope is given by
	\[
		\khilb_1^{\infty} = \bhilb \oplus \mathbb{C}
	\]
	via the inclusion 
	\[ 
	\begin{array}{rcl}
		\khilb_1 & \lhook\joinrel\longrightarrow & \bhilb\oplus \mathbb{C} \\
		a+ \lambda 1 &\longmapsto& (a+\lambda 1, \lambda).
	\end{array}
	\] 
	To prove this, we verify the conditions of \cref{thm:char_pb}.
	First, note that separability of $\mathcal{H}$ ensures that $\bhilb$ is $\sigma$-generated by the nonunital algebra of compact operators $\khilb$, and therefore $\khilb_1$ $\sigma$-generates $\bhilb \oplus \mathbb{C}$.
	Concerning the second condition, every state on $\khilb_1$ is a convex combination of a state on $\khilb$ and the state $a+\lambda 1 \mapsto \lambda$,
	and similarly every $\sigma$-state on $\bhilb \oplus \mathbb{C}$ is a convex combination of a $\sigma$-state on $\bhilb$ and the identity of $\mathbb{C}$.
	By the construction of the inclusion $\khilb_1 \hookrightarrow \bhilb\oplus \mathbb{C}$, we infer that the desired correspondence of states holds if it does so between states on $\khilb$ and $\sigma$-states on $\bhilb$. 
	Indeed by \Cref{ex:sigma_normal_is_normal}, the $\sigma$-states on $\bhilb$ are exactly the normal states, and therefore of the form $a \mapsto \mathrm{tr}(a \rho)$ for a positive operator $\rho$ of unit trace~\cite[Theorem~2.4.21]{bratteli87operator}, but these are also exactly the states on $\khilb$~\cite[Theorem~4.2.1]{murphy90cstar}.
	Therefore \cref{thm:char_pb} applies and yields $\khilb_1^{\infty} = \bhilb\oplus \mathbb{C}$. 
\end{ex}

We now consider how the elements of the Pedersen--Baire envelope of a \cstar{} $\A$ behave as functions on the state space.
To this end, recall first that the pairing between $\A$ and its state space $\state(\A)$ extends uniquely to a pairing between $\A^{**}$ and $\state(\A)$, and this still separates the elements of $\A^{**}$.
In this way, one can identify the elements of $\A^{**}$ with the bounded affine functions on $\state(\A)$~\cite[Proposition~2.128]{alfsen2001statespaces}.

\begin{prop}
	\label{prop:pb_state}
	The pairing between $\A^{**}$ and $\state(\A)$ induces a commutative diagram
	\begin{equation}
		\label{eq:pb_state_diag}
		\begin{tikzcd}
			\A \ar[hookrightarrow,r] \ar[hookrightarrow,d,"\textrm{\scpu{}}"'] & \A^\infty \ar[hookrightarrow,d,"\textrm{\scpu{}}"'] \\
			C(\state(\A)) \ar[hookrightarrow,r] & \Linf(\state(\A))
		\end{tikzcd}
	\end{equation}
\end{prop}

\begin{proof}
	Since every $\phi \in \state(\A)$ is normal on $\A^{**}$, the induced map $\A^{**} \to \ell^\infty(\state(\A))$ is \scpu{}.
	We still need to prove that the image of $\A^{\infty}$ under this map is contained in the subspace $\Linf(\state(\A)) \subseteq \ell^\infty(\state(\A))$.
	But this is because the latter is $\sigma$-closed and contains $\A$.
	Therefore it contains $\A^{\infty}$ as well.
\end{proof}

In fact, it seems plausible that this can be used to give another characterization of the Pedersen--Baire envelope.

\begin{conj}
	The vertical right arrow in~\eqref{eq:pb_state_diag} identifies $\A^{\infty}$ with those bounded affine functions on $\state(\A)$ that are Baire measurable.
\end{conj}

We conclude this section with an interesting result motivated by our forthcoming study of duality in the concrete setting (see \cref{sec:concrete_duality}). 

\begin{thm}[{\cite[Corollary~4.5.13]{pedersen2018automgroups}}]
	\label{thm:pb_pure}
    Every Pedersen--Baire envelope is purely $\sigma$-representable (i.e., separated by its pure $\sigma$-states).
\end{thm}

We provide a new proof that we believe gives better insight into how Pedersen--Baire envelopes are well-behaved from a measure-theoretic perspective.

\begin{proof}
	Let $\A$ be a \cstar{} and $\A^{\infty}$ its Pedersen--Baire envelope.
	Due to the construction of $\A^\infty$ in the universal representation, we know that $\A^\infty$ is separated by its $\sigma$-states.

	The main task is to cut down to pure $\sigma$-states from there.
	As before, we consider the state space $\state(\A)$ as a compact Hausdorff space with the weak-* topology, and we also write $P \subseteq \state(\A)$ for the subset of pure $\sigma$-states.
	Then by the Choquet--Bishop--de Leeuw theorem~\cite[Section 4]{phelps01choquet}, for every $\phi \in \state(\A)$ there is a probability measure $\mu$ on $\state(\A)$, which vanishes on every Baire subset of $\state(\A)$ which is disjoint from $P$, such that 
    	\begin{equation}
		\label{eq:state_decomposition}
		\ev_a (\phi)= \int_{\psi \in \state(\A)} \ev_a (\psi) \, \mathrm{d} \mu(\psi)
    	\end{equation}
	for all $a \in \A$.
	
	We now claim that~\eqref{eq:state_decomposition} holds for all $a \in \A^\infty$.
	To see this, let $\B \subseteq \A^\infty$ be the set of $a \in \A^\infty$ for which the evaluation map $\ev_a$ is measurable with respect to the Baire $\sigma$-algebra on $\state(\A)$ and for which~\eqref{eq:state_decomposition} holds.\footnote{In fact, $a \mapsto \ev_a$ is the natural description of both vertical arrows in \eqref{eq:pb_state_diag}.}
	Since $\mu$ is finite and $\ev_a$ is bounded for all $a \in \A^{\infty}$, it is clear that $\B$ is a linear subspace. 
	If we manage to prove that $\B$ is closed under countable monotone suprema, so that it is a $\sigma$-subspace, then our claim follows by a $\pi$-$\lambda$-type argument: $\A \subseteq \B$ means that $\B$ contains the $\sigma$-closure of $\A$, which is $\A^{\infty}$, and therefore $\B = \A^{\infty}$. 
   
   	So let us take a sequence $(a_n)$ in $\B$ with $(a_n) \nearrow a$.
	Then for every $\phi\in\state(\A)$, we have $0\le \ev_{a_n}(\phi) \le \ev_{a_{n+1}}(\phi) \le \lambda$ for all $n$ with some fixed $\lambda > 0$, and therefore the monotone convergence theorem shows that
   \[
	   \sup_n \int_{\psi \in \state(\A)} \ev_{a_n} (\psi) \, \mathrm{d}\mu(\psi) = \int_{\psi \in \state(\A)} \sup_{n} \, \left( \ev_{a_n}(\psi) \right) \, \mathrm{d}\mu(\psi) = \int_{\psi \in \state(\A)} \ev_{a} (\psi) \, \mathrm{d}\mu(\psi),
   \]
   where the second equation uses that every $\psi$ is a $\sigma$-state.
   Since $a_n \in \B$, the left-hand side is equal to $\sup_n \ev_{a_n} (\phi) = \ev_a (\phi)$, and moreover this equality shows that $\ev_a \colon \state(\A) \to \mathbb{R}$ is indeed measurable. 
   Hence $\B$ is a $\sigma$-subspace, and so~\eqref{eq:state_decomposition} holds for all $a \in \A^{\infty}$.

    For nonzero $a\ge 0$, let now ${\phi}$ be a $\sigma$-state with ${\phi}(a) >0$, which must exist by the separation noted at the beginning.
    Using the integral decomposition~\eqref{eq:state_decomposition}, we have
    \[
	    0<\ev_a ({\phi}) = \int_{\psi \in \state(\A)} \ev_a(\psi) \, \mathrm{d}\mu(\psi),
    \]
    where $\mu$ is a Baire probability measure on $\state(\A)$ vanishing on every Baire set disjoint from $P$.
    By the positivity, there must exist a Baire set $E \subseteq \state(A)$ with $\mu(E)>0$ and $\psi(a) > 0$ for every $\psi\in E$. 
	But then $E$ cannot be disjoint from $P$, and therefore $\psi(a)>0$ for some pure $\sigma$-state $\psi$.
\end{proof}

\begin{rem}
	From the proof of \cref{thm:pb_pure}, we can also infer that for every positive element $a$ in a Pedersen--Baire envelope $\A^{\infty}$,
	\[
		\norm{a} \coloneqq \sup \lbrace \psi(a) \mid \psi\colon \A \to \mathbb{C} \text{ is a pure }\sigma\text{-state}\rbrace.
	\]
	The fact that this is true for $\sigma$-states was already discussed in \cref{rem:norm_sup}.
	Now let $\eps>0$ and consider a $\sigma$-state $\phi$ such that $\phi(a)\ge \norm{a}-\eps$. 
	Using the Baire probability measure $\mu$ associated to $\phi$, from the proof above we obtain a set of positive measure $E$ where the integrand is $\ge \norm{a}-\eps$ pointwise, and such that $E$ nontrivially intersects the set of pure $\sigma$-states.
	Therefore there exists a pure $\sigma$-state $\psi$ such that $\psi (a) \ge \norm{a}- \eps$. 
\end{rem}

\subsection{Borel functional calculus}
\label{sec:functional_calculus}

For $W^*$-algebras, a crucial tool is their Borel functional calculus, which allows for the application of any measurable function $\mathbb{C} \to \mathbb{C}$ to any normal element in the algebra.
As we show here, this generalizes to \cstars{}s.
We derive this from standard continuous functional calculus together with the Pedersen--Baire envelope construction. 

We denote by $\spectrum(a)$ the spectrum of an element $a$ in a \cstar.

\begin{thm}[Functional calculus for \cstars s]\label{thm:fun_calculus}
    	Let $\A$ be a \cstars{} and $a \in \A$ a normal element. Then the there exists a unique \starshom
	\[
		\Phi_a \colon \Linf (\spectrum(a)) \to \A
	\]
	with $\Phi_a(\id_{\spectrum(a)})=a$. 
\end{thm}

\begin{proof}
    By continuous functional calculus, we have a unique $*$-homomorphism
    \[
	    \Psi_a \colon C(\spectrum(a)) \to \A
    \]
    with $\Psi_a(\id_{\spectrum(a)}) = a$.
    By \cref{cor:baire_eq}, $\Linf(\spectrum(a))$ is the Pedersen-Baire envelope of $C(\spectrum(a))$.
    Therefore, $\Psi_a$ uniquely extends to a \starshom{} $\Linf(\spectrum(a)) \to \A$ by~\cref{thm:pb_univ}.
\end{proof}

Given $f\in \Linf (\spectrum(a))$, it is customary to write $f(a)$ as shorthand for $\Phi_a(f)$. 

\begin{rem}[Spectral theorem] \label{rem:spectral_thm}
	In particular, we also obtain the spectral theorem: given a normal operator $a \in \bhilb$, the \starshom{} obtained by Borel functional calculus is associated to a unique PVM $\mu$ as discussed in \cref{ex:pvms}. 
	In terms of integral notation, this means that we have
	\[
		f(a) = \int_{\spectrum(a)} f(\lambda) \, \mathrm{d} \mu (\lambda),
	\]
	and the spectral theorem follows by setting $f= \id_{\spectrum(a)}$.
\end{rem}
\begin{rem}
	In a commutative \cstar{} $\A \coloneqq C(X)$, continuous functional calculus is simply given by composition: $f(a) = f \circ a$.
	If the compact Hausdorff space $X$ is such that $C(X)$ is a commutative \cstars, which by \cref{fact:rickart_gelfand} equivalently means that $X$ is a Rickart space, then the same is usually \emph{not} true for merely measurable $f$.

	Take for example $X = \beta\mathbb{N}$, which satisfies $C(\beta\mathbb{N}) \cong \ell^{\infty}(\mathbb{N})$, and consider the one-point compactification $\overline{\mathbb{N}} = \mathbb{N} \cup \{\infty\}$.
	The subalgebra $C(\overline{\mathbb{N}}) \subseteq C(\beta \mathbb{N})$ consists of all continuous functions $\beta \mathbb{N} \to \mathbb{C}$ which factor across $\beta \mathbb{N} \twoheadrightarrow \overline{\mathbb{N}}$. 
	If the formula $f(a) = f \circ a$ held for all measurable $f$, then this subalgebra would have to be closed under the functional calculus.
	However, under the isomorphism  $C(\beta \mathbb{N})\cong \ell^\infty(\mathbb{N})$, the subalgebra $C(\overline{\mathbb{N}})$ corresponds to the space of convergent sequences $c(\mathbb{N})$.
	Since applying the functional calculus to a convergent sequence can result in a non-convergent sequence, we reach a contradiction.
\end{rem}

We recall that a \emph{projection} in a \cstar{} is an element $p$ such that $p=p^*=p^2$. We say that an element $a$ is \newterm{simple} if it can be written as a finite sum $a = \sum_{i=1}^n \lambda_i p_i$ for some scalars $\lambda_i \in \mathbb{C}$ and projections $p_i$. 
We do not require the $p_i$ to be pairwise orthogonal or even to commute (although the following statement would still be true with this restriction).

\begin{cor}\label{cor:simple_dense}
	In every \cstars, we have:
	\begin{enumerate}
		\item Every self-adjoint is the supremum of a sequence of simple self-adjoint elements.
		\item The set of simple elements is norm-dense.
	\end{enumerate}
\end{cor}

In the commutative case, where the set of simple elements is also a subalgebra, this statement has already been noted in~\cite[p.~37]{saito15monotone}.

\begin{proof}
    For any self-adjoint element $a$, the function $\id_{\spectrum(a)}$ can be approximated from below by simple functions, which correspond to simple elements $s_n \in \Linf (\spectrum(a))$. Therefore, $\id_{\spectrum(a)} = \sup_n s_n$. Using functional calculus, this translates to $a= \sup_n s_n(a)$, which proves the first claim.
    By choosing the approximation by simple functions to be uniform, we can moreover achieve that this supremum is norm-convergent.
    Since any element in a \cstar{} is a sum of two self-adjoints, the second claim follows in general.
\end{proof}

It is a standard fact that every $*$-homomorphism preserves functional calculus.
We now show that this extends to measurable functional calculus, and it is easy to see based on \Cref{thm:fun_calculus}.

\begin{lem}
	\label{lem:fun_calculus_starshom}
	If $\phi \colon \A \to \B$ is a \starshom{} between \cstars{}s and $a \in \A$ is normal, then for any $f \in \Linf(\spectrum(a))$ we have
	\[
		\phi(f(a)) = f(\phi(a)).
	\]
\end{lem}

Here, we have slightly abused notation by writing $f$ instead of $f|_{\spectrum(\phi(a))}$ on the right-hand side.

\begin{proof}
	Considered as functions of $f$, both sides are \starshom{}s $\Linf(\spectrum(a)) \to \B$ sending $\id_{\spectrum(a)}$ to $\phi(a)$, and hence they coincide by the uniqueness part of \Cref{thm:fun_calculus}.
\end{proof}

\subsection{\texorpdfstring{$\sigma$}{σ}-completions}\label{sec:sigma_compl}

In the context of measurable spaces, it is often of interest to consider the $\sigma$-algebra generated by an algebra of sets, and this motivated the $\sigma$-completions of Boolean algebras considered in \cref{sec:bool_sigma_completions}.
For a \cstar{} $\A$, we now study analogous notions of $\sigma$-completion, first introduced by Wright in \cite{wright74minimal,wright76regular}.

\begin{defn}
	Let $\A$ be a \cstar. A \newterm{$\sigma$-completion} of $\A$ is a pair $(\B,j)$, where
	\begin{enumerate}
		\item $\B$ is a \cstars;
		\item $j\colon \A \to \B$ is an injective \starshom;
		\item $\B$ is $\sigma$-generated by $j(\A)$.
	\end{enumerate}
\end{defn}

The set of $\sigma$-completions of $\A$ carries a canonical preorder $\preceq$: we say that a $\sigma$-completion $(\B,j)$ is \emph{below} a $\sigma$-completion $(\C,k)$ if and only if there exists a \starshom {} $\phi \colon \B \to \C$ such that 
\[
\begin{tikzcd}
    \A \ar[r,"j"]\ar[rd,"k" below left] & \B\ar[d, dashed, "\phi"]\\
    & \C
\end{tikzcd}
\] 
The assumption that $\B$ is $\sigma$-generated by $j(\A)$ guarantees that such $\phi$ is unique if it exists.
Also it is clear that $\preceq$ is reflexive and transitive.

As for Boolean algebras, there are two special $\sigma$-completions of a \cstar{}~\cite{wright74minimal,wright76regular}.

\begin{defn} Let $\A$ be a \cstar. 
    \begin{enumerate}
        \item The \newterm{universal $\sigma$-completion} of $\A$ is a $\sigma$-completion $(\B,j)$ which is minimal with respect to $\preceq$.
        \item  The \newterm{regular $\sigma$-completion} of $\A$ is a $\sigma$-completion $(\B,j)$ such that for every $b \in \B_{\sa}$,
        \begin{equation}
		\label{eq:regular_completion}
	        b = \sup \, \lbrace j(a) \mid j(a) \le b\rbrace.
        \end{equation}
    \end{enumerate}
\end{defn}

The use of ``the'' in the previous definition is justified by the fact that both the universal and the regular $\sigma$-completions are unique up to unique isomorphism under $A$. 
For the universal $\sigma$-completion, this follows from the universal property (\cref{cor:univ_minimal_2}), while for the regular $\sigma$-completion the proof is rather involved~\cite[Theorem 2.4]{wright76regular}. 

\begin{nota}
	Given any \cstar{} $\A$, we write $\univ{\A}$ and $\reg{\A}$ to denote the universal and the regular $\sigma$-completions, respectively.
\end{nota}

Moreover, these $\sigma$-completions exist for every \cstar{} $\A$ (see~\cite[Theorem~2.5]{wright74minimal} and~\cite[Theorem~2.3]{wright76regular}). 
We refrain from an explicit description of the regular $\sigma$-completion outside of the commutative setting, since in the present paper we are not interested in specific details of such a construction (recall that, already for \bsigma{}s, regular $\sigma$-completions are not functorial, see \cref{rem:nonfunctoriality_regular}).

For the universal $\sigma$-completion, we first provide an existence proof which simplifies Wright's, and we then prove a universal property which motivates the naming.
This will turn out to be stronger than its analogue given in the case of \bsigma{}s, since it also holds for the probabilistic situation (\cref{cor:univ_minimal_2}).

Since kernels of \scpu\ maps are not necessarily closed under multiplication, a direct analogue of \cref{lem:univ_sigma_completion_bool} in the \bsigma{} setting would be too restrictive.
More explicitly, we aim to describe the $\sigma$-ideal associated to the universal $\sigma$-completion $\univ{A}$ simply as a $\sigma$-subspace $\mathscr{I} \subseteq A^\infty$ generated by some subspace $\mathscr{I}_0$, without explicitly requiring closure under multiplication.
From this perspective, our current notion of $\sigma$-closure appears to be unsuitable due to the subtleties already highlighted before \cref{lem:sigmaclosed_is_closed}.
For this reason, we will consider a new notion of $\sigma$-closure, which will be relevant only in this subsection.
\par\vspace{1ex}

\begin{samepage}
	\noindent\rule{\linewidth}{1pt}\par\vspace{-3ex}
	\subsubsection*{Interlude on strong $\sigma$-closures}	
\end{samepage}
We begin with a preliminary definition.

\begin{defn}[{\cite[Definition~2.1.17]{saito15monotone}}]\label{def:orderlimit}
	Let $\A$ be a \cstar{}.
	\begin{enumerate}
		\item\label{it:orderconv_sa}
			A bounded sequence $(a_n)$ in $\A_{\sa}$ is \newterm{order convergent} if there exists a decomposition $a_n = b_n - c_n$ with bounded increasing sequences $(b_n)$ and $(c_n)$ both having a supremum.
			In this case, we define its \newterm{order limit} as
			\[
				\ordlim_n a_n \coloneqq \sup_n b_n -\sup_n c_n.
			\]
		\item A bounded sequence $(a_n)$ in $\A$ is \newterm{order convergent} if its sequences of real $(\Re(a_n))$ and imaginary parts $(\Im(a_n))$ are order convergent, and in this case we define the \newterm{order limit} as
			\[
				\ordlim_n a_n \coloneqq \ordlim_n \Re(a_n) + \mathsf{i} \, \ordlim_n \Im(a_n),
			\]
			where $\mathsf{i}$ is the imaginary unit.
	\end{enumerate}
\end{defn}
We refer to~\cite[Lemma~2.1.16]{saito15monotone} for the proof that the order limit in~\ref{it:orderconv_sa} does not depend on the choice of decomposition $a_n = b_n - c_n$.

\begin{rem}\label{rem:sigma_normal_orderlimit}
	By linearity, any $\sigma$-normal map $\phi$ also preserves order limits, in the sense that $\phi (\ordlim_n a_n )= \ordlim_n \phi(a_n)$ for any order convergent sequence $(a_n)$. 
\end{rem}

To get a conceptual understanding of this notion, let us note the following.

\begin{prop}\label{prop:order_limit_variation}
	Let $\A$ be a \cstars{} and let $(a_n)$ a self-adjoint sequence. 
	If the \emph{total variation}
	\begin{equation}
		\sum_n \, \abs{a_{n+1} - a_n}\label{eq:total_variation}
	\end{equation}
	is bounded by an element of $\A$, then $(a_n)$ has an order limit.
	\end{prop}
	\begin{proof}
		We consider $d_{n+1} \coloneqq a_{n+1} - a_{n}$ (for $n=1$, $d_1\coloneqq a_1$), and set
		\[
		b_{n} \coloneqq \sum_{k=1}^n (d_k)^+, \qquad c_{n} \coloneqq \sum_{k=1}^n (d_k)^- .
		\]
		The sequences $(b_n)$ and $(c_n)$ are monotone increasing, and bounded by assumption. Moreover, $b_{n}-c_{n}=a_n$, so $(a_n)$ indeed has an order limit.
	\end{proof}
	
	\begin{prop}
		\label{prop:order_limit_commutative}
		In a commutative \cstars{} $\A = C(X)$, a self-adjoint sequence $(a_n)$ has an order limit if and only if the {total variation} \eqref{eq:total_variation} is bounded by an element of $\A$.
		In this case, the order limit is pointwise except on a meager set.
	\end{prop}
	
	\begin{proof}
		%
	By \cref{prop:order_limit_variation}, it suffices to show that whenever a sequence has an order limit, then the total variation is bounded.
		If $a_n = b_n - c_n$ is a decomposition with bounded increasing sequences $(b_n)$ and $(c_n)$, then we have
		\[
			|a_{n+1} - a_n| = |(b_{n+1} - b_n) - (c_{n+1} - c_n)| \le (b_{n+1} - b_n) + (c_{n+1} - c_n),
		\]
		where both brackets on the right are nonnegative.\footnote{This triangle inequality is where commutativity enters.}
		We therefore obtain
		\[
			\sum_{n=1}^m |a_{n+1} - a_n| \le \sum_{n=1}^m (b_{n+1} - b_n) + \sum_{n=1}^m (c_{n+1} - c_n) = b_{m+1} - b_1 + c_{m+1} - c_1,
		\]
		which is upper bounded by the assumption that $(b_n)$ and $(c_n)$ are bounded sequences.
	
	
		The final statement follows from \Cref{lem:sup_chaus}.
	\end{proof}

The most important feature of order limits is that they commute with multiplication, a statement which would not even make sense for suprema or infima.
More generally, they have the following properties~\cite[Section 2.1]{saito15monotone}.\footnote{In particular, see Lemma 2.1.19 and the brief discussion after its proof, Lemma 2.1.21, Proposition 2.1.22 and Corollary 2.1.23.}

\begin{fact}\label{fact:orderlimit}
	In a \cstar{} $\A$, the following assertions hold.
	\begin{enumerate}
		\item\label{it:commutingprops} Order limits commute with addition, multiplication and involution: if $(a_n)$ and $(b_n)$ are order convergent and $c,d \in \A$ are arbitrary, then 
		\begin{align*}
			\ordlim_n \left( a_n + b_n \right) & = \ordlim_n a_n + \ordlim_n b_n, &  
			\ordlim_n c a_n d & = c \left( \ordlim_n a_n \right) d, &
			\ordlim_n a_n^* & = \left( \ordlim_n a_n \right)^*.
		\end{align*}
		\item \label{it:normlimit_order}
			Every sequence $(a_n)$ which converges in norm to some $a \in \A$ has an order convergent subsequence $(a_{n(j)})_{j}$ such that $\ordlim_j a_{n(j)} = a$. 
	\end{enumerate}
\end{fact}

\begin{defn}
	Let $\A$ be a \cstar{}.
	\begin{enumerate}
		\item A \newterm{strong $\sigma$-subspace} $V\subseteq \A$ is a $*$-subspace that is \newterm{strongly $\sigma$-closed}: whenever $(a_n)\subseteq V$ is order convergent, then $\ordlim_n a_n \in V$.
		\item The \newterm{strong $\sigma$-closure} of a $*$-subspace $W$ is the smallest strong $\sigma$-subspace containing $W$.
	\end{enumerate} 
\end{defn}
\begin{rem}\label{rem:strongsigma_closed}
	\begin{enumerate}
		\item\label{it:strongsigma_closed_closed} By \cref{fact:orderlimit}\ref{it:normlimit_order}, every strong $\sigma$-subspace is closed.
		\item\label{it:strongsigma_closed_kernel} The kernel of a $\sigma$-normal map is strongly $\sigma$-closed because of \cref{rem:sigma_normal_orderlimit}. 
		In particular, $\sigma$-ideals are strongly $\sigma$-closed because they are the kernel of their quotient map.
	\end{enumerate}
\end{rem}
It remains unclear whether $\sigma$-subspaces are strongly $\sigma$-closed in general. 
For example, given a $\sigma$-injection $\A\hookrightarrow \B$ of \cstars{}s, we do not know whether the image of $\A$ is strongly $\sigma$-closed in $\B$.  
This uncertainty arises because $b_n$ and $c_n$ as in \cref{def:orderlimit} are arbitrary, meaning they could a priori belong to $\B\setminus \A$ even if $a_n = b_n - c_n \in \A$. 
This resembles in spirit the open problem discussed in \cite[Section 5.3.1]{saito15monotone}, as both involve different notions of $\sigma$-closure.
\par\vspace{-1.5ex}
\noindent\rule{\linewidth}{1pt}
\par\vspace{1ex}
We are now ready to discuss the universal $\sigma$-completion.
\begin{nota}
	Let $\A^{\infty}$ be the Pedersen--Baire envelope of a \cstar{} $\A$. We denote by $\mathscr{I}_{0}$ the $*$-subspace of $\A^{\infty}$ consisting of all elements $a \in \A^{\infty}$ for which there exists a sequence $(a_n)$ in $\A$ such that $\ordlim_n a_n = a$ in $\A^{\infty}$ and $\ordlim_n a_n = 0$ in $\A$.	
\end{nota}

\begin{lem}\label{lem:minimal_sigma_sub}
	The strong $\sigma$-closure of $\mathscr{I}_0$ in $\A^{\infty}$, denoted by $\mathscr{I}$, is a $\sigma$-ideal in $\A^\infty$.
\end{lem}

In the proof below, we will make abundant use of the properties 
of order limits highlighted in \cref{fact:orderlimit}.
\begin{proof}
	By definition and \cref{rem:strongsigma_closed}\ref{it:strongsigma_closed_closed} (or \cref{lem:sigmaclosed_is_closed}), it only needs to be shown that $\mathscr{I}$ is closed under multiplication with elements of $\A^\infty$.
	To start, we show that $\mathscr{I}_0$ is closed under multiplication by elements of $\A$.
	If $a \in \mathscr{I}_0$, as witnessed by a sequence $(a_n)$ in $\A$ such that $\ordlim_n a_n = a$ in $\A^{\infty}$ and $\ordlim_n a_n = 0$ in $\A$, then for all $b \in \A$, we have
	\[
		\ordlim_n b a_n = b \, \ordlim_n a_n = 0
	\]
	in $\A$, and hence $ba = b \ordlim_n a_n = \ordlim_n b a_n$ belongs to $\mathscr{I}_0$.
	The situation is analogous for right multiplication.
	The result now follows by considering two $\pi$-$\lambda$ arguments:
	\begin{itemize}
		\item The $*$-subspace $\lbrace a \in \mathscr{I} \mid ab,ba\in \mathscr{I} \; \forall b \in \A \rbrace$ is strongly $\sigma$-closed and contains $\mathscr{I}_0$, so it coincides with $\mathscr{I}$.
		\item The $*$-subspace $\lbrace b \in \A^{\infty} \mid ab,ba \in \mathscr{I} \; \forall a \in \mathscr{I}\rbrace$ is strongly $\sigma$-closed and contains $\A$ by the previous item, so it coincides with $\A^{\infty}$.
	\end{itemize} 
	Therefore, the statement holds.
\end{proof}

Let us consider this ideal in a family of examples before tackling the universal $\sigma$-completion itself in general.

\begin{prop}\label{prop:separable_unique_completion}
	Let $X$ be a compact Hausdorff space and $\M$ the $\sigma$-ideal of functions in $\Linf(X)$ whose support is meager (\cref{ex:meager_ideal}).
	Then:
	\begin{enumerate}
		\item We have $\mathscr{I} \subseteq \M$.
		\item If $X$ is separable, then $\mathscr{I} = \M$.
	\end{enumerate}
\end{prop}

\begin{proof}
	\begin{enumerate}
		\item Since $\M$ is a $\sigma$-ideal, it is enough to show that $\mathscr{I}_0 \subseteq \M$ (\cref{rem:strongsigma_closed}\ref{it:strongsigma_closed_kernel}).
			Thus consider a sequence $(f_n)$ in $C(X)$ such that $\ordlim_n f_n = 0$ in $C(X)$ and $\ordlim_n f_n = f$ in $\Linf(X)$.
			Then the claim that $f$ has meager support follows by \Cref{prop:order_limit_commutative}.
		\item It remains to prove that $\M \subseteq \mathscr{I}$.
			For example by arguing is in the proof of \cref{cor:simple_dense}, it is easy to see that every $f \in \M$ can be written as a countable supremum of linear combinations of indicator functions $\chi_C$, with each $C$ a measurable meager set.
			Therefore, it suffices to prove that $\chi_C \in \mathscr{I}$ for every measurable meager set $C$.
	
			We first consider the case of $C$ being a closed nowhere dense set. 
			By assumption, there exists a dense sequence $(x_n)_{n \in \mathbb{N}}$ in $X$. 
			Since $\{x_n \: : \: n \in \mathbb{N}\} \cap (X \setminus C)$ is also dense, we can assume without loss of generality that $\{x_n \: : \: n \in \mathbb{N} \} \cap C = \emptyset$.
			By Urysohn's lemma, there exist continuous functions $g_n\colon X \to [0,1]$ such that $g_n(x_n)=0$ and $g_n(C)=\{1\}$ for all $n$.
			We now define $f_n \coloneqq g_1 \cdots g_n$. 
			By construction, $f_n \ge f_n\cdot g_{n+1} = f_{n+1}$ and $f_n(C) = \{1\}$ for all $n$.
			Moreover, $\inf_n f_n =0$ in $C(X)$ since the dense sequence $(x_n)_{n \in \mathbb{N}}$ is sent to zero by the pointwise infimum.
			If we denote the pointwise infimum by $f$, then we have $f \in \mathscr{I}_0$ by definition since the infimum in the Pedersen--Baire envelope $\Linf(X)$ is pointwise.
			Since in addition $0 \le \chi_C \le f$, we conclude $\chi_C \in \mathscr{I}$ because the positive cone of a closed ideal is hereditary~\cite[Theorem 1.5.2]{pedersen2018automgroups}.
	
			If $C$ is merely a measurable nowhere dense set, then its closure $\overline{C}$ is still nowhere dense, and we conclude again by the hereditary property: $0 \le \chi_C \le \chi_{\overline{C}} \in \mathscr{I}$ implies $\chi_C \in \mathscr{I}$.
			As $\mathscr{I}$ is $\sigma$-closed, the indicator function of any measurable meager set also belongs to $\mathscr{I}$, and therefore the claim $\M \subseteq \mathscr{I}$ follows.
			\qedhere
	\end{enumerate}
\end{proof}

\begin{ex}
	One may wonder whether $\mathscr{I} = \mathscr{I}_0$ in general.
	For an example where this is not the case, consider $\A \coloneqq C([0,1])$ with Pedersen--Baire envelope $\A^\infty = \Linf([0,1])$.
	Then the indicator function $f \coloneqq \chi_{\mathbb{Q} \cap [0,1]}$ belongs to $\mathscr{I}$ but not to $\mathscr{I}_0$.
	Indeed the former is clear by \cref{prop:separable_unique_completion}.
	For the latter, it is enough to note that $f$ is not of Baire class one (not the pointwise limit of a sequence of continuous functions).
	This is for example because the set of points of continuity of a Baire class one function is comeager~\cite[Theorem~24.14]{kechris95classical}.
\end{ex}

While our next auxiliary result could be deduced from \cite[Lemma 5.4.1]{saito15monotone} and \cite[Theorem 3.3]{wright72measures},\footnote{See also \cite[proof of Lemma 2.4]{wright74minimal} for an additional insight.} the following proof offers a conceptually clean argument.

\begin{lem}
	\label{lem:AI_disjoint}
	We have $\A \cap \mathscr{I} = \lbrace 0 \rbrace$.
\end{lem}

\begin{proof}
	We consider the embedding into $\Linf(\state(\A))$ from \cref{prop:pb_state}.
	Then since countable infima in $C(\state(\A))$ are pointwise outside of a meager set by \cref{lem:sup_chaus}, every $a \in \mathscr{I}_0$ is supported on a meager subset of $\state(\A)$.
	This implies that $\mathscr{I}_0$ is contained in the $\sigma$-ideal of Baire measurable functions with meager support, and therefore so is $\mathscr{I}$.
	The claim now follows by the Baire category theorem.
\end{proof}

\begin{prop}\label{prop:minimal_description}
	With notation as above, $\A^{\infty}/ \mathscr{I}$ is the universal $\sigma$-completion of $\A$.
\end{prop}
\begin{proof}
	By \cref{lem:AI_disjoint}, the $\ast$-homomorphism $j \colon \A \to \A^{\infty}/\mathscr{I}$ is injective.
	It is a $\sigma$-homomorphism by \cref{rem:sigmanormal_inf0} because $\mathscr{I}_0 \subseteq \mathscr{I}$. 
	Its image is $\sigma$-generating since this already holds before taking the quotient.
	It thus remains to check the universality of the $\sigma$-completion.
	If $k \colon \A \hookrightarrow \B$ is any $\sigma$-completion of $\A$, then we obtain an induced \starshom{} $k^{\infty}\colon \A^\infty \to \B$ by the universal property of $\A^\infty$.
	Now the assumption that $k$ is a \starshom{} itself shows that $\mathscr{I}_0 \subseteq \ker(k^{\infty})$, hence $\mathscr{I} \subseteq \ker(k^{\infty})$ by \cref{rem:strongsigma_closed}\ref{it:strongsigma_closed_kernel}.
\end{proof}

Using \Cref{lem:minimal_sigma_sub}, we can actually prove a stronger universal property of the universal $\sigma$-com\-ple\-tion.

\begin{thm}[Universal property of the universal $\sigma$-completion]\label{cor:univ_minimal_2} 
    Consider a \cstar{} $\A$ and let $j \colon \A \to \univ{\A}$ be its universal $\sigma$-com\-ple\-tion.
	Then for every \cstars{} $\B$, every \scpu{} map $\phi \colon \A \to \B$ extends uniquely to a \scpu{} map $\tilde{\phi} \colon \univ{\A} \to \B$, i.e. $\tilde{\phi} j = \phi$.
\end{thm}

Thus we obtain the existence of an adjunction
\begin{equation}\label{eq:univ_adjunction}
    \begin{tikzcd}
        \Calgscpu \ar[rr,phantom, "\bot"]\ar[rr,bend left=15,"\univ{(-)}"] && \SigmaCcpu \ar[ll,bend left=15]
    \end{tikzcd}
\end{equation}
where the lower arrow is the forgetful functor, while its left adjoint $\univ{(-)}$ sends $\A$ to its universal $\sigma$-completion $\univ{\A}$. 
Since the unit components $A \hookrightarrow \univ{A}$ are monomorphisms, it follows that this left adjoint is faithful.
\begin{proof}
	This is by an argument similar to the proof of \Cref{prop:minimal_description}.
	We consider the \scpu{} map $\phi^{\infty} \colon \A^{\infty} \to \B$ given by \cref{cor:pb_univ2}. 
	Then the kernel of $\phi^\infty$ contains $\mathscr{I}_0$, since if a sequence $(a_n)$ in $\A$ satisfies $\ordlim_n a_n =0$, then we also have 
    	\[
		\phi^\infty( \ordlim_n a_n ) = \ordlim_n \phi^\infty(a_n) = \ordlim_n \phi(a_n) = \phi(\ordlim_n a_n) = 0,
	\]
	where we have used the assumption that $\phi$ is \scpu.
    	Since the kernel of a \scpu{} map is a strong $\sigma$-subspace (\cref{rem:strongsigma_closed}\ref{it:strongsigma_closed_kernel}), this means that $\mathscr{I}$ is contained in the kernel as well,
	and hence we obtain the desired factorization across $\univ{\A} = \A^{\infty}/\mathscr{I}$.
\end{proof}

\begin{rem}
	\cref{cor:univ_minimal_2} has the flavor of a noncommutative version of Carath\'{e}odory's extension theorem.
	However, it does not exactly recover that result in the commutative case.
	To see why, recall that $\sigma$-algebras generated by algebras of sets are not necessarily $\sigma$-completions (\cref{rem:alg_sets_not_sigma}).
	The upcoming \cref{prop:universal_completions} implies that such counterexamples can be transferred to the C*-setting as well.
\end{rem}

\begin{rem}\label{rem:univ_prop_hom}
	Similar to the Boolean algebra case from \cref{prop:sigma_completion}, 
	it is clear that the universal property of the universal $\sigma$-completion restricts to \starshom{}s: every \starshom{} $\A \to \B$ extends uniquely to a \starshom{} $\univ{\A} \to \B$.
	In other words, the adjunction in \eqref{eq:univ_adjunction} restricts to the following:
	\[
    \begin{tikzcd}
        \Calgshom \ar[rr,phantom, "\bot"]\ar[rr,bend left=15,"\univ{(-)}"] && \SigmaC \ar[ll,bend left=15]
    \end{tikzcd}
\]
\end{rem}

\begin{rem}
	\cref{cor:univ_minimal_2} holds also for $\sigma$-normal positive maps, using~\cite[Proposition 5.4.7 and Exercise 5.4.8]{saito15monotone} instead of \cref{cor:pb_univ2} for the proof.
\end{rem}

In the following example, we describe the regular $\sigma$-completion in the commutative case.

\begin{ex}\label{ex:commutative_regular}
	For $X$ a compact Hausdorff space, we again consider the $\sigma$-ideal $\M \subseteq \Linf(X)$ of Baire measurable functions with meager support.
	Then the regular $\sigma$-completion of $C(X)$ is
	\[
		\reg{C(X)} = \Linf(X) / \M,
	\]
	where the $\sigma$-injection is given by the composition $C(X) \hookrightarrow \Linf(X) \twoheadrightarrow \Linf(X)/\M$, which is indeed injective by the Baire category theorem.
	The regularity property~\eqref{eq:regular_completion} holds by \cite[Corollary~4.2.11]{saito15monotone}.
\end{ex}

Thus by \cref{prop:separable_unique_completion}, separability implies that there is only one $\sigma$-completion. This is not the case in general, as we will see in \cref{rem:univ_reg_differ2}.

\begin{prop}\label{prop:separable_unique_completion2}
	Let $X$ be a separable compact Hausdorff space. Then $C(X)$ has a unique (up to isomorphism) $\sigma$-completion.
\end{prop}
\begin{proof}
	By \cref{prop:separable_unique_completion}, the universal and the regular $\sigma$-completions coincide. 
	Since the universal $\sigma$-completion is $\preceq$-minimal and the regular is $\preceq$-maximal \cite[Corollary 3.4]{wright76regular}, all the other $\sigma$-completions are necessarily isomorphic to this one.
\end{proof}

Interestingly, under these conditions the $\sigma$-completion is monotone complete, see \cite[Theorem 4.2.9 and Corollary 4.2.24]{saito15monotone}.

\begin{cor}
	Let $X$ be a second countable compact Hausdorff space with no isolated points. 
	Then $C(X)$ has no $\sigma$-state.
\end{cor}

\begin{proof}
	By \Cref{ex:commutative_regular,prop:separable_unique_completion2}, the unique $\sigma$-completion of $C(X)$ is $\Linf(X) / \M$.
	Since every $\sigma$-state would have to factor across this $\sigma$-completion by \cref{cor:univ_minimal_2}, the claim follows by \cref{prop:meager_nosigmastate}.
\end{proof}

\begin{prop}
	Let $\mathcal{H}$ be a separable infinite-dimensional Hilbert space.
	Then the only $\sigma$-com\-ple\-tion of $\khilb_1$, the algebra of compact operators plus multiples of the identity, is $\bhilb$.
\end{prop}

\begin{proof}

	The Pedersen--Baire envelope is $\khilb_1^\infty = \bhilb \oplus \mathbb{C}$ by \cref{ex:hilb_pb}.
	This cannot be a $\sigma$-completion because the inclusion $\khilb_1 \hookrightarrow \bhilb \oplus \mathbb{C}$ is not a $\sigma$-homomorphisms. For example, let $(p_n)$ be a monotone sequence of projections of finite rank whose supremum in $\bhilb$ is $1$.
	Then also $\sup_n p_n = 1$ in $\khilb_1$, but their images in $\bhilb \oplus \mathbb{C}$ satisfy $\sup_n (p_n,0) = (1,0) \neq (1,1)$.

	By \cref{rem:sigmageneral}, any $\sigma$-completion of $\khilb_1$ is determined by a $\sigma$-ideal of $\bhilb \oplus \mathbb{C}$ having trivial intersection with $\khilb_1$.
	Since the only nontrivial closed two-sided ideal of $\bhilb$ is $\khilb$, the only closed two-sided ideal with such trivial intersection is $0 \oplus \mathbb{C}$.\footnote{With the complete list of closed two-sided ideals being $0 \oplus 0$, $0 \oplus \mathbb{C}$, $\khilb \oplus 0$, $\khilb \oplus \mathbb{C}$, $\bhilb \oplus 0$ and $\bhilb \oplus \mathbb{C}$.}
	As $\sigma$-completions exist for every \cstar{} (cf.\ \cref{prop:minimal_description}), this ideal must identify a $\sigma$-completion, hence the only one, and it is clearly $\bhilb$ with the inclusion of $\khilb_1$ being the usual one.
\end{proof}

\section{Equivalences and dualities}\label{sec:equivalences_dualities}

This section investigates the connections between the categories previously introduced. 
For an overview of their definitions, we refer back to \cref{tab:meas,tab:cstar,tab:bool} in the introduction.

Our first goal is to prove the equivalence $\cSigmaC\cong \SigmaB$, which we perform in \cref{sec:equivalence}. 
We present two distinct approaches: one based on the Gelfand and Stone dualities, and another that emphasizes the measurable aspects.
Another key result here is \cref{thm:sigma_normal_proj}, a criterion for $\sigma$-normality which we use in \cref{sec:tensor_cstars} to construct a tensor product for commutative \cstars{}s.
Another is \cref{thm:povm_integration}, which sheds light on the probabilistic nature of \scpu{} maps and is fundamental for the study of equivalences and dualities with probabilistic morphisms in \cref{sec:probabilistic}.
The main result of the paper, the measurable Gelfand duality, is stated and proved in \cref{sec:concrete_duality}, where we focus on the concrete side.

\subsection{Equivalence of commutative \texorpdfstring{\cstars s}{σC*-algebras} and \texorpdfstring{\bsigma s}{Boolean-sigma-algebras}}
\label{sec:equivalence}

The aim of this section is to construct an equivalence $\cSigmaC \cong \SigmaB$.
To begin, constructing a functor $\cSigmaC \to \SigmaB$ is quite straightforward as follows.

\begin{defn}
	For a commutative \cstars{} $\A$, we write $\proj(\A)$ for its poset of projections with respect to the induced order.
\end{defn}

Thanks to the commutativity, it is not hard to see that $\proj(\A)$ is a \bsigma{}, with algebraic operations given by
\[ 
	p \wedge q \coloneqq pq, \quad p \vee q\coloneqq  p+q-pq, \quad \text{and} \quad \neg p \coloneqq 1-p,
\]
as well as $\top=1$ and $\bot=0$, and suprema and infima formed as in $\A$ itself (\cref{ex:compression}).\footnote{See also~\cite[Corollary~2.2.6]{saito15monotone} for more details, where it is shown that the projections form a $\sigma$-complete lattice even if $\A$ is not commutative.}
Since every \starshom{} maps projections to projections in such a way that these algebraic operations and countable suprema are preserved, we obtain the desired functor
\[
	\proj\colon \cSigmaC \to \SigmaB.
\]
The remaining task is to show that this functor is an equivalence.
We provide two approaches for doing this.

The first approach implemented in \cref{sec:gelfand_stone_ap} uses Gelfand and Stone duality, together with a result ensuring that $\sigma$-additivity is preserved by both dualities. 
Although quite short to write down, this proof is indirectly involved by relying heavily on the established theory, and it does not shed much light on the measurable nature of \bsigma{}s and \cstars{}: for a measurable space $X$, the Gelfand spectrum of $\Linf(X)$ typically is much larger than $X$, and it may seem unsatisfactory to consider this space as the ``geometric object'' associated to $\Linf(X)$ rather than $X$ itself.

The second strategy, offered in \cref{sec:measurable_ap} describes directly how to obtain a \cstars{} from a \bsigma{}.
The correspondence at the level of morphisms will be based on a generalized notion of POVMs and PVMs, which specializes to Markov kernels (\cref{rem:markov_kernels}), and an integration theory that mimics the Lebesgue integral (\cref{thm:povm_integration}).
The disadvantage is that the functor $\SigmaB \to \cSigmaC$ constructed there is harder to work with directly.

\subsubsection{Gelfand--Stone approach}\label{sec:gelfand_stone_ap}

Both \bsigma{}s and commutative \cstars{}s correspond to Rickart spaces under Stone and Gelfand duality (\cref{cor:rickart_stone,fact:rickart_gelfand}).
In particular, $\proj$ is naturally isomorphic to the composite functor
\[
	\clopen \circ \spec \: : \: \cSigmaC \to \SigmaB,
\]
which sends a commutative \cstars{} $\A$ to the clopen sets in its Gelfand spectrum.
In particular, a candidate essential inverse of $\proj$ is given by $C(\stone(-))$, which sends a \bsigma{} $B$ to the complex-valued continuous functions on its Stone space.
The challenge is to ensure that this interacts well with $\sigma$-normality.

This is addressed by the following results, which allow us to prove fullness of $\proj$, or alternatively well-definedness of $C(\stone(-))$ as a functor $\SigmaB\to \cSigmaC$.
The fact that we formulate and prove these without assuming $\sigma$-completeness will come in handy later when we to discuss the relation between the universal $\sigma$-completions (see \cref{prop:universal_completions}).

\begin{lem}
	\label{lem:clopen_approx}
	Let $X$ be a Stone space.	
	Then for every continuous $f \colon X \to [0,1]$ and $\eps > 0$, there is a projection $p \in C(X)$ such that
	\begin{equation}
		\label{eq:clopen_approx}
		f- 2 \eps \le p \le \eps^{-1} f. 
	\end{equation}
\end{lem}

\begin{proof}
	Every Stone space is ultranormal~\cite[Proposition 8.8]{prolla82topics}.
	In particular, the disjoint closed sets 
	$f^{-1} ([2\eps , 1])$ and $f^{-1}([0,\eps])$ are separated by a clopen $U$ which contains the first set but is disjoint from the second.
	In other words, we have
	\[
		f(U) \subseteq (\eps,1], \qquad f(X \setminus U) \subseteq [0,2\eps).
	\]
	With $p \coloneqq \chi_U$ the indicator function, the desired inequalities hold because they hold both on $U$ and on $X \setminus U$.
\end{proof}

\begin{rem}\label{rem:clopen_approx_functional}
	In a \cstars{} $\A$, \cref{lem:clopen_approx} is true even without the factor of $2$ on the left in~\eqref{eq:clopen_approx}.
	Indeed, given any $a \in \A$ with $0 \le a \le 1$ and $\eps>0$, let us consider the projection $p \coloneqq \chi_{[\eps,1]}(a)$.
	Since
	\[
		\id_{[0,1]} - \eps \le \chi_{[\eps,1]} \le \eps^{-1} \id_{[0,1]}
	\]
	in $\Linf([0,1])$, this is true also in $\Linf(\spectrum(a))$, and thus $a - \eps \le p \le \eps^{-1} a$.  
\end{rem}

The following auxiliary notion is a weaker form of $\sigma$-normality (cf.~\cref{rem:sigmahom_add_mon}\ref{it:sigmaadd}).

\begin{defn}
	Let $X$ be a Stone space and $\A$ a \cstar{}.
	A positive map $\phi \colon C(X) \to \A$ is \newterm{$\sigma$-additive on projections} if for every sequence $(p_n)$ of pairwise orthogonal projections for which $\sum_n p_n$ exists in $C(X)$,\footnote{See~\cref{def:infinite_sum}. When it exists, this may be different from the pointwise sum.}
	\[
	\phi\left(\sum_n p_n \right) = \sum_n \phi(p_n).
	\] 
\end{defn}

\begin{rem}\label{rem:sigmaadd_mon}
	A positive map $\phi \colon C(X) \to \A$, with $X$ a Stone space, is $\sigma$-additive on projections if and only if it preserves suprema (or infima) of monotone increasing (resp.~decreasing) sequences of projections. 
	The proof is precisely the one of \cref{rem:sigmahom_add_mon}, noting that the meet of commuting projections is their product.
\end{rem}
\begin{thm}
	\label{thm:sigma_normal_proj}
	For a \cstar{} $\A$ and a Stone space $X$, a positive map
	\[
		\phi \colon C(X) \longrightarrow \A
	\]
	is $\sigma$-normal if and only if it is $\sigma$-additive on projections.
\end{thm}

\begin{proof}
	The ``only if'' part is clear by \Cref{rem:sigmaadd_mon}, so we focus on the less obvious ``if'' direction.
	To this end, it suffices to show that for every $(f_n) \searrow 0$ in $C(X)$, we have $(\phi(f_n))\searrow 0$. 
	We assume $f_n \le 1$ without loss of generality and fix $\eps > 0$.

	Then by \Cref{lem:clopen_approx} applied to $f_n$, we can find projections $p_{n,\eps} \in C(X)$ such that
	\[
		f_n - 2 \eps \le p_{n,\eps} \le \eps^{-1} f_n.
	\]
	Upon recursively replacing $p_{n+1,\eps}$ by the product $p_{n+1,\eps} p_{n,\eps}$, we may assume the monotonicity $p_{n+1,\eps} \le p_{n,\eps}$. 
	Indeed whenever $p_{n, \eps}$ and $p_{n+1, \eps}$ satisfy the desired inequalities, then also
	\[
		f_{n+1} - 2 \eps \le p_{n+1,\eps} p_{n,\eps} \le \eps^{-1} f_{n+1},
	\]
	where the first inequality follows from $f_{n+1}- 2 \eps \le f_{n} - 2 \eps \le p_{n,\eps}$ and the fact that the product of commuting projections is their meet.
	Therefore, without loss of generality, the sequence $(p_{n,\eps})$ is monotone decreasing.

	But then for every $g \in C(X)_{\sa}$, we have
	\[
		g \le p_{n,\eps} \quad \forall n \qquad \Longleftrightarrow \qquad g \le 0.
	\]
	Indeed if $g \le p_{n,\eps}$ for all $n$, then we get $g \le \eps^{-1} f_n$ for all $n$, and thus $g \le 0$; the converse is clear by $p_{n,\eps} \ge 0$.
	This proves that $\inf_n p_{n,\eps} = 0$.
	Hence by \cref{rem:sigmaadd_mon},
	\[
		\inf_n \phi(f_n) \le \inf_n \phi(p_{n,\eps}) + 2 \eps = 2 \eps,
	\]
	and so we conclude the desired $\inf_n \phi(f_n) = 0$ by taking $\eps \to 0$.
\end{proof}

\begin{thm}\label{thm:restriction_gs_shom}
	There is a commuting diagram of functors 
	\begin{equation}
		\label{eq:CStone}
    		\begin{tikzcd}[row sep=7ex]
			\SigmaB \ar[rr,hookrightarrow] \ar[d, "C(\stone(-))"] && \Boolshom \ar[rr,hookrightarrow] \ar[d, "C(\stone(-))"] && \Bool \ar[d,"C(\stone(-))"] \\
			\cSigmaC \ar[rr,hookrightarrow] && \cCalgshom \ar[rr,hookrightarrow] && \cCalg
    		\end{tikzcd}
	\end{equation}
	where the vertical functors are fully faithful.
\end{thm}

\begin{proof}
	By Gelfand and Stone duality, the functor $C(\stone(-))$ on the right is fully faithful. 
	Moreover, the restrictions displayed in \eqref{eq:CStone} indeed hold, meaning that a $\sigma$-homomorphism $A\to B$ between Boolean algebras translate to a $\sigma$-homomorphism $C(\stone(A))\to C(\stone(B))$ between the associated \cstar{}s. This is a consequence of \cref{thm:sigma_normal_proj}.
	Commutativity of the diagram is trivial.

	It remains to show that the functors $\Boolshom\to\cCalgshom$ and $\SigmaB\to \cSigmaC$ are fully faithful. 
	While faithfulness is clear as the right vertical arrow is faithful, 
	fullness holds because every $\sigma$-homomorphism $C(\stone(A))\to C(\stone(B))$ is the unique extension of its restriction to a $\sigma$-homomorphism $A\to B$ by the fact that simple elements are dense in $C(\stone(A))$.
	%
	%
\end{proof} 


\begin{cor}\label{cor:proj_equiv}
	There is an equivalence of categories
	\[
    \begin{tikzcd}
        \SigmaB \ar[rr,bend left=15,"C(\stone(-))"] && \cSigmaC \ar[ll,bend left=15,"\proj"]
    \end{tikzcd}
    \]
\end{cor}

\begin{proof}
	For a \bsigma{} $B$, the natural isomorphism $\proj(C(\stone(B))) \cong B$ is clear by Stone duality.
	Moreover, $C(\stone(-))$ is essentially surjective because commutative \cstars{}s and \bsigma{}s both correspond to Rickart spaces (\cref{cor:rickart_stone,fact:rickart_gelfand}), so it is an equivalence of categories by \cref{thm:restriction_gs_shom}.
	Essential uniqueness of the inverse is then sufficient to ensure that $\proj$ is the inverse equivalence of $C(\stone(-))$.
\end{proof}

We now turn to the question of how this equivalence makes the Baire envelope of a Boolean algebra match up with the Pedersen--Baire envelope of the associated \cstar{}, and likewise for universal $\sigma$-completions.

\begin{prop}\label{prop:universal_completions}
	The following diagrams of functors commute up to natural isomorphism:
	\[
	\begin{tikzcd}[row sep=7ex]
		\Bool \ar[rr,"(-)^{\infty}"]\ar[d, "C(\stone(-))"] && \SigmaB \ar[d,"C(\stone(-))"] 
								   && \Boolshom \ar[rr,"\univ{(-)}"]\ar[d, "C(\stone(-))"] && \SigmaB \ar[d,"C(\stone(-))"] \\
		\cCalg \ar[rr,"(-)^{\infty}"] && \cSigmaC 
					     && \cCalgshom \ar[rr,"\univ{(-)}"] && \cSigmaC 
	\end{tikzcd}
	\]
	In other words, the Baire envelope of a Boolean algebra corresponds to the Pedersen--Baire envelope of the associated \cstar{}, and similarly for the universal $\sigma$-completions.
\end{prop}

\begin{proof}
	We focus on the first diagram, as the proof for the second is perfectly analogous.
	For this case, we observe that the lower composite is left adjoint to $\proj \colon \cSigmaC \to \Bool$, since we have bijections
	\begin{equation*}
		\begin{array}{rll}
			\cSigmaC (C(\stone(B))^{\infty},\A) \!{} & \cong \: \cCalg (C(\stone(B)), \A)& \text{by \cref{thm:pb_univ}} \\
			& \cong \: \cCalg (C(\stone(B)),C(\stone(\proj(\A))))\quad & \text{by \cref{cor:proj_equiv}}\\ 
			& \cong \: \Bool (B, \proj(\A))& \text{by \cref{thm:restriction_gs_shom}}
		\end{array}
	\end{equation*}
	natural in $\A \in \cSigmaC$ and $B \in \Bool$.
	Similar reasoning shows that the upper composite $C(\stone((-)^{\infty}))$ is left adjoint to $\proj$ as well.
	Therefore the claims follow by essential uniqueness of the adjoint.
\end{proof}

\subsubsection{Measurable approach}\label{sec:measurable_ap}

We now present a more direct approach to the equivalence $\cSigmaC \cong \SigmaB$ from \Cref{cor:proj_equiv} which avoids the use of Gelfand or Stone duality.
As per \Cref{def:baire}, we write $\baire(\mathbb{C})$ for the Boolean $\sigma$-algebra given by the Baire (or, equivalently, Borel) sets in $\mathbb{C}$.

\begin{defn}\label{def:linf_bool}
    For a Boolean $\sigma$-algebra $B$, we define 
\[
	\Linfbool(B) \coloneqq \left \lbrace f\colon \baire(\mathbb{C}) \to B\ \text{ }\sigma\textrm{-hom.} \ \middle\vert\, \begin{array}{l}
    \text{there is bounded }U \in \baire(\mathbb{C})\\ \text{such that }f(U)= \top\end{array}\hspace{-1ex}  \right\rbrace.
\]
\end{defn}

Here, we read the subscript ``abs'' as standing for ``abstract''.

Our first task is to equip $\Linfbool(B)$ with a $*$-algebra structure.
To this end, we use the following notion of tupling.
By \cref{prop:concrete_regular,prop:tensor_standardborel}, there is a canonical isomorphism
\[
	\baire(\mathbb{C}^n) \cong \baire(\mathbb{C})^{\unitensor n}
\]
where the right-hand side denotes the $n$-fold universal tensor power of $\baire(\mathbb{C})$, or equivalently the $n$-fold coproduct in $\csigmaB$.
Therefore, for every $f_1,\dots, f_n \in \Linfbool(B)$, we have a tupling 
\[
	\langle f_1 , \dots ,f_n \rangle \colon \baire(\mathbb{C}^n)\to B.
\] 
This is the unique \bsigmahom{} such that
\[
	\langle f_1, \dots , f_n \rangle (U_1 \times \dots \times U_n)=f_1 (U_1) \wedge \dots \wedge f_n (U_n)
\]        
for all measurable $U_1, \dots, U_n \subseteq \mathbb{C}$.
With this in mind, we can define a sum $f + g \in \Linfbool(B)$ for any $f, g \in \Linfbool(B)$ by considering
\[
	\operatorname{sum} \colon  \baire(\mathbb{C})\to \baire(\mathbb{C}^2)
\]
which is given taking the preimage along the addition map $\mathbb{C}^2 \to \mathbb{C}$, and composing it with the pairing,
\[
f+g \coloneqq \operatorname{sum} \circ \, \langle f,g \rangle  .   
\]  
Using multiplication on $\mathbb{C}$ instead of addition also gives us a product on $\Linfbool(B)$.
A straightforward check shows that these operations make $\Linfbool(B)$ into a commutative ring, with the following additive and multiplicative identities:
\[
  0(U) \coloneqq \begin{cases}
    \top & \text{if }0 \in U,\\
    \bot & \text{otherwise},
    \end{cases} \qquad \text{and}\qquad  1(U) \coloneqq \begin{cases}
        \top & \text{if }1 \in U,\\
        \bot & \text{otherwise.}
    \end{cases}
\]
In fact, one can similarly check that $\Linfbool(B)$ is a complex $*$-algebra with scalar multiplication $(\lambda f)(U)\coloneqq f(\frac{1}{\lambda}U)$ for every $\lambda\in \mathbb{C}\setminus \lbrace 0\rbrace$ and involution $(f^*)(U)\coloneqq f(U^*)$.
Then the following statement follows by direct check. 

\begin{prop}\label{prop:staralg}
    $\Linfbool(B)$ is a commutative unital $*$-algebra.
\end{prop}

A similar result was already discussed by Sikorski in \cite[\S~43]{sikorski69boolean}, where he mentioned the algebra structure on the set of all \bsigmahom{} $\baire(\mathbb{R}) \to B$ (i.e.~without our boundedness requirement).
Our treatment is somewhat similar to Olmsted's approach~\cite{olmsted42lebesgue}, although we make the intrinsically categorical nature of the construction more explicit.
In particular, our construction of the $*$-algebra structure is best understood from the perspective of Lawvere theories.

Unfortunately, proving that $\Linfbool(B)$ is a \cstars{} is not straightforward from this categorical perspective.
Instead, it is more convenient to utilize a more down-to-earth approach, mindful of the one discussed by Sikorski, that describes $\Linfbool(B)$ as a particular quotient of some $\Linf(X)$, with $X$ a measurable space.
So let us start with the concrete case first.

\begin{lem}
	\label{lem:measurableviewpoint}
    	Let $(X, \Sigma(X))$ be a measurable space. Then there is a natural isomorphism of commutative unital $*$-algebras
    	\begin{equation}
	    	\label{eq:Linf_boolvsmeas}
	    	\Linf(X) \cong \Linfbool(\Sigma(X))
    	\end{equation}
	given by $f \mapsto (U \mapsto f^{-1}(U))$.
\end{lem}

\begin{proof}
	By Loomis--Sikorski duality (\cref{thm:ls_duality}) and sobriety of $\mathbb{C}$ as a measurable space, the \bsigmahom{}s $\baire(\mathbb{C}) \to \Sigma(X)$ correspond to the measurable functions $X \to \mathbb{C}$, and it is straightforward to see that our boundedness condition corresponds to boundedness of the functions.
\end{proof}

This justifies the notation $\Linfbool$, which suggests that the construction is a variant of $\Linf$ for abstract \bsigma{}s --- that is, not necessarily concrete.

\begin{nota}
    $\ball{\lambda}{r}$, resp.~$\cball{\lambda}{r}$, denotes the open, resp.~closed, ball of radius $r\ge 0$ around $\lambda \in \mathbb{C}$.
\end{nota}

We now equip $\Linfbool(B)$ with a norm given by
\begin{equation}\label{eq:norm_linf}
	\norm{f} \coloneqq \inf \, \lbrace r \mid f(\cball{0}{r})= \top \rbrace.
\end{equation}
This is finite by the boundedness assumption in \cref{def:linf_bool}, and the minimum is attained because $\cball{0}{r}= \bigcap_{n} \cball{0}{r+\frac{1}{n}}$ and $f$ preserves countable intersections.
It is easy to see that this norm makes the isomorphism~\eqref{eq:Linf_boolvsmeas} isometric.

\begin{prop}\label{prop:linf_morphisms}
	For two \bsigma{}s $A$ and $B$, every \starshom{} $\phi \colon A \to B$ induces a $*$-homomorphism given by postcomposition,
	\begin{align*}
		\phi^* \colon \Linfbool(A) & \longrightarrow \Linfbool(B) \\
		f & \longmapsto  \phi f.
	\end{align*}
	Moreover, this morphism is \emph{contractive}, i.e.~$\norm{\phi^* (f)}\le \norm{f}$ for all $f \in \Linfbool(A)$.
\end{prop}
\begin{proof}
	Concerning addition,
    \begin{equation*}
        \begin{split}
        \phi^* (f+g) (U) & = \phi \operatorname{sum} \langle f,g \rangle (U) = \phi \langle f,g \rangle (\operatorname{sum}^{-1} (U))\\ 
			 &=  \langle \phi f, \phi g \rangle (\operatorname{sum}^{-1} (U))= \operatorname{sum} (\langle \phi f, \phi g \rangle ) (U) = \left( \phi^* f + \phi^* g \right) (U),
    \end{split}
    \end{equation*}
    where the third equality comes from the fact that $\phi \langle f,g \rangle=\langle \phi f, \phi g \rangle$, since these two coincide on rectangles. Analogously, one can prove that also multiplication is preserved.
	From direct computations, we also see that $\phi^*$ is unital, contractive, and respects scalar multiplication and the $*$-operation.
\end{proof}

\begin{rem}
Despite the advantages offered by \cref{def:linf_bool}, such as \cref{prop:linf_morphisms} being easy to prove from first principles, proving the completeness of $\Linfbool(B)$ is not at all straightforward, and similarly for the monotone $\sigma$-completeness.

To illustrate the difficulty, consider the concrete case where $B = \Sigma(X)$ for some measurable space $(X, \Sigma(X))$, and take a Cauchy sequence $(f_n)$ in $\Linf(X)$ with limit $f$. 
Then based on the fact that the limit of a Cauchy sequence of real numbers is equal to its $\limsup$ and its $\liminf$, one might guess that
\begin{equation}
	\label{eq:liminf_cauchy_fail}
	f^{-1} (U) \stackrel{?}{=} \liminf_n f_n^{-1}(U)\coloneqq \bigcup_{n} \, \bigcap_{m \ge n} \, f_m^{-1}(U)
\end{equation}
for every $U \in \Sigma(X)$, and that this $\liminf$ coincides with the analogously defined $\limsup$.
However, this approach fails: if we use the $\eps$-$\delta$ definition of limit to compute $f^{-1}(C)$ for a closed set $C \subseteq \mathbb{C}$, it turns out that
\begin{equation}
	\label{eq:liminf_cauchy_correct}
	f^{-1}(C)= \bigcap_{\ell} \, \bigcup_{n} \, \bigcap_{m \ge n} \, f_m^{-1}(C + \ball{0}{1/\ell}),
\end{equation}
and in general this is different from the suspected~\eqref{eq:liminf_cauchy_fail}.
To continue the proof in terms of $*$-algebra structure as above, one would then need to show that the right-hand side of~\eqref{eq:liminf_cauchy_correct} indeed extends to a \bsigmahom, and that this is exactly the limit of the Cauchy sequence $(f_n^{-1}) \in \Linfbool(\Sigma(X))$ with respect to the norm~\eqref{eq:norm_linf}.
Even if such an extension can be constructed, working with~\eqref{eq:liminf_cauchy_correct} explicitly looks quite cumbersome and would be difficult to follow for a human reader.
\end{rem}

So to prove that $\Linfbool(B)$ is indeed a \cstars, we instead opt for a ``concrete perspective'' by writing $B$ as a quotient of a concrete \bsigma.
Indeed by the Loomis--Sikorski representation theorem (\cref{thm:ls_rep}), there exists a measurable space $(X, \Sigma(X))$ together with a $\sigma$-ideal $I\subseteq \Sigma(X)$ such that $B\cong \Sigma(X)/I$.

\begin{prop}\label{prop:linf_cstars2}
     Let $(X, \Sigma(X))$ be a measurable space and $I \subseteq \Sigma(X)$ a $\sigma$-ideal $I$.
     Then there is a $*$-isomorphism
     \[
	     \Linfbool(\Sigma(X) / I) \,\cong\, \Linf(X)/ \I,
     \]
     where the $\sigma$-ideal
     \[
	     \I \coloneqq \lbrace g \colon X \to \mathbb{C} \text{ bounded measurable} \mid \mathrm{supp}(g) \in I \rbrace
     \]
     is the $\sigma$-subspace generated by the $\chi_p$ for $p \in I$.
\end{prop}

Thus the elements of $\Linfbool(\Sigma(X)/I)$ can be represented as bounded measurable functions on $X$ modulo ``$I$-almost sure equality''.

\begin{proof}
    By composition, we obtain a map
    \[
	    \phi \: \colon \: \Linf(X) \stackrel{\cong}{\longrightarrow} \Linfbool (\Sigma(X)) \longrightarrow \Linfbool(\Sigma(X)/I).
    \]
    This $\phi$ is a $*$-homomorphism by \cref{prop:linf_morphisms}.
    Moreover, by \cite[32.5]{sikorski69boolean}, every \bsigmahom {} $\baire(\mathbb{C}) \to \Sigma(X) / I$ can be lifted to $\baire(\mathbb{C}) \to \Sigma(X)$,
    and therefore $\phi$ is surjective.\footnote{In order to see that this still works with our boundedness condition, we really need to use a sufficiently large bounded subset in place of $\mathbb{C}$.}
    The kernel of $\phi$ coincides, by definition, with $\I$.
    We therefore have constructed the desired isomorphism.

    The final statement is an instance of \cref{ex:ideal_from_ideal}.
\end{proof}

\begin{cor}
	\label{cor:linf_cstars}
    Let $B$ be a \bsigma. 
    Then $\Linfbool(B)$ is a commutative \cstars{} with respect to the norm \eqref{eq:norm_linf}.
\end{cor}

\begin{proof}
  	Continuing in the same notation, we assume $B = \Sigma(X)/I$ without loss of generality.  
	Then by \cref{fact:sigma_closure}, we know that $\Linf(X)/\I$ is a \cstars, and hence so is $\Linfbool(B)$.
	It remains to be shown that the norm \eqref{eq:norm_linf} is indeed the quotient norm.

	Denoting the quotient norm with a prime, we get that for $f \in \Linfbool(B)$,
    \[
    \norm{f}' = \inf_{g \in \I} \norm{h+g} = \inf_{g \in \I} \inf \, \lbrace r\ge 0 \mid h(x) + g(x) \in \cball{0}{r} \text{ for all }x \in X \rbrace,
    \]
    where $h : X \to \mathbb{C}$ is any bounded measurable function representing the class of $f$. To show that this norm coincides with the one induced by \eqref{eq:norm_linf} via the isomorphism $\Linfbool(\Sigma(X) / I) \cong \Linf(X)/ \I$, we need to prove that $\norm{f}'$ is equal to
    \[
    \alpha \coloneqq \inf \left\lbrace r \ge 0 \:\bigg|\: h^{-1} \! \left( \cball{0}{r}^c \right) \in I \right\rbrace. 
    \]
    So let $r> \norm{f}'$. Then there exists $g \in \I$ such that $h(x) + g(x) \in \cball{0}{r}$ for all $x \in X$. In particular,
    \[
	    \lbrace x \in X \mid h(x) \notin \cball{0}{r} \rbrace \:\subseteq\: \lbrace x \in X \mid g(x) \neq 0 \rbrace \in I,
    \]
    and therefore $r \ge \alpha$.
    Hence $\norm{f}' \ge \alpha$. 

    Conversely, suppose $r > \alpha$. 
	Then $N\coloneqq h^{-1} (\cball{0}{r}^c) \in I$; we define $g\coloneqq h \chi_{N}$. 
	Thus $g \in \I$ by construction.
	We also have achieved $h(x) - g(x) \in \cball{0}{r}$ for all $x \in X$, and so $r \ge \norm{f}'$.
	Hence $\norm{f}' \le \alpha$, and thus $\norm{f}' = \alpha$.
\end{proof}

\begin{rem}\label{rem:Linf_specs}
    Using the isomorphism $\Linfbool(B) \cong \Linf(X) / \I$ is one way to see the following for $f\in \Linfbool(B)$:
	\begin{enumerate}
		\item $f$ is self-adjoint if and only if $f(\mathbb{R})=\top$.
		\item $f$ is positive if and only if $f(\mathbb{R}^{\ge 0})= \top$.
		\item\label{it:Linf_specs_proj} $f$ is a projection if and only if $f(\lbrace 0,1 \rbrace)= \top$.
	\end{enumerate}
\end{rem}

We now continue our study of morphisms started in \cref{prop:linf_morphisms}.

\begin{nota}\label{nota:char_fun}
	For any $p \in B$, we write $\chi_p$ for the \newterm{characteristic function} $\chi_p \colon \baire(\mathbb{C}) \to B$, the unique element of $\Linfbool(B)$ with $\chi_p(\{0\}) = \neg p$ and $\chi_p(\{1\}) = p$.
\end{nota}

The following holds as a simple consequence of \cref{rem:Linf_specs}\ref{it:Linf_specs_proj}.

\begin{lem}\label{lem:proj_B}
	The map
    	\begin{align*}
		\eta_B \colon \proj(\Linfbool(B)) & \longrightarrow B \\
		\chi & \longmapsto \chi (\lbrace 1 \rbrace)
	\end{align*}
	is an isomorphism.\footnote{$\eta_B$ is actually a component of a natural isomorphism of functors, but to state this we will need to define the $\Linfbool$ functor first (see \cref{propdef:linf,thm:general_duality}).} 
	In particular, the projections of $\Linfbool(B)$ are precisely the characteristic functions.
\end{lem}

To study morphisms, we give the following definitions of PVMs and POVMs, which generalize the ones usually studied in quantum probability~\cite[pp.~48,49]{paulsen2002} by allowing the domain to be a general \bsigma{}.

\begin{defn}\label{def:povm}
	Let $B$ be a \bsigma{} and let $\A$ be a \cstars{}.
	An $\A$-valued \newterm{POVM} on $B$ is a map $\mu \colon B \to {\A}$ satisfying the following properties:
	\begin{enumerate}
		\item (positivity) $\mu(p)\ge 0$ for all $p\in B$;
		\item (normalization) $\mu(\top)=1$;
		\item ($\sigma$-additivity) If $(p_n)$ is a countable sequence of pairwise disjoint elements, then\footnote{This equation requires the infinite sum on the right hand side to exist \emph{and} to coincide with the left hand side. Since any POVM is automatically monotone by binary additivity, we necessarily have $\sum_n \mu(p_n) \le 1$.} 
		\[
		\mu ( \sup_n p_n ) = \sum_n \mu (p_n).
		\]
	\end{enumerate}
	Such a $\mu$ is a \newterm{PVM} if, additionally, 
	\begin{enumerate}\setcounter{enumi}{3}
		\item (spectrality) $\mu(p \wedge q )= \mu(p) \mu(q)$ for all $p,q \in B$.
	\end{enumerate}
\end{defn}

Note that if $\mu$ is a PVM, then $\mu(p) \mu(p) = \mu(p \wedge p )= \mu(p)$ immediately implies that $\mu(p)$ is a projection, and these projections are required to commute as $p$ varies.

\begin{rem}\label{rem:povm_mon}
	By arguing as in \cref{rem:sigmahom_add_mon}, it is clear that a POVM preserves suprema (resp.~infima) of monotone increasing (resp.~decreasing) sequences. 
\end{rem}

\begin{ex}
	The canonical map $\chi_- \colon B \to \Linfbool(B)$ is a PVM.
	Indeed this is clear in the concrete case, and this implies the general case by \cref{prop:linf_cstars2}.

	More generally, every \bsigmahom{} $\phi \colon A \to B$ induces a PVM given as the composite
	\[
		\begin{tikzcd}
			A \ar[r,"\phi"] & B \ar[r,"\chi_-"] & \Linfbool(B).
		\end{tikzcd}
	\]
\end{ex}

\begin{thm}\label{thm:povm_integration}
	Let $B$ be a \bsigma{} and $\A$ a \cstars{}.
	Then for every POVM $\mu : B \to \A$, there is a unique \scpu{} map $\phi \colon \Linfbool(B) \to \A$ which extends $\mu$ in the sense that the diagram
	\[
		\begin{tikzcd}
			B \ar[rd,"\mu"'] \ar[rr,"\chi_-"] && \Linfbool(B) \ar[dl,"\phi"] \\
			& \A 
		\end{tikzcd}
	\]
	commutes. Moreover, this extension is unique even among all positive maps.
\end{thm}

In other words, we obtain a bijective correspondence between POVMs $B \to \A$ and \scpu{} maps $\Linfbool(B) \to \A$.
This is well-known in the concrete case, in which case one also writes
\[
\phi(f) = \int f \,\mathrm{d}\mu,
\]
and we will adopt this notation in the general case as well.
As far as we are aware, the most general previous result has been given by Wright, who considered positive $\sigma$-additive measures on measurable spaces with values in a $\sigma$-space, i.e.~an ordered vector space closed under suprema of monotone sequences~\cite[Section~2]{wright72measures}. 
Mutatis mutandis, Wright's proof is a rewrite of the existence proof of the Lebesgue integral.
Our contribution is the generalization of the domain to \bsigma{}s (not necessarily concrete).\footnote{Our statement and proof also readily generalize to POVMs on \bsigma{}s taking values in a $\sigma$-space, but we prefer to stick to the case of \cstars{}s and normalized measures to reduce the proliferation of concepts.}



\begin{proof}
	Since points play a crucial role in the standard arguments, we once again employ the Loomis--Sikorski representation theorem (\cref{thm:ls_rep}) to reduce the problem to the concrete case.
	As already mentioned, the statement is known if $B$ is concrete~\cite[Section 2]{wright72measures}.

	In general, we again assume $B = \Sigma(X)/I$ without loss of generality.
	Then the upper square in the diagram
	\begin{equation}
		\label{eq:LS_integral_diagram}
		\begin{tikzcd}
			\Sigma(X) \ar[rrrr,"\chi_-"] \ar[dr,twoheadrightarrow] \ar[dr,twoheadrightarrow] &&&& \Linfbool(\Sigma(X)) \ar[dl,twoheadrightarrow] \ar[bend left,ddll,"\psi"] \\
			& \Sigma(X)/I \ar[rr,"\chi_-"] \ar[dr,"\mu"' pos=0.4] && \Linfbool(\Sigma(X)/I) \ar[dl,"\phi" pos=0.4, dashed] \\
			&& \A \\
		\end{tikzcd}
	\end{equation}
	commutes.
	Thanks to the statement in the concrete case, we know that the left composite extends to a \scpu{} map $\psi \colon \Linfbool(\Sigma(X)) \to \A$.
	Thus by \cref{prop:linf_cstars2}, it factors across the quotient $\Linfbool(\Sigma(X)/I)$, because $\psi(\chi_p) = \mu(\bot) = 0$ for all $p \in I$.
	The last part of the statement is an immediate consequence of \cref{thm:sigma_normal_proj}.
\end{proof}

\begin{cor}\label{cor:pvm}
	For every PVM $\mu \colon B \to \A$, there is a unique \starshom{} $\phi \colon \Linfbool(B) \to \A$ which extends $\mu$. Moreover, this extension is unique even among all positive maps.
\end{cor}

\begin{proof}
	We need to prove that for every PVM $\mu \colon B \to \A$, the unique extension $\Linfbool(B) \to \A$ is multiplicative.
	This is again standard in the case where $B$ is concrete, and reduces to the concrete case in general: if $\psi$ in \eqref{eq:LS_integral_diagram} is multiplicative, then so is $\phi$.
\end{proof}

We are now ready to prove that $\Linfbool$ is a functor, and actually an equivalence. 

\begin{prop}\label{propdef:linf}
    The association $\Linfbool\colon \SigmaB \to \cSigmaC$ defined by
     \[
    B \mapsto \Linfbool(B) \qquad \text{and} \qquad (\phi\colon A \to B) \mapsto \left(\arraycolsep=1.4pt \begin{array}{rrcl}
        \phi^* \colon & \Linfbool(A) & \to& \Linfbool(B)\\
        & f &\mapsto& \phi f
    \end{array}\right)
    \]
    is a functor. 
\end{prop}
\begin{proof}
	$\Linfbool(B)$ is a commutative \cstars{} by \cref{cor:linf_cstars}.
	To prove that $\phi^*$ is indeed a \starshom{}, it suffices to combine \cref{prop:linf_morphisms,cor:pvm}, since $\phi^*$ extends the PVM given by the composition of $\phi$ with the inclusion $\chi_{-} \colon B \to \Linfbool(B)$.
	Functoriality is clear by the uniqueness of the extension.
\end{proof}

\begin{thm}\label{thm:general_duality}
    There is an equivalence
    \[
    \begin{tikzcd}
	    \SigmaB \ar[rr,bend left=15,"\Linf"] \ar[rr,phantom,"\cong"] && \cSigmaC \ar[ll,bend left=15,"\proj"]
    \end{tikzcd}
    \]
\end{thm}

By the equivalence in the Gelfand--Stone approach (\cref{cor:proj_equiv}), it also follows that $\Linfbool$ is naturally isomorphic to $C(\stone(-)) \colon \SigmaB \to \cSigmaC$.

\begin{proof} 
	\cref{lem:proj_B} has constructed a natural isomorphism $\proj \circ \Linfbool \cong \id$.
	In the other direction, the inclusion $\proj(\A) \to \A$ is a PVM for every commutative \cstars{} $\A$, and by \cref{cor:pvm} this extends to a \starshom{} $\phi_\A \colon \Linfbool(\proj(\A))\to \A$. 
	It is an isomorphism since it is bijective on projections, and projections are sufficient to determine all elements by \cref{cor:simple_dense}.
	Naturality in $\A$ is obvious.
\end{proof}

As one benefit, we obtain the analogue of \cref{prop:universal_completions} for the regular $\sigma$-completions.
Due to the lack of functoriality of the regular $\sigma$-completion, we only state it as an objectwise isomorphism.

\begin{cor}\label{cor:regular_completions}
	Regular $\sigma$-completions commute, i.e.\ $\Linfbool (\reg{B})\cong \reg{C(\stone(B))}$ for every \bsigma{} $B$.
\end{cor}
\begin{proof}
	As noted after the statement of \cref{thm:general_duality}, $\Linfbool \cong C(\stone(-))$.
	Moreover, by \cref{prop:linf_cstars2}, regular $\sigma$-completions of Boolean algebras are sent to regular $\sigma$-completions of their associated \cstar{}s because they are both given by quotienting by meager sets: compare \cref{ex:commutative_regular} with \cref{prop:boolean_regular}.
\end{proof}

\begin{rem}\label{rem:univ_reg_differ2}
	Given that both the universal and the regular $\sigma$-completions are preserved by $\Linfbool$ (\cref{prop:universal_completions,cor:regular_completions}), \cref{ex:tensor_product} lets us conclude that the universal and regular $\sigma$-completions of a \cstar{} can differ.
	This was already known to Wright~\cite[p.~303]{wright76regular}.
\end{rem}

\subsection{Concrete equivalence and measurable Gelfand duality}\label{sec:concrete_duality}

In the Boolean setting, we had introduced concrete \bsigma{}s as those isomorphic to the $\sigma$-algebra of a measurable space (\cref{def:concrete_bool}).
On first thought, one may expect that the analogous notion in the \cstar{} setting would be the $\sigma$-representable \cstars{}s from \cref{def:sigma_representable}.
However, under the equivalence $\SigmaB \cong \cSigmaC$, the $\sigma$-representable \cstars{}s do not generally correspond to the concrete \bsigma{}s, as for example $L^{\infty}([0,1])$ is $\sigma$-representable, but its \bsigma{} of projections is not concrete (\cref{ex:nonconcI}).

On a commutative \cstars{}, the $\sigma$-homomorphisms to $\mathbb{C}$ are precisely the pure $\sigma$-states.
Therefore a notion of concreteness in the \cstar{} setting that matches the Boolean one 
is that of pure $\sigma$-representability.

\begin{cor}\label{cor:concrete}
    The equivalence of \cref{thm:general_duality} restricts to 
    \[
    \begin{tikzcd}
	    \csigmaB \ar[rr,bend left=15,"\Linfbool"] \ar[rr,phantom,"\cong"] &&  \cmsigmaC \ar[ll,bend left=15,"\proj"]
    \end{tikzcd}
    \]
\end{cor}

Hence by \cref{thm:concrete_duality}, we also obtain a dual equivalence $\cmsigmaC \cong^\op \sobmeas$.

\begin{proof}
	In the commutative case, pure states are the same as \starshom{}s to $\mathbb{C}$.
	Now if a commutative \cstars{} $\A$ is separated by \starshom{}s to $\mathbb{C}$, then $\proj(\A)$ is trivially separated by \bsigmahom{}s to $\{\bot, \top\}$ and therefore concrete.
	Conversely if a \bsigma{} $B$ is separated by \bsigmahom{}s to $\{\bot, \top\}$, then \cref{lem:measurableviewpoint} implies that $\Linfbool(B)$ is separated by \starshom{}s to $\mathbb{C}$, namely the point evaluation maps.
	Therefore the claim follows by \cref{thm:general_duality}.
\end{proof}

	Whether pure $\sigma$-representability can be of interest for further research remains unclear; \cref{rem:meas_not_inclusion} highlights an undesirable feature that can be considered an argument against its use.	

	Let us now focus on the connection between commutative \cstars{}s and measurable spaces.
	By composing \cref{thm:general_duality} and Loomis--Sikorski duality (\cref{thm:ls_duality}), we obtain a contravariant adjunction between $\cSigmaC$ and $\meas$, which we will call \newterm{measurable Gelfand duality}.
	But let us consider a more direct construction of the measurable space associated to a given commutative \cstars{} $\A$ first.

\begin{defn}
	For a commutative \cstars{} $\A$, the \newterm{Gelfand $\sigma$-spectrum} is
	\[
		\specs(\A) \coloneqq \{\phi \colon \A \to \mathbb{C} \mid \phi \text{ is a \starshom{}}\}
	\]
	equipped with the smallest $\sigma$-algebra which makes the evaluation maps
	\begin{align*}
			\ev_a \colon \specs(\A) & \longrightarrow \mathbb{C}\\
			\phi & \longmapsto \phi(a)
	\end{align*}
	measurable for all $a \in \A$. 
\end{defn}

This is clearly a functor $\specs\colon \cSigmaC \to \meas$, which acts on morphisms by composition.
We now show that it coincides with the functor obtained by composing \cref{thm:ls_duality,thm:general_duality}.

\begin{prop}\label{prop:psigma_ssigma}
	The functor $\specs$ is naturally isomorphic to $\stones \circ \proj$. 
\end{prop}
\begin{proof}
	Let $\A$ be a \cstars{}.
	Then the $\sigma$-algebra of $\specs(\A)$ is equivalently the $\sigma$-algebra $B$ generated by the $\ev_p$ with $p \in \proj(\A)$, as one can see by noting that the set of $a \in \A$ for which $\ev_a$ is measurable (with respect to $B$) is a sub-\cstars{} of $\A$, and $\proj(\A)$ generates $\A$ (\cref{cor:simple_dense}).
	In other words, this $\sigma$-algebra is generated by the sets of the form
	\begin{equation}\label{eq:evp_sigmaalg}
		\ev_p^{-1}(\lbrace 1 \rbrace) = \lbrace \phi \in \specs(\A) \mid \phi(p)= 1\rbrace.
	\end{equation}
	Comparing this with the analogous \eqref{eq:a_sigmaalg} shows that the bijection
	\[
		\specs(\A) = \cSigmaC(\A, \mathbb{C}) \cong \SigmaB(\proj(\A), \{\bot, \top\}) = \stones(\proj(\A))
	\]
	obtained from \cref{thm:general_duality} is an isomorphism of measurable spaces.
	Naturality is immediate.
\end{proof}

\begin{thm}[Measurable Gelfand duality]\label{thm:measurableGelfand}
		There is a contravariant idempotent adjunction 
		\[
		\begin{tikzcd}
			\meas\ar[rr,bend left=20,"\Linf"]\ar[rr,phantom,"\bot^{\op}"] &&\cSigmaC, \ar[ll,bend left=20,"\specs"]
		\end{tikzcd}
		\]
		where $\Linf$ is the functor described in \cref{ex:sigma_Linf}.
		It restricts to an equivalence
		\begin{equation}
			\label{eq:measurableGelfand}
			\begin{tikzcd}
				\sobmeas \ar[rr,bend left=15,"\Linf" pos=0.55] \ar[rr,phantom,"\cong^{\op}" pos=0.6] && \cmsigmaC \ar[ll,bend left=15,"\specs" pos=0.4]
			\end{tikzcd}
		\end{equation}
\end{thm}
\begin{proof}
	By \cref{prop:psigma_ssigma}, $\specs$ is naturally isomorphic to $\stones \circ \proj$, while $\Linf$ is naturally isomorphic to $\Linfbool \circ \Sigma$ by \cref{lem:measurableviewpoint}.
	The statement now follows by \cref{thm:general_duality,thm:ls_duality,cor:concrete}.
\end{proof}

\begin{prop}\label{prop:top_vs_meas_gelfand}
	Let $\cHaus$ be the category of compact Hausdorff spaces and continuous maps.
	Then measurable Gelfand duality is compatible with ordinary (topological) Gelfand duality in the sense that the diagram
	\begin{equation}
		\label{eq:measurableGelfand_chaus}
		\begin{tikzcd}[row sep=large]
			\cHaus \ar[rr] \ar[d,"C(-)"] && \sobmeas \ar[d,"\Linf"] \\\
			\cCalg \ar[rr,"(-)^{\infty}"] \ar[d,"\spec"] && \cmsigmaC \ar[d,"\specs"] \\
			\cHaus \ar[rr] && \sobmeas
		\end{tikzcd}
	\end{equation}
	commutes up to canonical natural isomorphism, where the upper and lower unlabeled arrow is the functor taking a compact Hausdorff space to its Baire measurable space,
	while $\spec\colon \cCalg \to \cHaus$ sends a commutative \cstar{} to its Gelfand spectrum.
\end{prop}

\begin{proof}
	Omitting ``up to natural isomorphism'' here and in the following, commutativity of the upper square is part of the statement of \cref{cor:baire_eq}.
	Since the left vertical functors are inverse equivalences, the commutativity of the lower square follows from the trivial commutativity of the outer square.
\end{proof}

\begin{cor}[Pedersen--Baire duality]\label{cor:pb_duality}
    Measurable Gelfand duality restricts to a contravariant equivalence
    \begin{equation}
	    \label{eq:pb_duality}
    	\begin{tikzcd}
        	\bairemeas \ar[rr,bend left=15,"\Linf"] \ar[rr,phantom,"\cong^{\op}"] &&  \cPBenv.\ar[ll,bend left=15,"\specs"]
    	\end{tikzcd}
    \end{equation}
\end{cor}

\begin{proof}
	The fact that $\Linf$ maps $\bairemeas$ to $\cPBenv$ is part of \cref{cor:baire_eq}.
	Since Baire measurable spaces are sober (\cref{rem:baire_sober}), we can apply~\eqref{eq:measurableGelfand} and~\eqref{eq:measurableGelfand_chaus} to conclude that $\specs$ is its essential inverse.
\end{proof}

We may restrict further to standard Borel spaces.

\begin{cor}\label{cor:borelmeas}
    The contravariant equivalence of \cref{cor:pb_duality} restricts further to 
    \begin{equation}
	    \label{eq:borelmeas}
	    \begin{tikzcd}
		    \borelmeas \ar[rr,bend left=15,"\Linf"] \ar[rr,phantom,"\cong^{\op}"] &&  \cpbsep. \ar[ll,bend left=15,"\proj"]
	    \end{tikzcd}
    \end{equation}
\end{cor}

\begin{proof} 
	While standard Borel spaces are usually defined as the Baire (or equivalently Borel) measurable spaces associated to separable completely metrizable spaces, we can equivalently use second countable compact Hausdorff spaces.\footnote{Recall that a compact metric space is necessarily complete.}
	Then the claim follows from \cref{cor:pb_duality} by the standard fact that a compact Hausdorff space $X$ is second countable if and only if $C(X)$ is separable.
\end{proof}

\subsection{With probabilistic morphisms}
\label{sec:probabilistic}

We now extend some of the duality results to probabilistic morphisms.
In the context of \cstars{}s, the obvious candidates for probabilistic maps are \scpu{} maps, given their similar nature to that of \cpu{} maps between \cstar{}s.
As far as we know, there is no separate concept for probabilistic morphisms between \bsigma{}s: any such definition will have to linearize the structure of the \bsigma{}s, which essentially amounts to moving to the associated \cstars{}s.
This issue is easier to resolve for measurable spaces, as the following standard concept shows.

\begin{defn} 
	For measurable spaces $(X,\Sigma(X))$ and $(Y,\Sigma(Y))$, a \newterm{Markov kernel} $\mu \colon X \rightsquigarrow Y$ is a function
	\begin{align*}
		\mu \colon \Sigma(Y) \times X & \longrightarrow [0,1]\\
        	(E,x) &\longmapsto \mu(E | x)
	\end{align*}
	such that:
	\begin{enumerate}
		\item $E \mapsto \mu(E | x)\in [0,1]$ is a probability measure on $Y$ for every $x \in X$.
		\item $x \mapsto \mu(E | x)$ is measurable for every $E\in \Sigma(Y)$.
	\end{enumerate} 
\end{defn}

Throughout this section, we denote Markov kernels by squiggly arrows $\rightsquigarrow$ to emphasize that they are not actual functions.

\begin{defn}[{e.g.~\cite[Section~4]{fritz2019synthetic}}]
	\label{def:stoch}
	$\stoch$ is the category whose objects are measurable spaces, morphisms are Markov kernels, and composition is given by the Chapman--Kolmogorov equation:
	for $\mu \colon X \rightsquigarrow Y$ and $\nu \colon Y \rightsquigarrow Z$,
	\[
		(\nu \circ \mu) (E | x) \coloneqq \int_Y \nu(E | y) \, \mathrm{d}\mu(y | x).
	\] 
	We denote by $\sobstoch$, $\bairestoch$ and $\borelstoch$, the full subcategories of $\stoch$ given by sober measurable spaces, Baire measurable spaces and standard Borel spaces, respectively.
\end{defn}

The identity morphism of a measurable space $X$ is the Markov kernel $\id_X \colon X \rightsquigarrow X$ given by $\id_X(E | x) \coloneqq \chi_E(x)$.
With the obvious abuse of notation, every measurable function $f \colon X \to Y$ induces a $\lbrace 0,1 \rbrace$-valued Markov kernel $X \rightsquigarrow Y$ by
$f(E | x) \coloneqq \chi_E(f(x))$,
and this construction defines a functor $\meas \to \stoch$.

\begin{rem}\label{rem:sobstoch}
	$\stoch$ is the Kleisli category of the Giry monad on $\meas$ (which we briefly discussed in \cref{sec:concrete_loomis_sikorski}).
	Since the Giry monad takes values in $\sobmeas$, this suggests that sober measurable spaces should be special in $\stoch$.
	This manifests itself in two ways:
	\begin{enumerate}
		\item For every $X \in \meas$, the sobrification unit
			\[
				X \longrightarrow \stones(\Sigma(X))\cong \specs(\Linf(X))
			\]
			becomes an isomorphism in $\stoch$, where its inverse is the obvious $\lbrace 0,1\rbrace$-valued Markov kernel. 
		In particular, the inclusion $\sobstoch \hookrightarrow \stoch$ is an equivalence.
		\item Between sober measurable spaces, the map from measurable maps to $\{0,1\}$-valued Markov kernels is bijective.
			In one direction, if $\mu \colon X \rightsquigarrow Y$ is such a Markov kernel, then $\mu(\_ | x)$ is a $\lbrace 0,1 \rbrace$-valued probability measure for every $x \in X$. 
		Since $Y$ is sober, this corresponds to a unique $y \in Y$.
		The unique function $f \colon X \to Y$ so obtained is measurable because $f^{-1}(U)= \mu (U|\_)^{-1}(\lbrace 1 \rbrace)$.
	\end{enumerate}
	For a general categorical discussion of these properties, we refer to Proposition 5.5, Theorem 3.14 and Remark 5.2 of \cite{moss2022probability}.
\end{rem}

\begin{rem}\label{rem:markov_kernels}
	There is a natural bijection between Markov kernels $X \rightsquigarrow Y$ and POVMs $\Sigma(Y) \to \Linf(X)$.
	In one direction, it is quite clear that a Markov kernel $\nu \colon X \rightsquigarrow Y$ induces a POVM given by 
	\begin{align*} 
		\Sigma(Y) & \longrightarrow \Linf(X) \\
		E & \longmapsto \nu(E | \_).
	\end{align*}
	Conversely, since a POVM $\mu \colon \Sigma(Y) \to \Linf(X)$ is normalized and positive, it takes values in the unit interval of $\Linf(X)$, i.e.~for every $E \in \Sigma(Y)$ we have a measurable function $\mu(E) \colon X \to [0,1]$. 
	Then 
	\begin{align*}
		\Sigma(Y) \times X & \longrightarrow [0,1] \\
		(E,x) & \longmapsto \mu(E)(x)
	\end{align*} 
	is a Markov kernel because every $x \in X$ induces a \starshom\ $\Linf(X)\to \mathbb{C}$ by evaluating functions, and therefore $E\mapsto \mu(E)(x)$ is a probability measure by \cref{ex:sigma_Linf}\ref{it:sigmastates_Linf}.
	It is obvious that these two constructions are inverse to each other.
\end{rem}

Combining this with \cref{thm:povm_integration} shows that every Markov kernel $\mu \colon X \rightsquigarrow Y$ induces a \scpu{} map $\Linf(\mu) \colon \Linf(Y) \to \Linf(X)$ given by
\begin{equation}
	\label{eq:Koopman}
	\Linf(\mu)(f)(x) = \int_Y f(y) \, \mathrm{d}\mu(y | x).
\end{equation}
Especially in the more applied context of stochastic dynamical systems, $\Linf(\mu)$ is known as the \emph{Koopman operator} associated to $\mu$ (e.g.~\cite{tafazzol2024koopman}).

\begin{prop}\label{prop:Linf_stoch}
	This assignment $\Linf \colon \stoch \to \cSigmaCcpu$ is a fully faithful functor, which restricts to an equivalence
	\begin{equation}
		\label{eq:measurableGelfand_stoch}
		\begin{tikzcd}
			\qquad \stoch \ar[rr,bend left=20,"\Linf"] \ar[rr,phantom,"\cong^{\op}"] && \cmsigmaCcpu \ar[ll,bend left=20,"\specs" pos=0.4]
		\end{tikzcd}
	\end{equation}
\end{prop}
\begin{proof}
	Given two composable Markov kernels $\mu \colon X \rightsquigarrow Y$ and $\nu \colon Y \rightsquigarrow Z$, for all $x \in X$ and $f \in \Linf(Z)$ we have
	\begin{equation*}
		\begin{split}
			\Linf (\nu \circ \mu)(f)(x) &= \int_Z f(z) \, \mathrm{d}(\nu \circ \mu) (z | x) = \int_Z f(z) \int_Y \mathrm{d}\nu (z|y) \, \mathrm{d}\mu(y|x) \\
					    &= \int_Y \left(\int_Z f(z) \, \mathrm{d}\nu(z|y)\right) \mathrm{d}\mu(y|x) = \int_Y \Linf(\nu)(f)(y) \, \mathrm{d}\mu (y|x) \\[4pt]
					    &= \Linf(\mu)(\Linf(\nu) (f)(x)) = (\Linf(\nu) \circ \Linf(\mu))(f)(x),
		\end{split}
	\end{equation*}
	where the third equality follows by Fubini's theorem. 
	Therefore, composition is preserved, and since preservation of identities is straightforward, we obtain functoriality.
	Fullness and faithfulness is ensured by \cref{rem:markov_kernels} and \cref{thm:povm_integration}.
	The equivalence claim now follows by the essential surjectivity of $\sobstoch \hookrightarrow \stoch$ and~\eqref{eq:measurableGelfand}.
\end{proof}

By restricting $\Linf \colon \stoch \to \cSigmaCcpu$ to Baire measurable spaces, we immediately obtain the following probabilistic versions of \cref{cor:pb_duality,cor:borelmeas}.

\begin{cor}\label{cor:bairestoch}
	There are equivalences 
	\[
		\begin{tikzcd}
			\bairestoch \ar[rr,bend left=15,"\Linf" pos=0.55] \ar[rr,phantom,"\cong^{\op}" pos=0.55] && \cpbenvcpu \ar[ll,bend left=15,"\specs" pos=0.4] &&
				\borelstoch \ar[rr,bend left=15,"\Linf"] \ar[rr,phantom,"\cong^{\op}" pos=0.55] && \cpbsepcpu \ar[ll,bend left=15,"\specs" pos=0.4]
		\end{tikzcd}
	\]
\end{cor}

\begin{rem}\label{rem:no_prob_adj}
	In~\eqref{eq:measurableGelfand_stoch}, we have considered the Gelfand $\sigma$-spectrum as a functor 
	\[ 
	\specs \colon \cmsigmaCcpu \to \stoch.
	\]
	One may also wonder whether measurable Gelfand duality (\cref{thm:measurableGelfand}) extends to a probabilistic version involving $\cSigmaCcpu$ and $\stoch$.
	Unfortunately, the functor $\specs \colon \cSigmaC \to \meas$ cannot be extended to a functor
	\[
		\cSigmaCcpu \to \stoch
	\]
	while keeping the units $\A \to \Linf(\specs(\A))$ natural. 
	For example, take $L^{\infty}([0,1])$, as before with respect to the Lebesgue measure, which is a $\sigma$-state $\lambda \colon L^{\infty} ([0,1]) \to \mathbb{C}$. 
	Since $\specs (L^{\infty}([0,1])) = \emptyset$, the naturality diagram with respect to $\lambda$ would take the form
	\[
	\begin{tikzcd}
		L^{\infty}([0,1]) \ar[d,"\lambda"'] \ar[r, two heads]& \Linf(\emptyset)=\lbrace 0 \rbrace \ar[d]\\
		\mathbb{C} \ar[r,"\id"] & \mathbb{C}
	\end{tikzcd}
	\] 
	where the two horizontal arrows are the units.
	But clearly such a diagram cannot commute.
\end{rem}

\begin{rem}
	Probabilistic Gelfand duality~\cite{furber_jacobs_gelfand} is an equivalence $\cHausstoch\cong^\op \cCalgcpu$,
	where the former is the category of compact Hausdorff spaces with continuous Markov kernels.\footnote{A Markov kernel $\mu$ is \emph{continuous} (with respect to the $A$-topology) if $x \mapsto \mu(U|x)$ is lower semi-continuous for every open set $U$~\cite[Section 4]{fritz2019support}.}
	In the forward direction, it is given by $C(-) \colon \cHausstoch \to \cCalgcpu$, which is defined just like~\eqref{eq:Koopman}: composing with a continuous Markov kernel $\mu \colon X \rightsquigarrow Y$ like that turns out to map continuous functions to continuous functions as $C(\mu) \colon C(Y) \to C(X)$.
	The other direction is given by the Gelfand spectrum functor $\spec \colon \cCalgcpu \to \cHausstoch$, defined in terms of the Riesz--Markov--Kakutani representation theorem.

	With this in mind, we can now extend the diagram~\eqref{eq:measurableGelfand_chaus} to the probabilistic setting:
	\[
		\begin{tikzcd}[row sep=large]
			\cHausstoch \ar[rr] \ar[d,"C(-)"] && \stoch \ar[d,"\Linf"] \\
			\cCalgcpu \ar[rr,"(-)^{\infty}"] \ar[d,"\spec"] && \cmsigmaCcpu \ar[d,"\specs"] \\
			\cHausstoch \ar[rr] && \stoch
		\end{tikzcd}
	\]
	Here, as before, the topological side appears on the left and the measurable side on the right. 
	The proof proceeds exactly as in the case of~\eqref{eq:measurableGelfand_chaus}.
\end{rem}

\section{Tensor products of \texorpdfstring{\cstars{}s}{σC*-algebras}}\label{sec:tensor_cstars}

Using the duality theorems discussed above, we now briefly investigate candidate definitions of tensor products for \cstars{}s. 
As in the case of \bsigma{}s studied in \cref{sec:tensor_bool}, the general and the concrete case have different flavors, since they are naturally aligned with the universal and the regular $\sigma$-completions, respectively.

\subsection{Tensor products of commutative \texorpdfstring{\cstars{}s}{σC*-algebras}}\label{sec:tensor_commutative_cstars}

Let us focus on the commutative case for the moment, where we can leverage the equivalence $\cSigmaC \cong \SigmaB$ from \cref{thm:general_duality}.
As we will see after, the noncommutative case poses additional difficulties that we have not been able to resolve yet.

\begin{defn}\label{def:cstars_unitensor}
	Let $\A$ and $\B$ be two commutative \cstars{}s and $\A \otimes \B$ their C*-algebraic tensor product.\footnote{Since commutative \cstar{}s are nuclear, there is only one C*-algebraic tensor product, namely their coproduct in $\cCalg$.} Then:
	\begin{enumerate}
		\item The \newterm{universal tensor product} $\A\unitensor \B$ is the universal $\sigma$-completion of $\A \otimes \B$.
		\item The \newterm{regular tensor product} $\A \regtensor \B$ is the regular $\sigma$-completion of $\A \otimes \B$.
	\end{enumerate}
\end{defn}

To relate these tensor products to those for \bsigma{}s, we first note that the diagram
\[
	\begin{tikzcd}[row sep=large]
		\SigmaB \times \SigmaB \ar[rr,"\otimes"] \ar[d,hookrightarrow,"C(\stone(-)) \times C(\stone(-))"'] && \Bool \ar[d,hookrightarrow,"C(\stone(-))"] \\
		\cSigmaC \times \cSigmaC \ar[rr,"\otimes"] && \cCalg
	\end{tikzcd}
\]
commutes by Gelfand and Stone duality (up to canonical natural isomorphism), since the product of two Stone spaces coincides with their product as compact Hausdorff spaces.
By \cref{prop:universal_completions,cor:regular_completions}, we can conclude that it likewise commutes with $\unitensor$ or $\regtensor$ in place of $\otimes$: both tensor products match those for \bsigma{}s, and in this sense they are not new.

\begin{prop}\label{prop:tensor_coproduct}
	\begin{enumerate}
		\item If $\A, \B \in \cSigmaC$, then $\A \unitensor \B$ is the coproduct of $\A$ and $\B$ in $\cSigmaC$.
		\item If $\A, \B \in \cmsigmaC$, then $\A \regtensor \B$ is the coproduct of $\A$ and $\B$ in $\cmsigmaC$.
	\end{enumerate}
\end{prop}

\begin{proof}
	\begin{enumerate}
		\item By the fact that the Boolean $\unitensor$ is the coproduct in $\SigmaB$ (\cref{prop:sigma_tensor_product}).
		\item Using $\cmsigmaC \cong \csigmaB$, by the fact that the Boolean $\regtensor$ is the coproduct in $\csigmaB$ (\cref{cor:univprop_regular}).
			\qedhere
	\end{enumerate}
\end{proof}

Unfortunately, the universal tensor product of two commutative purely $\sigma$-representable \cstar{}s is not a purely $\sigma$-representable \cstar{} in general: \cref{ex:tensor_product} is the corresponding statement in the Boolean setting. 
However, the case of Pedersen--Baire envelopes is quite well-behaved.

\begin{prop}\label{prop:pb_tensor}
	For commutative \cstar{}s, there is an isomorphism
	\begin{equation} 
		\label{eq:tensor_pb}
		(\A \otimes \B)^{\infty} \cong \A^{\infty} \unitensor \B^{\infty} \cong \A^{\infty} \regtensor \B^{\infty}
	\end{equation}
	natural in $\A$ and $\B$.
	In particular, $\A^{\infty} \unitensor \B^{\infty}$ is a purely $\sigma$-representable \cstar{}.
\end{prop}
A slightly more general result was established by Jamneshan and Tao \cite[Proposition 6.7.(iv)]{jamneshan23foundational} in terms of \bsigma{}s. 
We include the following proof to highlight how our approach yields a concise and transparent argument.
\begin{proof}
	Like every left adjoint, the functor $(-)^\infty \colon \cCalg \to \cSigmaC$ preserves coproducts, which gives the first isomorphism.
	Therefore $\A^{\infty} \unitensor \B^{\infty}$ is a purely $\sigma$-representable \cstar{} in particular.
	The second isomorphism amounts to the claim that the natural projection $\A \unitensor \B \to \A \regtensor \B$ is an isomorphism, which we can affirm now since each tensor product is the coproduct in $\cmsigmaC$.
\end{proof}

\begin{rem}
	\label{rem:tensor_baire}
	In terms of measurable Gelfand duality (\cref{thm:measurableGelfand}), \eqref{eq:tensor_pb} also ensures that the product of two Baire measurable spaces is again a Baire measurable space, and this is induced by the product topology.
	This is true even for infinite products~\cite[Lemma~2.1]{jamneshan23moore}.
\end{rem}

\subsection{Symmetric monoidal equivalences}\label{sec:symmetric_monoidal_equivalences}

We now briefly investigate to which extent the equivalences considered in previous sections are not just equivalences of categories, but equivalences of \emph{symmetric monoidal} categories, where we consider the monoidal structures considered in the previous subsection.
In many cases, this holds because the products and tensor products that we consider have the same universal properties.
This applies in particular to all of the equivalences~\eqref{eq:measurableGelfand},~\eqref{eq:pb_duality} and~\eqref{eq:borelmeas}.

However, at the level of probabilistic morphisms, our products and tensor products no longer enjoy the same universal properties, and we have to be more careful.

\begin{prop}\label{prop:linf_sym}
	The equivalence~\eqref{eq:measurableGelfand_stoch} is a symmetric monoidal equivalence
	\[
		\begin{tikzcd}
			(\stoch,\times) \ar[rr,bend left=15,"\Linf"] \ar[rr,phantom,"\cong^{\op}" pos=0.6] && (\cmsigmaCcpu, \regtensor) \ar[ll,bend left=15,"\specs" pos=0.4]
		\end{tikzcd}
	\]
\end{prop}

On the left, we write $\times$ for the product of measurable spaces, although it is \emph{not} the categorical product in $\stoch$ (but merely a symmetric monoidal structure).
The analogous statement and proof hold for the restricted versions considered in \cref{cor:bairestoch} as well.

\begin{proof}
	We emphasize that, a priori, we do not know whether $\regtensor$ is even functorial with respect to \scpu{} maps; we only know from \cref{prop:tensor_coproduct} that the subcategory $\cmsigmaC \subseteq \cmsigmaCcpu$ is symmetric monoidal.

	First, the inclusion $\sobstoch \hookrightarrow \stoch$, which we noted to be an equivalence in \cref{rem:sobstoch}, is symmetric monoidal because the product of sober measurable spaces is sober (\cref{rem:sober_tensor_product}). 
	Since $\specs$ lands in $\sobstoch$ anyway, we can consider $\sobstoch$ instead of $\stoch$ for the remainder of the proof.
	Then what we already know from~\eqref{eq:measurableGelfand} is the symmetric monoidal equivalence
	\[
		\begin{tikzcd}
			(\sobmeas,\times) \ar[rr,bend left=15,"\Linf"] \ar[rr,phantom,"\cong^{\op}"] && (\cmsigmaC,\regtensor) \ar[ll,bend left=15,"\specs"]
		\end{tikzcd}
	\]
	But now the fact that these two functors also implement an equivalence $\stoch \cong^\op \cmsigmaCcpu$ makes $\regtensor$ functorial with respect to \scpu{} maps by simply transporting the structure, and this results in a symmetric monoidal equivalence by construction.
\end{proof}

\begin{rem}\label{rem:markov_functor}
	Readers familiar with Markov categories may wonder whether $\Linf$ is also an equivalence of Markov categories. 
	A Markov category~\cite{fritz2019synthetic} is a symmetric monoidal category where every object comes equipped with a designated comonoid structure, subject to certain compatibility conditions.
	$\stoch$ is one of the most important examples: there, the comultiplication is given by the diagonal map 
	\[
		\Delta \colon X \to X \times X
	\]
	while the counit is the Markov kernel $X\rightsquigarrow \lbrace *\rbrace$ that associates to each $x\in X$ the only probability measure on the singleton space $\lbrace *\rbrace$. 
	
	In $(\cmsigmaCcpu)^{\op}$, and more in general in any category opposite to a category of algebras, a canonical choice for the comonoid structure is given by the multiplication and unit of the algebras (cf. \cite[\S~9]{fritz2019synthetic}).
	Based on this, $\Linf$ becomes an equivalence of Markov categories, since
	\[
		\Linf(\Delta) \colon \Linf(X) \regtensor \Linf(X) \cong \Linf(X \times X) \longrightarrow \Linf(X)
	\]
	is precisely the multiplication of $\Linf(X)$.
\end{rem}

\subsection{Towards tensor products for noncommutative \texorpdfstring{\cstars{}s}{σC*-algebras}}\label{sec:tensor_pbenv}

Let us finally turn to tensor products in the noncommutative setting. 
There, the main problem is functoriality.
To illustrate this, suppose that we define $\unitensor$ as before, where the C*-algebraic tensor product $\otimes$ is any functorial tensor product of \cstar{}s (such as the minimal or the maximal one).
Then is $\unitensor$ really a functor in each argument?
That is, if we have \starshom{}s
\[
	\phi \colon \A \to \C, \qquad \psi \colon \B \to \D,
\]
then do we get an induced \starshom{} $\phi \unitensor \psi \colon \A \unitensor \B \to \C \unitensor \D$?
The universal property of the universal $\sigma$-completion does not immediately guarantee this, because it is unclear whether the $*$-homomorphism $\phi \otimes \psi$ is $\sigma$-normal.
Indeed, the difficulties already observed in the Boolean case (\cref{sec:tensor_bool}) still arise, while the solution provided by Sikorski (\cref{prop:sigma_tensor_product}) does not directly extend to the noncommutative setting. 
One could try to approach functoriality by developing a noncommutative generalization of \cref{thm:sigma_normal_proj}, showing that $\sigma$-normality follows from $\sigma$-normality on projections.
But even then important challenges remain, as one still needs to understand the structure of projections in the tensor products.
The situation is even more complicated when dealing with the regular $\sigma$-completion instead, where we have no universal property at our disposal (recall \cref{rem:nonfunctoriality_regular}). 
A possible workaround might be to restrict to a suitable subcategory like $\msigmaC$, where one could hope for a universal property.

The situation is much better on the category of Pedersen--Baire envelopes $\PBenv$.
There, we indeed have a way to extend \cref{prop:pb_tensor} to the noncommutative case, and this can be used as a starting point for defining a tensor product.

\begin{defn}
	For \cstar{}s $\A$ and $\B$, their \newterm{maximal Pedersen--Baire tensor product} is
	\[
		\A^{\infty} \otimes^{\infty}_{\max} \B^{\infty} \coloneqq (\A \otimes_{\max} \B)^{\infty}.
	\]
\end{defn}
	
Here it is important that we employ the maximal tensor product $\otimes_{\max}$, as the following arguments would not work e.g.~with the minimal tensor product $\otimes_{\min}$.

\begin{prop}
	\label{prop:nc_pb_tensor}
	$\otimes^{\infty}_{\max}$ makes $\PBenv$ into a symmetric monoidal category.
\end{prop}

\begin{proof}
	We actually show that $\otimes^{\infty}_{\max}$ has the same universal property in $\PBenv$ as $\otimes_{\max}$ does in $\cCalg$.
	This will straightforwardly imply all other properties, including functoriality, the construction of the associator, unitor and braiding isomorphisms, and the commutativity of the coherence diagrams, in just the same way as one proves that $\otimes_{\max}$ makes $\cCalg$ into a symmetric monoidal category.

	So suppose that we have two \starshom{}s
	\[
		\phi \colon \A^{\infty} \to \C^{\infty}, \qquad \psi \colon \B^{\infty} \to \C^{\infty}
	\]
	with commuting ranges.
	Then we obtain an induced (restricted) $*$-homomorphism
	\[
		\A \otimes_{\max} \B \longrightarrow \C^{\infty}.
	\]
	The universal property of the Pedersen--Baire envelope now lets us extend this to a \starshom{} of the desired form
	\[
		\A^{\infty} \otimes^{\infty}_{\max} \B^{\infty} \longrightarrow \C^{\infty}.
	\]
	The fact that this \starshom{} recovers $\phi$ on $\A^{\infty}$ follows by noting that this holds on $\A$ which $\sigma$-generates $\A^{\infty}$, and similarly for $\psi$.
	The uniqueness is clear as well by combining both universal properties.
\end{proof}

\begin{rem}
	The reasoning of \cref{prop:nc_pb_tensor} holds verbatim with \scpu{} maps as morphisms.\footnote{Here we use the universal property of Pedersen--Baire envelopes in the form of \cref{cor:pb_univ2}; we also recall that the universal property of the maximal tensor product holds even for \cpu{} maps~\cite[Exercise 3.5.1]{brownozawa}.} 
	Therefore, $\otimes^{\infty}$ also makes $\pbenvcpu$ into a symmetric monoidal category.
	In fact, its monoidal unit is initial, which makes its opposite $(\pbenvcpu)^{\op}$ a \emph{semicartesian} category.
	This is relevant insofar as semicartesian categories are one abstract framework for quantum probability~\cite{houghtonlarsen2021dilations,fritz2022dilations}, and we hope that $\pbenvcpu$ could be a good candidate for a convenient category for quantum probability.
\end{rem}

\bibliographystyle{alpha}
{\footnotesize 
\bibliography{references}}






\end{document}